%% file: main.tex
\documentclass[reqno, oneside]{amsart}

\input{preamble.tex}

\date{\today}

\title{Model Theory of Local Real Closed SV-Rings of Finite Rank}
\author{Ricardo Palomino Piepenborn}
\address{The University of Manchester, Department of Mathematics, Oxford Road, Manchester,  M13 9PL, United Kingdom}
\email{\href{mailto:ricardo.palomino@rjpp.net}{ricardo.palomino@rjpp.net}}
\urladdr{\href{http://www.rjpp.net}{rjpp.net}}

\begin{document}

\begin{abstract}
	This note begins the model-theoretic study of local real closed SV-rings of finite rank; to this end, a structure theorem for reduced local SV-rings of finite rank is given and branching ideals in local real closed rings of finite rank are analysed. The class of local real closed SV-rings of rank $ n \in \N^{\geq 2} $  is elementary in the language of rings $ \mathscr{L} := \{ +, -, \cdot, 0, 1 \} $ and its $ \mathscr{L} $-theory $ T_n $ has a model companion $ T_{n,1} $; models of  $ T_{n,1} $ are $ n $-fold fibre products $ (((V_1 \times_{\emph{\textbf{k}}} V_2) \times_{\emph{\textbf{k}}}V_3) \dots \times_{\emph{\textbf{k}}} V_{n-1})\times_{\emph{\textbf{k}}}  V_{n} $ of non-trivial real closed valuation rings $ V_i $ with isomorphic  residue field  $ \emph{\textbf{k}} $. The $ \mathscr{L} $-theory $ T_{n,1} $ is complete, decidable, and NIP. After enriching $ \mathscr{L} $ with a predicate for the maximal ideal,  models of  $ T_n $  have prime extensions in models of  $ T_{n,1} $, and $ T_{n,1} $ is the model completion of $ T_n $ in this enriched language. A quantifier elimination result for $ T_{n,1} $ is also given. The class of those local real closed SV-rings of rank $ n\in \N^{\geq 2} $ which are $ n $-fold fibre products $ (((V_1 \times_{W} V_2) \times_{W}V_3) \dots \times_{W} V_{n-1})\times_{W}  V_{n} $ of non-trivial real closed valuation rings $ V_i $ along surjective morphisms $ V_i \lonto W $  onto a non-trivial domain $ W $  is elementary in the language of rings, and its $ \mathscr{L} $-theory $ T_{n,2} $  is also complete, decidable, and NIP; after enriching $ \mathscr{L} $ with predicates for the maximal ideal and the unique branching ideal, $ T_{n,2} $ is model complete. 
\end{abstract}

\maketitle

\vspace{-1.21cm}

\setcounter{tocdepth}{2}  
\tableofcontents

\vspace{-1.21cm}
\section{Introduction}\label{SEC.intro}
\input{introduction.tex}

\section{SV-Rings}\label{SEC.SV}
\input{SV-rings.tex}

\section{Real Closed Rings}\label{SEC.rcSV}
\input{real_closed_SV-rings.tex}

\section{Model Theory}\label{SEC.mod_th}
\input{model_theory.tex}

\section{Towards a Uniform Approach to the Model Theory of Local Real Closed SV-Rings of Finite Rank}\label{SEC.concl_rem}
\input{concluding_remarks.tex}

\appendix
\section{Embedding Real Closed Valued Fields into Real Closed Hahn Series Fields}\label{SEC.app}
\input{appendix.tex}

\printbibliography

\end{document}

%% file: preamble.tex
\usepackage[utf8]{inputenc}    
\usepackage[shortlabels,inline]{enumitem} 
\usepackage{imakeidx} 
\usepackage[backend=biber,style=alphabetic]{biblatex}
\usepackage{subfiles}
\usepackage{comment}
\usepackage{microtype}
\usepackage[english]{babel}  
\usepackage{amsthm, xpatch}
\usepackage{amsmath,amssymb}
\usepackage{amsfonts}             
\usepackage{dsfont}              
\usepackage{stmaryrd}           
\usepackage{mathrsfs}
\usepackage{manfnt}
\usepackage{mathabx} 
\usepackage{upgreek}
\usepackage{apptools}

\usepackage[dvipsnames]{xcolor}
\usepackage[a4paper, left=1.5in, right=1.5in, top=1.5in, bottom=1.5in]{geometry} 
\usepackage{graphicx}           
\usepackage{mathtools}          
\usepackage{enumitem}          
\usepackage{listings}
\usepackage{todonotes} 
\usepackage{tikz}\usetikzlibrary{matrix,arrows,calc,decorations.pathmorphing} 
\usepackage{tikz-cd}            
\usepackage[autostyle]{csquotes}
\usepackage{hyperref}          
\usepackage{adjustbox}
\allowdisplaybreaks

\newtheorem{FactCounter}{dummy}[subsection] 
\newtheorem{theorem}[FactCounter]{Theorem}
\newtheorem{corollary}[FactCounter]{Corollary}
\newtheorem{lemma}[FactCounter]{Lemma}
\newtheorem{proposition}[FactCounter]{Proposition}

\theoremstyle{definition}
\newtheorem{definition}[FactCounter]{Definition}

\theoremstyle{definition}

\theoremstyle{remark}
\newtheorem{remark}[FactCounter]{Remark}

\theoremstyle{remark}
\newtheorem{remarks}[FactCounter]{Remarks}

\theoremstyle{definition}
\newtheorem{example}[FactCounter]{Example}

\theoremstyle{definition}
\newtheorem{convention}[FactCounter]{Convention}

\theoremstyle{definition}

\theoremstyle{definition}
\newtheorem{notation}[FactCounter]{Notation}

\theoremstyle{plain}
\newtheorem{conjecture}[FactCounter]{Conjecture}

\newtheorem{FactCounterApp}{dummy}[section] 

\theoremstyle{plain}
\newtheorem{theoremApp}[FactCounterApp]{Theorem}

\theoremstyle{plain}

\theoremstyle{plain}
\newtheorem{lemmaApp}[FactCounterApp]{Lemma}

\theoremstyle{plain}

\newtheorem{FactCounterConc}{dummy}[section] 

\theoremstyle{plain}
\newtheorem{theoremConc}[FactCounterConc]{Theorem}

\theoremstyle{plain}

\theoremstyle{plain}

\theoremstyle{plain}

\theoremstyle{plain}
\newtheorem{conjectureConc}[FactCounterConc]{Conjecture}

\theoremstyle{remark}

\theoremstyle{definition}

\theoremstyle{definition}

\newcommand\N{\mathds{N}}
\newcommand\Z{\mathds{Z}}
\newcommand\Q{\mathds{Q}}
\newcommand\R{\mathds{R}}

\newcommand{\into}{\hookrightarrow}
\newcommand{\onto}{\twoheadrightarrow}

\newcommand{\lra}{\longrightarrow}%
\newcommand{\ra}{\rightarrow}%
\newcommand{\Ra}{\Rightarrow}%
\DeclareRobustCommand\lonto
     {\relbar\joinrel\twoheadrightarrow}
\DeclareRobustCommand\linto
     {\lhook\joinrel\longrightarrow}

\DeclareMathOperator*{\bdwedge}{\bigwedge\mkern-15mu\bigwedge}

\newcommand{\subsetsim}{\mathrel{\ooalign{\raise.4ex\hbox{$\subset$}\cr$\raise-.9ex\hbox{$\sim$}$}}}
\newcommand{\subsetnsim}{\mathrel{\ooalign{\raise.4ex\hbox{$\subset$}\cr$\raise-.9ex\hbox{$\not\sim$}$}}}
\newcommand{\supsetsim}{\mathrel{\ooalign{\raise.4ex\hbox{$\supset$}\cr$\raise-.9ex\hbox{$\sim$}$}}}
\newcommand{\supsetnsim}{\mathrel{\ooalign{\raise.4ex\hbox{$\subset$}\cr$\raise-.9ex\hbox{$\not\sim$}$}}}

\let\emptyset\varnothing

\let\phi\varphi

\let\epsilon\varepsilon

\calclayout

\hypersetup{colorlinks=true}

\colorlet{red}{BrickRed}
\colorlet{green}{OliveGreen}
\colorlet{magenta}{Thistle}

\makeatletter
\tikzset{
    dot diameter/.store in=\dot@diameter,
    dot diameter=3pt,
    dot spacing/.store in=\dot@spacing,
    dot spacing=10pt,
    dots/.style={
        line width=\dot@diameter,
        line cap=round,
        dash pattern=on 0pt off \dot@spacing
    }
}
\makeatother

\makeatletter
\def\@tocline#1#2#3#4#5#6#7{\relax
	\ifnum #1>\c@tocdepth % then omit
	\else
	\par \addpenalty\@secpenalty\addvspace{#2}%
	\begingroup \hyphenpenalty\@M
	\@ifempty{#4}{%
		\@tempdima\csname r@tocindent\number#1\endcsname\relax
	}{%
		\@tempdima#4\relax
	}%
	\parindent\z@ \leftskip#3\relax \advance\leftskip\@tempdima\relax
	\rightskip\@pnumwidth plus4em \parfillskip-\@pnumwidth
	#5\leavevmode\hskip-\@tempdima
	\ifcase #1
	\or\or \hskip 1em \or \hskip 2em \else \hskip 3em \fi%
	#6\nobreak\relax
	\dotfill\hbox to\@pnumwidth{\@tocpagenum{#7}}\par
	\nobreak
	\endgroup
	\fi}
\makeatother

%% file: introduction.tex
  The present note studies $ n $-fold fibre products of non-trivial real closed valuation rings along surjective ring homomorphisms onto a fixed domain, where $ n \in \N^{\geq 2} $. This \enquote{bottom-up} definition of this class of rings has an equivalent \enquote{top-down} description, namely, these rings  are exactly \textit{local real closed SV-rings of finite rank with one branching ideal}. In what follows it will be explained what each of the words \enquote{real closed}, \enquote{SV}, \enquote{finite rank}, and \enquote{branching ideal} mean, giving along the way some motivation for the model-theoretic analysis of local real closed SV-rings of finite rank.

\textit{Survaluation rings} (\textit{SV-rings} for short)  were first introduced in \cite{henriksen/wilson.when_is_CX_a_val_ring} in connection to rings $ C(X) $ of continuous real-valued  functions on a completely regular topological space $ X $; in the aforementioned paper, the authors call the ring $ C(X) $ an SV-ring if $ C(X)/\mathfrak{p} $ is a valuation ring for every prime ideal $ \mathfrak{p}$ of $ C(X) $, and $ X $ an SV-space if $ C(X) $ is an SV-ring. A canonical partial order on the rings $ C(X) $ is defined by setting $ f \leq g $ if and only if  $ f(x) \leq g (x) $ for all $ x \in X $; this partial order is a lattice-order which gives  $ C(X) $ the structure of an $ f $-ring (\cite{gillman/jerison.rings_cts_func}, \cite{bkw.groupes}), and this motivated the study of SV-rings within the class of $ f $-rings in \cite{henriksen/larson.semiprime_f-rings_that} (see also \cite{henriksen/larson/martinez/woods.lattice-ord_alg}, \cite{larson.finitely_1-conv_f-rings}, as well as the survey \cite{larson.SV_and_related_f-rings_and_spaces}). SV-rings can also be studied without the presence of a partial order: Schwartz defines in \cite{schwartz.SV} a commutative and unital ring $ A $   to be an SV-ring if $ A/\mathfrak{p} $ is a valuation ring for all prime ideals $\mathfrak{p}  $ of $ A $, and this is what is meant here by an SV-ring (Definition \ref{def.SV-ring}). 

The article  \cite{schwartz.SV} contains a systematic study of SV-rings and it is the main reference on SV-rings for the present work.     \cite{schwartz.SV} also opened up the door for the model-theoretic study of SV-rings by  proving the first results on axiomatizability (in the sense of model theory) of SV-rings in the language of rings $ \mathscr{L}:= \{ +, - , \cdot, 0 ,1 \}  $. In particular, it is shown in \cite[Section 3]{schwartz.SV} that the question of whether a class of SV-rings is elementary or not is tightly connected with  the \textit{rank} of the rings in this class; the rank of a prime ideal $ \mathfrak{p} $ in a ring $ A $ is defined as the number (which is either a natural number or  $ \infty $) of minimal prime ideals  $ \mathfrak{q}  $ of $ A $ such that  $ \mathfrak{q} \subseteq \mathfrak{p} $, and the rank of the ring $ A $ is the supremum of the ranks of its prime ideals (Definition \ref{def.rank}), therefore a local ring has finite rank if and only if it has finitely many minimal prime ideals.

The   rings $ C(X) $ are particular examples of \textit{real closed rings} in the sense of  Schwartz (\cite{schwartz.basic}, \cite{schwartz.rcr}, \cite{schwartz.rings_of_cts_functions_as_rcrs}, \cite[Section 12]{schwartz/madden.safr}, \cite{tressl.super}). The terminology \enquote{real closed ring}  was first coined by Cherlin and Dickmann in \cite{cherlin/dickmann.rcrI} and \cite{cherlin/dickmann.rcrII}, and some results in the literature about real closed rings refer to real closed rings in the sense of Cherlin and Dickmann (e.g. \cite{weak}); in this note a real closed ring is always meant to be a real closed ring in the sense of Schwartz (Definition \ref{D.RCSVR.rcr}). Real closed rings in the sense of Cherlin and Dickmann are exactly real closed rings which are also valuation rings, i.e., they are \textit{real closed valuation rings} (\cite{schwartz.rcvr}); equivalently, these are local real closed SV-rings of rank $ 1 $ (Theorem \ref{T.RCSVR.equiv_rcvr}).

Non-trivial real closed valuation rings (i.e., those which are not fields) are exactly proper convex subrings of real closed fields, and this close relationship between these two classes of rings entails that   non-trivial real closed valuation rings have many of the good model-theoretic properties of real closed fields (\cite{cherlin/dickmann.rcrII}, \cite{becker}); in particular, the class of non-trivial real closed valuation rings is elementary in the language of rings $ \mathscr{L} $, and its theory is complete, decidable, and NIP (\cite{cherlin/dickmann.rcrII}, \cite{weak}). It follows that the class of local real closed SV-rings of rank $ 1 $ splits into the classes of models of two complete, decidable, and NIP $ \mathscr{L} $-theories, namely, the $ \mathscr{L} $-theory $ \sf{RCF} $ of real closed fields and the $ \mathscr{L} $-theory $ \sf{RCVR} $ of non-trivial real closed valuation rings.

Local real closed SV-rings of rank $ n \in \N^{\geq 2} $ are exactly those rings obtained by taking  iterated fibre products of finitely many non-trivial real closed valuation rings along surjective ring homomorphisms onto domains, see Theorem \ref{T.RCSVR.equiv_loc_real_closed_SV-ring_finite_rank} for a precise formulation of this statement. Moreover, the class of local real closed SV-rings of rank $ n \in \N^{\geq 2} $ is elementary in the language $ \mathscr{L} $; this follows from \cite[Proposition 2.2 and Corollary 3.16]{schwartz.SV}, but an equivalent axiomatization $ T_n $ for this class of rings is given in Definition \ref{D.MOD_TH.T_n}. 

A very particular class of local real closed SV-rings of rank $ n \in \N^{\geq 2} $ is the one whose rings have  exactly one \textit{branching ideal} (Definition \ref{D.RCSVR.branch_id} and Lemma \ref{L.RCSVR.equiv_loc_rcsvr_rk_n_exactly_one_branching_ideal}): a prime ideal $ \mathfrak{q} $ in a  ring $ A $ is defined to be a branching ideal if there exist distinct prime ideals $ \mathfrak{p}_1, \mathfrak{p}_2  \subseteq A$ such that $ \mathfrak{p}_1, \mathfrak{p}_2 \subsetneq \mathfrak{q} $ and $ \mathfrak{q}= \mathfrak{p}_1 + \mathfrak{p}_2 $. Local real closed rings of rank $ n \in \N^{\geq 2} $ have at least one branching ideal and at most $ n-1 $ branching ideals (Remarks \ref{R.RCSVR.branch_0.ii} and \ref{R.RCSVR.branch_2}), so those with exactly one branching ideal are the simplest rings in this class; moreover, there exists an $ \mathscr{L} $-sentence $ \phi_{\text{br}, n} $ (Definition \ref{D.MODTH.phi_br,n}) such that for all local real closed rings $ A $ of rank $ n  $, $ A \models \phi_{\text{br}, n}  $ if and only if $ A  $ has exactly one branching ideal (Lemma \ref{L.MODTH.phi_br,n}). 

If $ A $ is a local real closed ring of rank $ n $ with unique maximal ideal $ \mathfrak{m}_A $ and with a  unique branching ideal  $ \mathfrak{b}_A $, then either $ \mathfrak{b}_A = \mathfrak{m}_A$ or $ \mathfrak{b}_A \subsetneq \mathfrak{m}_A $, and this is an elementary property of the ring $ A $ (Proposition \ref{P.RCSVR.equiv_branching_max_id}). In particular, the elementary class of local real closed SV-rings of rank  $ n $ with exactly one branching ideal splits into two elementary classes of rings, namely, local real closed SV-rings of rank  $ n $ with exactly one branching ideal $ \mathfrak{b}_A $ such that $ \mathfrak{b}_A= \mathfrak{m}_A $ (called for brevity rings of \textit{type $ (n,1) $}, see Definition \ref{D.RCSVR.type}), and local real closed SV-rings of rank  $ n $ with exactly one branching ideal $ \mathfrak{b}_A $ such that $ \mathfrak{b}_A\subsetneq \mathfrak{m}_A $ (called for brevity rings of \textit{type $ (n,2) $});  geometric examples of rings of type $ (n,1) $ are rings of germs of continuous semi-algebraic functions $ X \lra R $ at a point $ x \in X $, where $ X \subseteq R^m $  is a semi-algebraic curve over a real closed field  $ R $ (Example \ref{ex.germs}).  

The next theorems summarize the main contributions in Section \ref{SEC.mod_th}:

\begin{theoremConc}\label{thm_1}
	The $ \mathscr{L} $-theory $ T_n \cup \{ \phi_{\emph{br},n} \}  $ of local real closed SV-rings of rank $ n \in \N^{\geq 2} $ with one branching ideal splits into two  complete, decidable, and NIP $ \mathscr{L} $-theories, namely, the theories $ T_{n,1} $ and $ T_{n,2} $ of rings of type $ (n,1) $ and of type $ (n,2) $, respectively.
\end{theoremConc}

\begin{theoremConc}\label{thm_2}
	Let $ n \in \N^{\geq 2} $.
	\begin{enumerate}[\normalfont(I)]
	\item 	The $ \mathscr{L} $-theory $ T_{n,1}  $ is model complete, and it is the model companion of $ T_n $ and also of the $ \mathscr{L} $-theory of local real closed rings of rank $ n $. 
		\begin{enumerate}[\normalfont(i)]
		\item 	Let $ \emph{\texttt{m}} $ be a unary predicate and let $ T_{n,1}(\emph{\texttt{m}}) $ and $ T_{n}(\emph{\texttt{m}}) $ be the $ \mathscr{L}(\emph{\texttt{m}}) $-theories $ T_{n,1} $ and $ T_{n} $ together with the sentence expressing that $ \emph{\texttt{m}} $ is the set of non-units, respectively. Models of $ T_{n}(\emph{\texttt{m}}) $  have prime extensions in the class of models of $ T_{n,1}(\emph{\texttt{m}}) $, and $ T_{n,1}(\emph{\texttt{m}}) $ is the model completion of $ T_{n}(\emph{\texttt{m}}) $.
		\item $ T_{n,1} $ has quantifier elimination in the language  of lattice-ordered rings $ \mathscr{L}(\vee, \wedge, \leq)$ together with the divisibility predicate $ \emph{div} $ and $ n  $ constant symbols $ e_1, \dots, e_n $ interpreted as non-zero pairwise orthogonal elements. 
		\end{enumerate}
	\item Let $ \texttt{\emph{b}} $ and $ \texttt{\emph{m}} $ be unary predicates and set $ \mathscr{L}^* := \mathscr{L}(\texttt{\emph{b}}, \texttt{\emph{m}}) $. Let $T_{n,2}^*$ be the $ \mathscr{L}^* $-theory $ T_{n,2} $ together with the sentence expressing that $ \texttt{\emph{b}} $ is the branching ideal and $ \texttt{\emph{m}} $  is the set of non-units; $ T_{n,2}^* $ is model complete.
	\end{enumerate}
\end{theoremConc}

Much of the work towards proving Theorems \ref{thm_1} and \ref{thm_2} rests on having good algebraic descriptions of local real closed SV-rings of finite rank and of their branching ideals, and this is the content of Section \ref{SEC.rcSV}. In particular, Theorem \ref{T.RCSVR.equiv_loc_real_closed_SV-ring_finite_rank} is a structure theorem for local real closed SV-rings of finite rank which is deduced from a structure theorem for reduced local SV-rings rings of finite rank (Theorem \ref{T.SV.equiv_red_loc_SV-ring_finite_rank}), and Proposition \ref{P.RCSVR.equiv_branching_max_id} gives various equivalent conditions for the maximal ideal in a local real closed ring of finite rank to be a branching ideal;  Proposition \ref{P.RCSVR.equiv_branching_max_id}  then yields several equivalent characterizations of branching ideals in rings of this latter class (Remark \ref{R.RCSVR.branch_1}). 

Since all local real closed SV-rings of rank $ 2 $ have exactly one branching ideal, it follows from Theorem \ref{thm_1} that the $ \mathscr{L} $-theory of local real closed SV-rings of rank $ 2 $ splits into two complete, decidable, and  NIP  $ \mathscr{L} $-theories; the results above together with this observation build up to the following: 

\begin{conjectureConc}\label{conj_1}
The $ \mathscr{L} $-theory $ T_n $ of local real closed SV-rings of rank $ n \in \N^{\geq 2} $ splits into finitely many complete, decidable, and NIP $ \mathscr{L} $-theories.
\end{conjectureConc}

A proof of the above conjecture would rest in a good model-theoretic understanding of arbitrary local real closed SV-rings of finite rank. Subsection \ref{SUBSEC.diff} highlights some of the key obstacles that arise in the model-theoretic analysis of such rings when there are two or more branching ideals; based on the results in Section \ref{SEC.mod_th}, Conjecture \ref{conj_elem} gives candidates for what each of the finitely many completions of $ T_n $ mentioned in Conjecture \ref{conj_1} could be.

\subsection{Structure of the paper} Section \ref{SEC.SV} starts by collecting the relevant material on SV-rings and on ranks of rings; in particular, Lemma \ref{L.SV.ann_intersection_of_min_prime_id.iii} and Corollary \ref{C.SV.min_prime_ann} describe minimal prime ideals in reduced local rings of finite rank as annihilator ideals of elements, and this  is crucially used in Section \ref{SEC.mod_th}. The remaining part of Section \ref{SEC.SV} is devoted to prove a structure theorem for reduced local SV-rings of finite rank (Theorem \ref{T.SV.equiv_red_loc_SV-ring_finite_rank}), where the key algebraic result being used in the proof is Goursat's lemma for rings (Lemma \ref{L.SV.goursat}). 

Real closed rings are defined in Section \ref{SEC.rcSV} and their main algebraic properties are summarized in Theorem \ref{T.RCSVR.properties_rcr}; here a particular emphasis is made on real closed valuation rings corresponding to the canonical valuation on a real closed Hahn series field, as these rings  are used to construct certain embeddings (Proposition \ref{P.RCSVR.rcvr_embedd_nice}, Lemma \ref{L.RCSVR.1_square}, and Lemma \ref{L.RCSVR.2_square}) needed for the model completeness proofs in Section \ref{SEC.mod_th}. In Subsection \ref{SUBSEC.branch} branching ideals in rings are defined and studied in local real closed rings of finite rank, the main result here being Proposition \ref{P.RCSVR.equiv_branching_max_id}, which gives equivalent characterizations for the maximal ideal in these rings to be a branching ideal; both Subsection \ref{SUBSEC.prelim_rcr} and Subsection \ref{SUBSEC.branch} can be read independently of Section \ref{SEC.SV}. In the last part of Section \ref{SEC.rcSV}  the real closed version of Theorem \ref{T.SV.equiv_red_loc_SV-ring_finite_rank} is proved (Theorem \ref{T.RCSVR.equiv_loc_real_closed_SV-ring_finite_rank}), and this is then used to formally define the main objects of this note, namely rings of type $ (n, 1) $ and of type $ (n,2) $ (Lemma \ref{L.RCSVR.equiv_loc_rcsvr_rk_n_exactly_one_branching_ideal} and Definition \ref{D.RCSVR.type}); the section concludes by proving some embedding lemmas of such rings which are used in Section \ref{SEC.mod_th}. 

All the model theory of the note is contained in Sections \ref{SEC.mod_th} and \ref{SEC.concl_rem}. Section \ref{SEC.mod_th} starts by defining all the relevant theories (cf. Theorems \ref{thm_1} and \ref{thm_2}), and this is followed by the model completeness results (Theorems \ref{T.MODTH.T_1n_mod_compl} and \ref{T.MODTH.T_2n_mod_compl}) from which much of the remaining statements in Section \ref{SEC.mod_th} stem from. Section \ref{SEC.concl_rem}  starts by explaining some  difficulties in the model-theoretic study of arbitrary local real closed SV-rings of finite rank, and in Subsection \ref{SUBSEC.branching} an approach to overcome these difficulties is proposed by introducing the notion of the branching spectrum of a local real closed ring of finite rank and connecting it with the model theory of real closed rings with a radical relation as developed in \cite{prestel.schwartz/mod_th_rcr}; in particular, Corollary \ref{corll}  shows that elementary equivalent local real closed rings of finite rank have poset-isomorphic branching spectra, and this yields a candidate for an elementary classification of all local real closed SV-rings of finite rank (Conjecture \ref{conj_elem}).

The sole purpose of Appendix \ref{SEC.app} is proving  (the seemingly folklore) Theorem \ref{T.APP.rcvf_hahn_embedd_II}, which is needed for the aforementioned embedding lemmas of rings of type $ (n,1) $ and of type $ (n,2) $.

\subsection{Conventions, notation, and terminology}\label{S.conventions}
Fix the following conventions, notation and terminology for the rest of the note:

\begin{enumerate}[\normalfont(I), ref=\ref{S.conventions} (\Roman*)]
\item 	All rings  are commutative and unital.
\item\label{S.conventions.II} If $ A  $ is a ring, then:
	\begin{enumerate}[\normalfont(i), ref=\ref{S.conventions.II} (\roman*)]
		\item\label{S.conventions.II.i} 	$ \text{Spec}(A) $ is the \textit{Zariski spectrum of $ A $}, i.e., the set of prime ideals $ \mathfrak{p} $ of $ A $ topologized by taking  the sets $ D(a):= \{\mathfrak{p} \in \text{Spec}(A)\mid a \notin \mathfrak{p}\} $ $ (a \in A) $ as a basis of open sets (cf. \cite[Chapter 12]{dickmann/schwartz/tressl.specbook}); 
		\item $ \text{Spec}^{\text{min}}(A) \subseteq \text{Spec}(A)  $ is the set of minimal prime ideals of $ A $;
		\item $ \text{Nil}(A):= \{ a \in  A\mid a^{n}=0 \text{ for some } n \in \N \}  $ is the \textit{nilradical of $  A$} (recall also that $ \text{Nil}(A) = \bigcap_{\mathfrak{p} \in \text{Spec}(A)}\mathfrak{p} $, see \cite[Corollary 12.1.3]{dickmann/schwartz/tressl.specbook}); 
	\item $ A^{\times} $ is the group of multiplicative units of $ A $; 
	\item a \textit{residue ring of $ A $} is the ring $ A/I $ for some ideal $I \subseteq A$, and if $ I  $ is a prime ideal, then $ A/I $ is a \textit{residue domain}; and
	\item if $ a \in A $, then $ \text{Ann}_A(a) := \{ b \in A \mid ba=0 \} $ (set $ \text{Ann}(a):= \text{Ann}_A(a) $ if $ A  $ is clear from the context).
	\end{enumerate}
\item If $ A  $ is a domain, then $ \text{qf}(A) $ denotes its quotient field.
\item A valuation ring $ A $ (that is, a domain such that for all $  a \in \text{qf}(A)^{\times} $ either $ a \in A $ or  $ a^{-1} \in A $) is \textit{non-trivial} if $ A \neq \text{qf}(A) $.
\item If $ A, B,  $ and $ C $ are rings, and $ f: A \lra C$ and  $ g: B \lra C $ are ring homomorphisms, then the  \textit{fibre product of $ A $ and $ B $ over $ C $  (along $ f $ and $ g $)} is \[
		A \times_{C} B := \{ (a, b) \in A \times B  \mid f(a)= g(b)\}   
	;\] $ A \times_{C} B $ is the pullback of $ A  $ and $ B $ over $ C $  (along $ f $ and  $ g $) in the category of rings (cf. \cite[71]{maclane.categories}).
\item\label{S.conventions.subdirect} Let $ \{ A_i \}_{i \in I}  $  be a non-empty set of rings.
	\begin{enumerate}[\normalfont(i)]
	\item 	For each $ j  \in I $, let $ \pi_j: \prod_{i \in I} A_i \lonto A_j$ be the projection map. 
	\item A ring $ A  $ is a \textit{subdirect product of $ \{ A_i \}_{i \in I} $} if $ A $ is a subring of $ \prod_{i \in I} A_i  $ and $ \pi_{i \upharpoonright A}: A \lra A_i  $ is surjective for all $ i \in I $; if $ A  $ is a subdirect product of $ \{ A_i \}_{i \in I} $, define $ p_i:= \pi_{i \upharpoonright A} $ for all $ i\in I $.
	\end{enumerate}
\item If $ A $ and $ B $ are local rings, then an injective ring homomorphism $ f: A \linto B $ is \textit{local} if $ f^{-1}(\mathfrak{m}_B) = \mathfrak{m}_A $. 
\item If $ n \in \N^{\geq 2} $, then set $ [n]:= \{ 1, \dots, n \}  $.
\end{enumerate}

\subsection{Acknowledgements} 
This work is part of the author's PhD thesis done under the supervision  of Marcus Tressl at the University of Manchester  and supported by the President's Doctoral Scholar Award; the author is grateful to Marcus Tressl for all his feedback and comments during the course of this  research.

%% file: SV-rings.tex
\subsection{Preliminaries on SV-rings}\label{SUBSEC.prelim_SV-rings}

\begin{definition}\label{def.SV-ring}
	A ring  $ A $ is an \textit{SV-ring} if $ A/\mathfrak{p} $ is a valuation ring for all $ \mathfrak{p}\in \text{Spec}(A) $.
\end{definition}

Residue domains of valuation rings are valuation rings, therefore valuation rings are the first examples of SV-rings. Rings of Krull dimension 0 (i.e., rings in which every prime ideal is maximal) are clearly also SV-rings; for more examples of SV-rings see Proposition \ref{P.SV.constr_SV-ring} and Example \ref{E.SV.bool_prod}. The next theorem collects some known equivalent characterizations of SV-rings:

\begin{theorem}\label{T.SV.equiv_SV-ring} 
Let $ A  $ be a ring. The following are equivalent:
\begin{enumerate}[\normalfont(i), ref=\ref{T.SV.equiv_SV-ring} (\roman*)]
\item $ A $ is an SV-ring.
\item\label{T.SV.equiv_SV-ring.min_prime}  $ A/\mathfrak{p} $ is a valuation ring for all $ \mathfrak{p} \in \emph{Spec}^{\emph{min}}(A)$.
\item $ A_{\emph{red}}:=A/\emph{Nil}(A) $ is an SV-ring.
\item $ \emph{Spec}(A) $ is a normal space\footnote{Equivalently, $ A $ is a \textit{Gelfand ring}; see \cite[199]{johnstone.stone_spaces} and \cite[Theorem 4.3]{schwartz/tressl.minmax}.} and the localization $  A_{\mathfrak{m}}$ is an SV-ring for every maximal ideal $ \mathfrak{m} $ of $ A $.
\item For all  $ a, b \in A_{\emph{red}} $ there exists a polynomial $ P\in A_{\emph{red}}[X,Y]$ of the form $P(X, Y):= \prod_{i =1}^r(X-c_i\cdot Y)  $ such that $ P(a,b)\cdot P(b,a)= 0 $.
\end{enumerate}
Moreover, if $ A  $ satisfies any of the conditions \emph{(i)} - \emph{(vi)} above, then  $ A_{\emph{red}} $ is isomorphic to a subdirect product \emph{(\ref{S.conventions.subdirect})} of valuation rings.
\end{theorem}

\begin{proof}
	The equivalence of (i) - (iii) is clear, the equivalence of (i) and (iv) is Proposition 1.5 in \cite{schwartz.SV}, and the equivalence of (i) and (v)  follows from the equivalence of (i) and (iii) together with Theorem 3.4 in \cite{schwartz.SV}.    To conclude, suppose that $ A $ is an SV-ring and let $ B:= A_{\text{red}} $;  $ B $ is a reduced SV-ring by the implication (i) $ \Ra $ (iii), therefore canonical map $ B \lra \prod_{\mathfrak{p} \in \text{Spec}(B)} B/\mathfrak{p} $  is injective and its image is a subdirect product of the valuation rings $ \{ B/\mathfrak{p} \}_{\mathfrak{p} \in \text{Spec}(B)}  $. 
\end{proof}

It follows from the equivalences (i) $ \Leftrightarrow $  (iii) $ \Leftrightarrow $  (iv) in Theorem \ref{T.SV.equiv_SV-ring} that in the study of SV-rings, the  class  of reduced local SV-rings is at the forefront.

\begin{remarks}
\begin{enumerate}[\normalfont(i)]
	\item Rings with normal Zariski spectrum are abundant in the realm of real algebra; in particular, the Zariski spectrum of a real closed ring is normal (Definition \ref{D.RCSVR.rcr} and Theorem \ref{T.RCSVR.properties_rcr.II.iii}), and thus  for such class of rings the SV property is a \enquote{local property} by the equivalence  (i) $ \Leftrightarrow $  (iv) in Theorem \ref{T.SV.equiv_SV-ring}.
	\item In general, subdirect products of valuation rings are not SV-rings:  If $ X $ is not an SV-space (\cite{henriksen/wilson.when_is_CX_a_val_ring}), then the ring $ C(X)  $ of continuous real-valued functions on $ X $ is a subdirect product of valuation rings which is not an SV-ring.   On the other hand, every subdirect product\footnote{The model theory of subdirect products of structures has been investigated via a $ 2 $-sorted set-up in \cite{weispfenning.subdirect}; see also \cite{weispfenning.latprod}.} of finitely many valuation rings is an SV-ring (see Subsection \ref{SUBSEC.str_thm} and Remark \ref{R.SV.semi-local}).
\end{enumerate}	
\end{remarks}

One way to obtain examples of SV-rings  to find ring-theoretic constructions  which preserve the SV property; the next proposition summarizes the main such constructions:

\begin{proposition}\label{P.SV.constr_SV-ring}
The class of SV-rings is closed under formation of residue rings and localizations by multiplicative subsets. Moreover:
\begin{enumerate}[\normalfont(i), ref=\ref{P.SV.constr_SV-ring} (\roman*)]
\item\label{P.SV.constr_SV-ring.i} A finite direct product of rings is an SV-ring if and only if each factor is an SV-ring.
\item Direct limits of SV-rings are SV-rings.
\item\label{P.SV.constr_SV-ring.iii} Direct products of valuation rings are SV-rings.
\item\label{P.SV.constr_SV-ring.iv} If $ A $ and $ B $ are SV-rings and $ f: A \lonto C$ and  $ g: B\lonto C $ are surjective ring homomorphisms onto a  ring $ C $, then fibre product  $ A \times_{C} B$ is an SV-ring.
\end{enumerate}
\end{proposition}

\begin{proof}
That the class of SV-rings is closed under the formation of residue rings and localizations by multiplicative subsets is clear, and item (i)  follows from the characterization of prime ideals in a finite direct product of rings; for the proof of items (ii) - (iv) see Example 1.2 in \cite{schwartz.SV}.
\end{proof}

 In particular, if $ V_1 $ and $ V_2 $ are non-trivial valuation rings with isomorphic residue field $ \textbf{\emph{k}} $, then by Proposition \ref{P.SV.constr_SV-ring.iv} the fibre product $ V_1\times_{\emph{\textbf{k}}}V_2$ is an SV-ring;  $ V_1\times_{\emph{\textbf{k}}}V_2 $ should be thought as constructed by \enquote{gluing} the valuation rings $ V_1 $ and $ V_2 $ in a particular way, and this construction of SV-rings via fibre products is a central theme of this note. The next example, which generalizes Proposition \ref{P.SV.constr_SV-ring.iii}, shows how to construct some SV-rings  by gluing valuations rings via a sheaf construction: 

\begin{example}\label{E.SV.bool_prod}
	Let $ A $ be a \textit{boolean product} of valuation rings in the language $ \mathscr{L}(\text{div}) $, where $ \mathscr{L}(\text{div})$ is the language of rings $ \mathscr{L} := \{ +, -, \cdot, 0, 1 \} $ together with a binary predicate $ \text{div} $ interpreted as divisibility; i.e., there exists a Boolean space  $ X  $ and a Hausdorff sheaf $ \mathcal{O}_X $ of valuation rings on $ X $  (that is,   the stalk $ \mathcal{O}_{X, x} $ is a valuation ring for each $ x \in X $) such that  $  A $ is the $ \mathscr{L}(\text{div}) $-structure $  \Gamma(X, \mathcal{O}_X) $ of global continuous sections, see \cite{burris/werner.sheaf} and \cite[Section 2]{guier.boolean}. It will be shown that $ A $ is an SV-ring; to this end, set $ A_{U}:= \Gamma(U, \mathcal{O}_X) $ for each $ U \subseteq X $ open, $ A_x:=   \mathcal{O}_{X, x}$ for each $ x \in X $, and write $ a_{\upharpoonright U} $ for the image of  $ a\in A $ in $ A_U $ and $ a(x) $ for the image of $ a \in A $ in $ A_x $. Pick $ a, b \in A $ and note that $ X= U \ \dot{\cup}\ V \ \dot{\cup} \ W $, where \[
		U:= \llbracket \text{div}(a, b) \rrbracket \cap (X \setminus \llbracket \text{div}(b, a) \rrbracket), \ 		V:= \llbracket \text{div}(b, a) \rrbracket \cap (X \setminus \llbracket \text{div}(a, b) \rrbracket), \ W:= \llbracket \text{div}(a, b) \rrbracket \cap \llbracket \text{div}(b, a) \rrbracket,
	\] and $ \llbracket \text{div}(a,b)\rrbracket := \{ x \in X \mid A_x \models \text{div}(a(x), b(x))\} $; since $ A  $ is a Boolean product in the language $ \mathscr{L}(\text{div}) $, $ U, V, W \subseteq X $ are clopen, and a standard compactness argument (e.g. as in the proof of \cite[Lemma 1.1]{burris/werner.sheaf}) yields elements  $ d \in A_U $, $ e\in A_V $, and  $f\in A_W $ such that  $ a_{\upharpoonright U}d = b_{\upharpoonright U} $,  $ b_{\upharpoonright V}e = a_{\upharpoonright V} $, and  $ b_{\upharpoonright W}f=a_{\upharpoonright W} $, hence $ c:= d \cup e \cup f \in A $ is an element such that  $ (a-cb)(b-ca)=0 $, therefore $ A $ is an SV-ring by the implication  (v) $ \Ra $  (i) in Theorem \ref{T.SV.equiv_SV-ring}. 
\end{example}

\subsection{The rank of a ring}\label{SUBSEC.rank}

\noindent One way of classifying the complexity of SV-rings is via the following notion of rank of a ring (this definition can be found in \cite[Section 2]{schwartz.SV}):

\begin{definition}\label{def.rank}
	Let $ A $ be a ring and $ \infty $ be a symbol such that $ n < \infty $ for all  $ n \in \N $.
\begin{enumerate}[(i)]
	\item For $ \mathfrak{p} \in \text{Spec}(A) $, define $ \text{rk}(A, \mathfrak{p}) \in \N  $ to be the number of minimal prime ideals $ \mathfrak{q} $ of $ A $ such that $ \mathfrak{q} \subseteq \mathfrak{p} $ if this number is finite, and $ \text{rk}(A, \mathfrak{p})= \infty $ otherwise.
	\item The \textit{rank of $ A $} is $ \text{rk}(A):= \sup\{\text{rk}(A, \mathfrak{p})\mid \mathfrak{p} \in \text{Spec}(A)\} \in \N  \cup \{ \infty \} $.
	\end{enumerate}
	The ring $ A $ is \textit{of finite rank} if $ \text{rk}(A) \neq \infty $.
\end{definition}

 It has been already noted that reduced local SV-rings are in the forefront of the study of  SV-rings; amongst reduced local SV-rings, those of rank 1 are exactly the simplest SV-rings, namely valuation rings: 

\begin{lemma}\label{L.SV.val_ring_equiv}
Let $ A  $ be a ring. The following are equivalent:
\begin{enumerate}[\normalfont(i)]
\item $ A $ is a reduced local SV-ring of rank 1.
\item $ A $ is a valuation ring.
\end{enumerate}
\end{lemma}

\begin{proof}
	The implication (ii) $ \Ra $  (i) is trivial. Conversely, if item (i) holds, then  $ A $ being local of rank 1 implies that $ A $ has exactly one minimal prime ideal $ \mathfrak{q} $, therefore $\text{Nil}(A) = \bigcap_{\mathfrak{p} \in \text{Spec}(A)}\mathfrak{p} = \mathfrak{q} $ is a prime ideal of $ A $; since $ A  $ is reduced, $ \mathfrak{q} = (0) $, and since $ A $ is an  SV-ring,  $ A/\mathfrak{q} =  A $  is a valuation ring, as required.
\end{proof}

The same proof as in Lemma \ref{L.SV.val_ring_equiv} shows that local domains are exactly the reduced local rings of rank $ 1 $; in particular, if a reduced local ring is not a domain, then its rank is at least $ 2 $, so the rank of a reduced local ring is related with its zero divisors. The next lemmas clarify the relationship between the rank of a reduced local ring and its zero divisors (cf. \cite[Proposition 2.1]{schwartz.SV} and \cite[Theorem 3.6]{larson.SV_and_related_f-rings_and_spaces}); Lemma \ref{L.SV.ann_intersection_of_min_prime_id.iii} and Corollary \ref{C.SV.min_prime_ann} will be of particular importance in Section \ref{SEC.mod_th}.

\begin{lemma}\label{L.SV.ann_intersection_of_min_prime_id}
	Let $ A $ be a reduced ring and set $ \emph{Spec}^{\emph{min}}(A) := \{ \mathfrak{p}_i \mid i \in I \}  $.
	
	\begin{enumerate}[\normalfont(I), ref=\ref{L.SV.ann_intersection_of_min_prime_id} (\Roman*) ]
	\item If $ a \in A  $ is a non-zero zero divisor, then $ S:=\{ i \in I \mid a \notin \mathfrak{p}_i \} $  is a non-empty proper subset of $ I $ such that $\bigcap_{i \in S}\mathfrak{p}_i= \emph{Ann}(a) $.
	\item Let $ m \in \N^{\geq 2}  $. If $ a_1, \dots, a_m \in A$ are non-zero and pairwise \emph{orthogonal} (i.e., $ a_ia_j = 0 $ for all $ i,j \in [n]  $ such that $ i \neq j $), then $ S_j := \{ i \in I \mid a_j \notin \mathfrak{p}_i \}  $ $ (j \in [m]) $ are pairwise disjoint non-empty proper subsets of $I $ such that $\bigcap_{i\in S_j}\mathfrak{p}_i = \emph{Ann}(a_j)$ for all $ j \in [m] $.
\item \label{L.SV.ann_intersection_of_min_prime_id.iii} Suppose that $ A $ is local.
	\begin{enumerate}[\normalfont(i), ref=\ref{L.SV.ann_intersection_of_min_prime_id.iii} (\roman*) ]

		\item \label{L.SV.ann_intersection_of_min_prime_id.ii.a}	$ \emph{rk}(A)= \sup \{ m\in \N \mid \exists a_1, \dots, a_m \in A \emph{ non-zero and pairwise orthogonal} \}  $.
		\item\label{L.SV.ann_intersection_of_min_prime_id.ii.b} If $ \emph{rk}(A) = |I| = n \in \N^{\geq 2}$, then for all $ a_1, \dots, a_n \in A$ non-zero and pairwise orthogonal there exists a bijection $ \sigma: [n] \lra [n] $ such that $\mathfrak{p}_i = \emph{Ann}(a_{\sigma(i)})$ for all $ i  \in [n] $.
	\end{enumerate}
\end{enumerate}
\end{lemma}

\begin{proof}
	(I).  Since  $ A $ is reduced and $ a \in A $ is a non-zero zero divisor, $ a \notin \bigcap_{i \in I}\mathfrak{p}_i = (0)$ and $ a\in \bigcup_{i \in I} \mathfrak{p}_i  $, therefore $ S:= \{ i \in I \mid a \notin \mathfrak{p}_i \}   $ is a non-empty proper subset of $I $; note in particular that $ a \in \mathfrak{p}_i $ for all $ i \in I \setminus S $. The inclusion $ \text{Ann}(a) \subseteq \bigcap_{i \in S } \mathfrak{p}_i$ is clear; conversely, if $ b \in  \bigcap_{i \in S } \mathfrak{p}_i$, then $ ba \in (\bigcap_{i \in S } \mathfrak{p}_i) \cap (\bigcap_{i \in I \setminus S } \mathfrak{p}_i) = (0)$, hence $ b \in \text{Ann}(a) $. 

	(II). Each $ a_j$  is a non-zero zero divisor, therefore by (I) it remains to show that the sets $ S_1,\dots, S_m  $ defined in (II) are pairwise disjoint. Assume for contradiction that there exist $ j, j' \in [m] $ with $ j \neq j' $ such that there exists $ i \in S_j \cap S_{j'}  $, i.e., $ a_{j}, a_{j'} \notin \mathfrak{p}_i $; since  $ j \neq j' $, $ a_{j}a_{j'} =0 \in \mathfrak{p}_i$, giving the required contradiction.

	(III) (i). Note first that  $ \text{rk}(A) = |\text{Spec}^{\text{min}}(A)| = |I| \in \N \cup  \{ \infty \}  $ by assumption on $ A $. Define $ \kappa:= \sup \{ m\in \N\mid \exists a_1, \dots, a_m \in A \text{  non-zero and pairwise orthogonal} \} \in \N \cup \{ \infty \}   $; if $ \kappa \notin \N $, then $ \text{rk}(A) \notin \N $ by (II), therefore $ \text{rk}(A) = \kappa $, and if $ \kappa = 1 $, then  $ A $ is a domain, therefore $ \text{rk}(A)= 1 = \kappa $. Suppose now that $\kappa = n \in \N^{\geq 2}$; let $ a_1, \dots, a_n \in A $ be non-zero pairwise orthogonal elements witnessing  $ \kappa = n $, and define the subsets $ S_j \subseteq I $ as in item  (II); by (II), $ n \leq \text{rk}(A) $, so assume for contradiction that $ n <  \text{rk}(A)$. 

	\underline{Case 1.} There exists $ i \in I \setminus (S_1\  \dot{\cup}\ \dots \ \dot{\cup}\ S_n)$. In this case $ \mathfrak{p}_i $ is a minimal prime ideal such that $ a_1, \dots, a_n \in \mathfrak{p}_i $; since $ \mathfrak{p}_i $ is a minimal prime ideal, by \cite[Proposition 1.2 (1)]{matlis.minimal}  there exists $ b_j \in A \setminus \mathfrak{p}_i $ such that $ a_jb_j = 0 $ for each $ j \in [n] $, therefore $ b:= b_1 \cdot \ldots \cdot b_n $ is a non-zero element which is orthogonal to each  $ a_i $, a contradiction to  $ \kappa = n $.

	\underline{Case 2.} $ I = S_1\  \dot{\cup}\ \dots \ \dot{\cup}\ S_n$ and there exists $ j \in [n] $ such that $ |S_j | \geq 2 $. Assume without loss of generality that $ |S_1| \geq 2 $ and suppose that $ \mathfrak{p}_{i_1}  $ and $ \mathfrak{p}_{i_2}  $ are distinct minimal prime ideals such that $ i_1, i_2 \in S_1  $, so that $ a_1 \notin \mathfrak{p}_{i_1} $ and $ a_1 \notin \mathfrak{p}_{i_2} $. Pick $ b \in \mathfrak{p}_{i_1} \setminus \mathfrak{p}_{i_2} $; since $ \mathfrak{p}_{i_1}$ is a minimal prime ideal, by \cite[Proposition 1.2 (1)]{matlis.minimal} there exists $c \in A \setminus  \mathfrak{p}_{i_1} $ such that $ bc =0 $, therefore  $ a_1b, a_1c, a_2, \dots, a_n $ are non-zero pairwise orthogonal elements of $ A $, a contradiction to $ \kappa = n $.

	(III) (ii). Combine items (II) and (III) (i).
\end{proof}

\begin{remark}
	Let $ A  $ be a reduced local ring of rank $ n\in \N^{\geq 2} $	and write $ \text{Spec}^{\text{min}}(A):= \{ \mathfrak{p}_i \mid i \in [n] \}  $; the canonical map $ A \lra \prod_{i \in [n]}A/\mathfrak{p}_i$ is an embedding, and it follows from this that (up to re-labelling) any $ n $ non-zero pairwise orthogonal elements  $ a_1 \dots, a_n\in A $ satisfy  $ a_i \in \bigcap_{j \in [n]\setminus \{ i \}}\mathfrak{p}_j \setminus \mathfrak{p}_i $, so that $ \mathfrak{p}_i = \text{Ann}(a_i) $ for all $ i \in [n] $.
\end{remark}

\begin{corollary}\label{C.SV.min_prime_ann}
Let $ A $ and  $ B $ be reduced local rings of rank $ n\in \N^{\geq 2} $ such that $ A \subseteq B $. If $ a_1, \dots, a_n \in A $ are non-zero and pairwise orthogonal, then \[
	\emph{Spec}^{\emph{min}}(A) =\{ \emph{Ann}_A(a_i) \mid i \in [n]\}  \ \text{ and } \ 	\emph{Spec}^{\emph{min}}(B) =\{ \emph{Ann}_B(a_i) \mid i \in [n]\}
.\]  
\end{corollary}

\begin{proof}	
	Clear by Lemma \ref{L.SV.ann_intersection_of_min_prime_id.ii.b}.
\end{proof}

\begin{remark}\label{R.SV.local_embedd_factor}
	Let $ A $ and  $ B $ be reduced local rings of rank $ n\in \N^{\geq 2} $ such that $ A \subseteq B $. Write $ \text{Spec}^{\text{min}}(A) = \{ \mathfrak{p}_{A, i} \mid i \in [n] \}  $ and $ \text{Spec}^{\text{min}}(B) = \{ \mathfrak{p}_{B, i} \mid i \in [n] \}  $, and assume without loss of generality that $ \mathfrak{p}_{B, i}\cap A = \mathfrak{p}_{A, i}$ for all $ i \in [n] $ (Corollary \ref{C.SV.min_prime_ann}). The embedding $ A \subseteq B $ induces embeddings  $ A/\mathfrak{p}_{A, i} \subseteq B/\mathfrak{p}_{B,i} $ for all $ i \in [n] $, and  $ A \subseteq B $ is a local embedding if and only if  $ A/\mathfrak{p}_{A, i} \subseteq B/\mathfrak{p}_{B,i} $ is a local embedding for all (equivalently, for some) $ i \in [n] $: this follows from the fact that the residue field of $ A$ is isomorphic to the residue field  of $ A/\mathfrak{p}_{A,i} $ for all $ i \in [n] $, and similarly for $ B $.
\end{remark}

\begin{lemma}\label{L.SV.subdirect_min_prime_id} 
	Let $ n \in \N^{\geq 2} $, $ A_1, \dots, A_n $ be domains, and $ A $ be  a subdirect product of  $\{  A_1, \dots, A_n  \} $.   
	\begin{enumerate}[\normalfont(I), ref=\ref{L.SV.subdirect_min_prime_id} (\Roman*)]
		\item\label{L.SV.subdirect_min_prime_id.i} 		For every $ \mathfrak{p} \in \emph{Spec}^{\emph{min}}(A) $ there exists $ i \in [n] $ such that  $ \mathfrak{p} = \emph{ker}(p_i) $  \emph{(}see \emph{\ref{S.conventions.subdirect}} for the definition of $ p_i $\emph{)}; in particular, $ \emph{rk}(A) \leq n $.
	\item\label{L.SV.subdirect_min_prime_id.ii} Suppose that $ A  $ is local, and for each $ S \subseteq [n]  $ let $ \pi_S: \prod_{i=1}^n A_i \lonto \prod_{i \in S}A_i$ be the canonical projection. The following are equivalent:
		\begin{enumerate}[\normalfont(i)]
			\item $ \emph{rk}(A) = n $.
			\item  For  all $ S \subsetneq [n] $, the map $ p_S:= (\pi_S)_{\upharpoonright A} : A \lra \prod_{i \in S}A_i$ is not injective.
			\item $ \emph{ker}(p_i) $ and $ \emph{ker}(p_j) $ are incomparable under subset inclusion  for all $ i , j \in [n] $  with $ i \neq j  $.
		\end{enumerate}
		Moreover, if any of the conditions \emph{(i)} - \emph{(iii)} holds, then $ \emph{Spec}^{\emph{min}}(A) = \{\emph{ker}(p_i) \mid i \in [n]\} $.
	\end{enumerate}
\end{lemma}

\begin{proof}
	(I). Let $ \mathfrak{p} \in \text{Spec}^{\text{min}}(A) $. Then $ \text{ker}(p_1)\cdot  \ldots \cdot\text{ker}(p_n) = (0) \subseteq \mathfrak{p} $, therefore there exists $ i \in [n]  $ such that $ \text{ker}(p_i) \subseteq \mathfrak{p} $; but $ \text{ker}(p_i) \in \text{Spec}(A)$, hence $ \text{ker}(p_i)= \mathfrak{p} $ by minimality of  $ \mathfrak{p} $.	

(II). Straightforward using (I).
\end{proof}

\subsection{Structure theorem for reduced local SV-rings of finite rank}\label{SUBSEC.str_thm}

The main statement in this subsection is Theorem \ref{T.SV.equiv_red_loc_SV-ring_finite_rank}, which  is just a refined formulation of the observation made in \cite[Remark 3.2]{schwartz.SV} that every reduced local SV-ring of finite rank can be constructed from valuation rings using iterated fibre products.

\begin{lemma}\label{L.SV.fibre}
Let $ A $ be a ring and  $ I, J \subseteq A$ be ideals. The ring homomorphism $ f: A \lra A/I \times A/J $ given by  $ f(a):=(a/I, a/J) $ has image $ \frac{A}{I}\times_{\frac{A}{I+J}} \frac{A}{J} $; in particular, $ \frac{A}{I \cap J} \cong \frac{A}{I}\times_{\frac{A}{I+J}} \frac{A}{J}$.
\end{lemma}

\begin{proof}
 Clearly $ \text{im}(f) \subseteq   \frac{A}{I}\times_{\frac{A}{I+J}} \frac{A}{J}$; moreover, if $ (a/I, b/J) \in \frac{A}{I}\times_{\frac{A}{I+J}} \frac{A}{J}$,	then there exist $c\in I  $ and $ d \in J $ such that $ a-b = c+d $, hence $ a-c = b+d $ and  thus $  (a/I, b/J) = f(a-c) = f(b-d) $, therefore $  \text{im}(f) =\frac{A}{I}\times_{\frac{A}{I+J}} \frac{A}{J}$. The last assertion follows by noting that $ \text{ker}(f) = I \cap J $.
\end{proof}

\begin{lemma}[Goursat's lemma for rings]\label{L.SV.goursat}
Let $ A_1 $ and $ A_2 $ be rings and  $ A \subseteq A_1 \times A_2 $ be a subring. The following are equivalent:
\begin{enumerate}[\normalfont(i)]
\item $ A $ is a subdirect product of $ \{ A_1, A_2 \}  $.
\item There exist ideals $ I\subseteq A_1 $, $ J \subseteq A_2 $, and an isomorphism  $ f: A_1/I \overset{\cong}{\lra} A_2/J $ such that $ A = A_1\times_{A_2/J} A_2  $.
\item There exist ideals $ H_1, H_2\subseteq A $ such that  $  H_1 \cap H_2 = (0) $, and there exist isomorphisms $ g_i: A_i \overset{\cong}{\lra} A/H_i $ \emph{(}$ i \in \{ 1, 2 \}  $\emph{)} yielding an isomorphism $ g : A_1 \times A_2 \overset{\cong}{\lra} A/H_1 \times A/H_2 $  such that $ g_{\upharpoonright A}(a) = (a/H_1, a/H_2) $ for all $ a \in A $, and $ g_{\upharpoonright A} $ is an isomorphism $ A \overset{\cong}{\lra} \frac{A}{H_1} \times_{\frac{A}{H_1\oplus H_2}}\frac{A}{H_2} $.
\end{enumerate}
\end{lemma}

\begin{proof}
	\underline{(i) $ \Ra $  (ii).} The following proof can be found in \cite[Lemma 2]{larson.finitely_1-conv_f-rings}; for future reference it is included here. Let $ I:= p_1(\text{ker}(p_2)) $ and $ J := p_2(\text{ker}(p_1))$, noting that since $ p_1 $ and $ p_2 $ are surjective, $ I\subseteq A_2$ and $ J \subseteq A_2 $ are ideals (see \ref{S.conventions.subdirect} for the definition of $ p_i $). 

	\noindent \textit{Claim.} The assignment $ f: A_1/I \lra A_2/J $ given by \[
		a_1/I \mapsto a_2/J  \overset{\text{def}}{\iff} (a_1, a_2) \in A
	\] 
	for all $ (a_1, a_2) \in A_1 \times A_2 $ is a well-defined ring isomorphism.

	\noindent \textit{Proof of Claim.} To see that $ f $ is indeed a function, pick $ a_1, b_1 \in A_1 $ and $ a_2, b_2 \in A_2 $ such that  $ a_1-b_1 \in I $ and $ (a_1, a_2), (b_1, b_2) \in A $; then $  (a_1-b_1, a_2-b_2) \in A $, and since $ a_1-b_1 \in I=p_1(\text{ker}(p_2)) $, $ (a_1-b_1, 0) \in A $, hence $ 0/J= f((a_1-b_1)/I)= (a_2-b_2)/J$, from which $ a_2/J = b_2/J $ follows. That $ f $ is a ring homomorphism follows from $ A $ being a subring of $ A_1\times A_2 $, and since $ A $ is a subdirect product of $ \{ A_1, A_2 \}  $ it follows that $ f $ is surjective, so it remains to show that $ f $ is injective. Pick $ a_1, b_1 \in A_1 $ and suppose that  $ f(a_1/I) = f(b_1/I) $, hence there exist $ a_2, b_2 \in A_2 $ with  $ a_2-b_2 \in J $ and $ (a_1, a_2), (b_1, b_2) \in A $; then $ (a_1-b_1, 0) = (a_1, a_2)-(b_1, b_2)-(0, a_2-b_2)\in A $, yielding $ a_1/I= b_1/I $. \hfill $ \square_{\text{Claim}} $

	Let $ q_I: A_1\lonto A_1/I $ and $ q_J: A_2 \lonto A_2/J $ be the projection maps. By the claim above, $ (f \circ q_I) \circ p_1 =  q_J \circ p_2 $, so to show that $ A = A_1 \times_{A_2/J}A_2 $, it suffices to show that $ A $ verifies the universal property of the fibre product $   A_1 \times_{A_2/J}A_2 $. Let $ f_i: C \lra A_i $ be ring homomorphisms such that $ (f \circ  q_I) \circ f_1 =  q_J \circ f_2 $, and define $ h: C \lra A_1\times A_2 $ by  $ h(c):= (f_1(c), f_2(c)) $; since $ f(f_1(c)/I) =f_2(c)/J $ by choice of $ f_1 $ and $ f_2 $, it follows by definition of $ f $ that $ (f_1(c),f_2(c)) \in A $ for all $ c \in C $, so $ h(C)\subseteq A $ and thus the corestriction of $ h  $  to  $ A $ verifies the universal property of $   A_1 \times_{A_2/J}A_2 $.

	\underline{(ii) $ \Ra $  (iii).} Let  $ q_I: A_1 \lonto A_1/I $ and $ q_J: A_2\lra A_2/J  $ be the projection maps.  Since $ f\circ q_I :A_1 \lonto A_2/J $ and $ q_J: A_2 \lonto A_2/J $ are surjective and $ A = A_1 \times_{A_2/J} A_2 $, $ p_i:= \pi_{i \upharpoonright A} : A \lonto A_i $ ($ i \in \{ 1, 2 \}  $) are surjective ring homomorphisms such that $ (f\circ q_I) \circ p_1 = q_J\circ p_2$. Set $ H_i:= \text{ker}(p_i) $  for $ i \in \{1,2\}  $; then $ H_1, H_2 \subseteq A $ are ideals of  $ A $ such that  $ H_1 \cap H_2 = (0) $ which induce isomorphisms $ g_i : A_i \overset{\cong}{\lra} A/H_i  $ ($ i \in \{ 1, 2 \}  $) yielding an isomorphism   $ g: A_1\times A_2 \overset{\cong}{\lra} A/H_1 \times A/H_2 $ given by   $ (a_1, a_2) \mapsto (g_1(a_1), g_2(a_2)) $. If $ a:=(a_1, a_2) \in A \subseteq A_1 \times A_2$, then  $ g(a) = (a/H_1, a/H_2) $ by definition of $ g $, therefore $ g_{\upharpoonright A} : A \lra A/H_1 \times A/H_2 $ is an isomorphism  onto its image $  \frac{A}{H_1} \times_{\frac{A}{H_1\oplus H_2}}\frac{A}{H_2}  $ by Lemma \ref{L.SV.fibre}.

	\underline{(iii) $ \Ra $  (i).} Let $ i \in \{1, 2\} $. By (iii), the following diagram  

\noindent\adjustbox{center}{
\begin{tikzpicture}[node distance=2cm, auto]
  \node (lb) {$\frac{A}{H_1}\times_{\frac{A}{H_1 \oplus H_2}}\frac{A}{H_2}$};
  \node (lt) [above of=lb] {$A$};
  \node (rb) [right = 2cm of lb] {$\frac{A}{H_1}\times\frac{A}{H_2}$};
  \node (rt) [above of=rb] {$A_1 \times A_2$};
  \node (rrb) [right = 2cm of rb] {$A/H_i$};
  \node (rrt) [above of=rrb] {$A_i$};
  \draw[<-] (lb) to node [swap] {$g_{\upharpoonright A}$} (lt);
  \draw[{Hooks[right]}->] (lb) to node {$\subseteq$} (rb);
  \draw[<-] (rb) to node [swap] {$g$} (rt);
  \draw[{Hooks[right]}->] (lt) to node {$\subseteq$} (rt);
  \draw[->>] (rt) to node  {} (rrt);
  \draw[->>] (rb) to node  {} (rrb);
  \draw[->] (rrt) to node  {$ g_i $}  (rrb);
\end{tikzpicture}
}

\noindent commutes;  since all the vertical arrows are isomorphisms, the composite top morphism is surjective if and only if the composite bottom morphism is surjective; since $ \frac{A}{H_1} \times_{\frac{A}{H_1\oplus H_2}} \frac{A}{H_2}  $ is a subdirect product of  $ \{ A/H_1, A/H_2  \} $, (i) follows. 
\end{proof}

\begin{remark}\label{R.SV.cofactor}
	Let $ A_1 $ and $ A_2 $ be rings and $ A \subseteq A_1 \times A_2 $ be a subring satisfying any of the items (i) - (iii) in Lemma \ref{L.SV.goursat}; then $ \frac{A}{H_1\oplus H_2} \cong A_1/I \ (\cong A_2/J) $, where $ I \subseteq A_1 $ and $ J \subseteq A_2 $ are ideals as in item  (ii) in Lemma \ref{L.SV.goursat}, and $ H_1, H_2 \subseteq A $ are ideals as in item (iii) in Lemma \ref{L.SV.goursat}. Indeed, by the proofs of the implications (i) $ \Ra $  (ii) and  (ii) $ \Ra $  (iii) in Lemma \ref{L.SV.goursat},  $ I= p_1(\text{ker}(p_2)) $ and $ H_i = \text{ker}(p_i) $ for $ i \in \{ 1, 2 \}  $, hence $ I= p_1(H_2) $; it is easy to check from this that  $ \frac{a}{H_1\oplus H_2} \mapsto p_1(a)/I $ is a well-defined isomorphism $ \frac{A}{H_1\oplus H_2} \lra A_1/I  $.
\end{remark}

\begin{corollary}\label{C.SV.local_goursat}
Let $ A_1 $ and $ A_2 $ be local rings and  $ A \subseteq A_1 \times A_2 $ be a subring. The following are equivalent:
\begin{enumerate}[\normalfont(i)]
\item $ A $ is a local ring and subdirect product of $\{  A_1, A_2  \} $.
\item There exist ideals $ I\subsetneq A_1 $ and  $ J \subsetneq A_2 $ and a ring isomorphism $ f: A_1/I \overset{\cong}{\lra} A_2/J $ such that  $ A = A_1\times_{A_2/J} A_2  $.
\item There exist ideals $ H_1, H_2\subseteq A $  such that  $  H_1 \cap H_2 = (0) $ and $ H_1 \oplus H_2 \neq A $, and there exist isomorphisms $ g_i : A_i \overset{\cong}{\lra} A/H_i $ \emph{(}$ i\in \{1, 2\}  $\emph{)} yielding an isomorphism $ g: A_1\times A_2 \overset{\cong}{\lra} A/H_1 \times A/H_2$ such that $ g_{\upharpoonright A}(a) = (a/H_1, a/H_2) $ for all $ a \in A $, and $ g_{\upharpoonright A} $ is an isomorphism $ A \overset{\cong}{\lra} \frac{A}{H_1} \times_{\frac{A}{H_1\oplus H_2}}\frac{A}{H_2} $.
\end{enumerate} 

\noindent Moreover, if any of the items \emph{(i)} - \emph{(iii)} holds, then one may choose:
\begin{enumerate}[\normalfont(a), ref=\ref{C.SV.local_goursat} (\alph*)]
	\item 	$ I:=  p_1(\emph{ker}(p_2))$ and $ J:=  p_2(\emph{ker}(p_1))$ to be the ideals of $ A_1 $ and $ A_2 $ in item \emph{(ii)} (respectively) and $ f: A_1/I \lra A_2/J $ to be the map $ a_1/I\mapsto a_2/J $; and
\item\label{C.SV.local_goursat.b} $ H_i:= \emph{ker}(p_i)$ to be the ideals of $ A $  \emph{(}$ i \in \{1, 2\}    $\emph{)} in item  \emph{(iii)}, and $ g_i : A_i \lra A/H_i$ to be the corresponding induced isomorphisms. 
\end{enumerate}
\end{corollary}

\begin{proof}
	The equivalence of items (i) - (iii) is immediate from Lemma \ref{L.SV.goursat} noting  that if $ B $ and $ C $ are local rings and $g: B \lonto D   $ and $ h : C \lonto D $ are surjective ring homomorphisms, then   $ B \times_D C $ is  a local ring if and only if $ D  $ is not the zero ring. 	 The last statement in the lemma follows from the proofs of the implications  (i) $ \Ra $  (ii) and  (ii) $ \Ra $  (iii) in Lemma \ref{L.SV.goursat}. 
\end{proof}

\begin{notation}\label{N.SV.pimj}
Let $ n \in \N^{\geq 2} $ and suppose that $ A_1,\dots, A_n $ are rings. For each $ k \in [n] $ and $ j \in [k]  $, let  $ \pi^{k}_{ j}:   \prod_{i=1}^{k}A_i \lonto \prod_{i= 1}^j A_i$  be the projection map onto the first $ j $ factors; in particular, $ \pi^{n}_1 = \pi_1 $ and $ \pi^n_n  $ is the identity map.
\end{notation}

\begin{proposition}\label{P.SV.goursat_fin_many}
Let $ n \in \N^{\geq 2} $, $ A_1, \dots, A_n $ be local rings, and  $ A \subseteq \prod_{i=1}^n A_i$ be a subring.  The following are equivalent:
\begin{enumerate}[\normalfont(i)]
\item $ A $ is a local ring and a subdirect product of $ \{ A_1, \dots,  A_n  \} $.
\item $ \pi^n_1(A)= A_1 $ \emph{(}see \emph{Notation \ref{N.SV.pimj}}\emph{)} and for all $ j \in [n-1]  $ there exist ideals $ I_j\subsetneq \pi^n_{j}(A) $  and  $ J_{j+1}  \subsetneq A_{j+1} $ such that $ \pi^n_{j}(A)/I_j\cong A_{j+1}/J_{j+1}  $ and $$\pi^n_{ j+1}(A) = \pi^n_{j}(A)\times_{A_{j+1}/J_{j+1}} A_{j+1};$$in particular,  $$\hspace{1.3cm} \pi^n_n(A) = A = (((A_1 \times_{A_2/J_2} A_2) \times_{A_3/J_3}A_3) \dots \times_{A_{n-1}/J_{n-1} } A_{n-1})\times_{A_{n}/J_{n}}  A_{n} .$$ 
\item There exist ideals $ H_1, \dots, H_{n} \subseteq A $ such that  $  \bigcap_{i =1}^n H_i  = (0) $ and $ G_j := (\bigcap_{i=1}^{j-1}H_i)+ H_{j} \neq A$ for all $ j \in \{2, \dots, n\}  $, and there exist isomorphisms $ g_i: A_i \overset{\cong}{\lra} A/H_i $ \emph{(}$ i\in [n]  $\emph{)} yielding an isomorphism $ g: \prod_{i =1}^n A_{i} \overset{\cong}{\lra} \prod_{i=1}^n A/H_i$ such that $ g_{\upharpoonright A}(a)= (a/H_1, \dots, a/H_n) $, and $ g_{\upharpoonright A} $ is an isomorphism    $$ \hspace{1.2cm}A \cong \left(\left(\left(\frac{A}{H_1} \times_{\frac{A}{G_2}}\frac{A}{H_2}\right) \times_{\frac{A}{G_3}}\frac{A}{H_3}\right) \dots \times_{\frac{A}{G_{n-1}}}\frac{A}{H_{n-1}}\right)\times_{\frac{A}{G_{n}}}\frac{A}{H_{n}} .$$
\end{enumerate}
\end{proposition}

\begin{proof}
	The proof is by induction on $ n \in \N^{\geq 2} $. The base case $ n = 2 $ is exactly Corollary \ref{C.SV.local_goursat}, so assume now that the proposition holds true for some $n:=k \in \N^{>2}  $. Let $ A_1, \dots, A_{k}, A_{k+1}  $ be local rings, $ A \subseteq \prod_{i=1}^{k+1}A_i $ be a subring, and define $ B := \pi^{k+1}_{ k}(A) \subseteq \prod_{i =1}^k A_i $; note in particular that  $ \pi^k_{j}(B)= \pi^{k+1}_{j}(A) $ for all $ j \in [k]  $.

	\underline{(i)$ \Ra $  (ii).} Since $ A $ is a subdirect product of $ \{ A_1, \dots, A_k, A_{k+1} \}   $, $ \pi_1^{k+1}(A)= A_1 $ and $B$ is a subdirect product of $ \{ A_1, \dots, A_k  \} $, and since $ B $ is a homomorphic image of $ A $ and $ A $ is local, $ B $ is also local,  hence by inductive hypothesis, for all $ j \in [k-1]  $ there exist ideals  $  I_j\subsetneq \pi^k_{ j}(B) $ and  $J_{j+1}  \subsetneq A_{j+1} $ such that $ \pi^k_{j}(B)/I_j\cong A_{j+1}/J_{j+1}  $ and $ \pi^k_{j+1}(B) = \pi^k_{ j}(B) \times_{A_{j+1}/J_{j+1}}A_{j+1} $; since $ A $ is also a subdirect product of $ \{ B, A_{k+1} \}  $ and $ B $ is a local, by the implication (i) $ \Ra $  (ii) in Corollary \ref{C.SV.local_goursat} there exist ideals $ I_{k}\subsetneq B  $ and $  J_{k+1} \subsetneq A_{k+1} $ such that $ B/I_k \cong A_{k+1}/I_{k+1} $ and $ A = B \times_{A_{k+1}/J_{k+1}} A_{k+1}$; (ii) now follows from  $ \pi^k_{j}(B) = \pi^{k+1}_{ j}(A) $ for all $ j \in [k-1]  $ and $ B= \pi^{k+1}_{k}(A)  $.

	\underline{(ii) $ \Ra $ (iii).} Since $ \pi^{k+1}_1(A) = A_1 $ by assumption, $ \pi^k_1(B) = A_1 $. Therefore, by inductive hypothesis there exist  ideals $ H'_1, \dots, H'_{k} \subseteq B $ such that  $  \bigcap_{i =1}^k H'_i  = (0) $,  $ G'_j := (\bigcap_{i=1}^{j-1}H'_i)+ H'_{j} \neq B$ for all $ j \in \{2, \dots, k\}  $, and there exist isomorphisms $ g'_i: A_i \overset{\cong}{\lra} B/H_i' $ ($ i\in [k]  $) yielding an isomorphism $ g': \prod_{i =1}^k A_{i} \lra \prod_{i=1}^k B/H'_i$ such that $ (g')_{\upharpoonright B}(b)= (b/H_1', \dots, b/H_k') $ for all $ b \in B $, and $ (g')_{\upharpoonright B} $ is an isomorphism  $$ \hspace{1.2cm} B \cong \left(\left(\left(\frac{B}{H'_1} \times_{\frac{B}{G'_2}}\frac{B}{H'_2}\right) \times_{\frac{B}{G'_3}}\frac{B}{H'_3}\right) \dots \times_{\frac{B}{G'_{k-1}}}\frac{B}{H'_{k-1}}\right)\times_{\frac{B}{G'_{k}}}\frac{B}{H'_{k}};$$  note in particular that since $ A_i \cong B/H'_i $ is a local ring for all $ i \in [k]  $ and  $ G'_j \neq B $ for all $ j \in \{ 2, \dots, k \}  $, $ B $ is also a local ring.  Since $ A = \pi_{k+1}^{k+1}(A) = \pi^{k+1}_k(A) \times_{A_{k+1}/J_{k+1}}A_{k+1} = B \times_{A_{k+1}/J_{k+1}}A_{k+1} $ by assumption, by the implication (ii) $ \Ra $  (iii) in Corollary \ref{C.SV.local_goursat} there exist ideals $ H''_1, H''_2 \subseteq A  $ such that $ H''_1 \cap H''_2 = (0) $ and $ H_1'' \oplus H_1'' \neq A $, and there exist isomorphisms  $ g_1'' : B \overset{\cong}{\lra}A/H''_1  $ and $  g''_2: A_{k+1} \overset{\cong}{\lra} A/H''_2  $ yielding an isomorphism $ g'': B \times A_{k+1} \lra A/H_1'' \times A/H_2'' $  such that $ (g'')_{\upharpoonright A}(a) = (a/H_1'', a/H''_2) $ for all $  a \in A $, and $ (g'')_{\upharpoonright A}  $ is an isomorphism $ A \cong A/H''_1\times_{A/(H''_1\oplus H''_2)}A/H''_2 $;  moreover, by Corollary \ref{C.SV.local_goursat.b}, $ H''_1= \text{ker}((\pi^{k+1}_{k})_{\upharpoonright A}) $ and $ H''_2= \text{ker}((\pi_{k+1})_{\upharpoonright A})) $. Define $H_i:= ((\pi^{k+1}_{k})_{\upharpoonright A})^{-1}(H'_i)  $ for all $ i \in [k]  $, $ H_{k+1}:=H''_2  $, and $ G_j:= (\bigcap_{i=1}^{j-1}  H_i)+H_j $ for all $ j \in \{2, \dots, k+1\}  $. Then \[
		\bigcap_{i=1}^{k}H_i = \bigcap_{i =1}^k((\pi^{k+1}_{k})_{\upharpoonright A})^{-1}(H'_i) = ((\pi^{k+1}_{k})_{\upharpoonright A})^{-1}\left( \bigcap_{i=1}^k H'_i \right)    = ((\pi^{k+1}_{k})_{\upharpoonright A})^{-1}(0) = \text{ker}((\pi^{k+1}_{k})_{\upharpoonright A}) = H''_1  
	,\]
	hence $ \bigcap_{i=1}^{k+1}H_i = H''_1 \cap H''_2 = (0) $ and $ G_{k+1}= (\bigcap_{i=1}^k H_i ) + H_{k+1} =  H_1''\oplus H_2'' \neq A $.  Moreover, since $ (\pi^{k+1}_k)_{\upharpoonright A}: A \lonto B $ is surjective, it follows from the above  that $ (\pi^{k+1}_k)_{\upharpoonright A} $ induces isomorphisms $  h_i: B/H'_i\lra A/H_i$  given by $ \pi^{k+1}_k(a)/H'_i \mapsto a/H_i $ for all $ i \in [k]  $, and also that $ G_j = ((\pi^{k+1}_{k})_{\upharpoonright A})^{-1}(G_j') $  and  $A/G_j\cong B/G'_j  $ for all $ j \in \{2, \dots, k\}  $, hence $ G_j \neq A $ for all  $ j \in \{2, \dots, k\}  $, and the isomorphisms $ g_i:= (h_i \circ g_i'): A_i \lra A/H_i$ ($ i \in [k]  $) and $ g_{k+1}:= g_2'': A_{k+1} \lra A/H_{k+1} $ yield an isomorphism $ g: \prod_{i =1}^{k+1} A_{i} \lra  \prod_{i =1}^{k+1} A/H_i $ verifying the last statement in item (iii).

	\underline{(iii) $ \Ra $ (i).} Clear using the same argument as in the proof of the implication (iii) $ \Ra $  (i) in Lemma \ref{L.SV.goursat} and noting that $ A  $ is a local ring since $ A/H_i $ is a local ring for all $ i \in [k+1]  $ and $ G_j \neq A $ for all  $ j \in  \{ 2 \dots,k +1  \}  $.  
\end{proof}

\begin{remark}\label{R.SV.pi_1_subdirect}
	The condition $ \pi^{n}_1(A) = A_1 $ is necessary for the equivalence of (i) and  (ii) in Proposition \ref{P.SV.goursat_fin_many} to be true. For example, let $ V $ be a non-trivial valuation ring, and define $ A_1:= \text{qf}(V) $, $ A_2 := V $, and $ A := V \times_{V/\mathfrak{m}_V} V $; then $ \pi^2_1(A) =V \neq A_1 $ and  $ A  $ is a subring of  $ A_1 \times A_2 $ such that $ A = \pi^2_1(A) \times_{A_2/J_2} A_2$ where $ I_{1} = J_2:= \mathfrak{m}_V $, but $ A $ it is not a subdirect product of $ \{ A_1, A_2\}  $.
\end{remark}

\begin{theorem}\label{T.SV.equiv_red_loc_SV-ring_finite_rank}
Let $ n \in  \N^{\geq 2}$ and $ A $ be a ring which is not a field. The following are equivalent:
\begin{enumerate}[\normalfont(i)]
	\item $ A $ is a reduced local  SV-ring of rank  at most $ n $.
	\item $ A $ is a local ring, and there exist non-trivial valuation rings $ A_1, \dots, A_{n}  $ and an injective ring homomorphism $ \epsilon: A \linto \prod_{i=1}^n A_i $ such that $ \epsilon(A)  $ is  a subdirect product of $ \{ A_1, \dots, A_{n} \}   $. 
\item There exist non-trivial valuation rings $ A_1, \dots, A_n $ and an injective ring homomorphism $ \epsilon: A \linto   \prod_{i=1}^n A_i  $, such that  $ (\pi^n_1 \circ \epsilon)(A)= A_1 $ \emph{(}see \emph{Notation \ref{N.SV.pimj}}\emph{)}, and for all $ j \in [n-1]  $ there exist ideals $I_j\subsetneq (\pi^n_j \circ \epsilon)(A) $ and  $  J_{j+1}  \subsetneq A_{j+1} $ such that $ (\pi^n_j \circ \epsilon)(A)/I_j\cong A_{j+1}/J_{j+1}  $  and $$(\pi^n_{j+1} \circ \epsilon)(A) =(\pi^n_j \circ \epsilon)(A)\times_{A_{j+1}/J_{j+1}} A_{j+1};$$in particular,  $$\hspace{1.3cm}  A \cong (\pi^n_n \circ \epsilon)(A) = (((A_1 \times_{A_2/J_2} A_2) \times_{A_3/J_3}A_3) \dots \times_{A_{n-1}/J_{n-1} } A_{n-1})\times_{A_{n}/J_{n}}  A_{n} .$$ 
\item There exist  prime ideals $ \mathfrak{p}_1, \dots, \mathfrak{p}_{n} \subseteq A $ such that  $  \bigcap_{i =1}^n \mathfrak{p}_i  = (0) $, $ G_j := (\bigcap_{i=1}^{j-1}\mathfrak{p}_i)+ \mathfrak{p}_{j} \neq A$ for all $ j \in \{2, \dots, n\}  $, $  A/\mathfrak{p}_i $ is a non-trivial valuation ring for all $ i\in [n]  $,  and the canonical ring homomorphism $ A \lra \prod_{i=1}^n A/\mathfrak{p}_i  $ given by $  a \mapsto (a/\mathfrak{p}_1, \dots, a/\mathfrak{p}_n) $ restricts to an isomorphism $$ \hspace{1.2cm} A \cong \left(\left(\left(\frac{A}{\mathfrak{p}_1} \times_{\frac{A}{G_2}}\frac{A}{\mathfrak{p}_2}\right) \times_{\frac{A}{G_3}}\frac{A}{\mathfrak{p}_3}\right) \dots \times_{\frac{A}{G_{n-1}}}\frac{A}{\mathfrak{p}_{n-1}}\right)\times_{\frac{A}{G_{n}}}\frac{A}{\mathfrak{p}_{n}} .$$ 
	 \end{enumerate} 
	 Moreover, if any of the items \emph{(i) - (iv)} holds, then  the following are equivalent:
\begin{enumerate}[\normalfont(a)]
\item 	 $ A $ has rank exactly $ n $.
\item For all $ S \subsetneq [n]  $, the map $ \pi_S \circ \epsilon: A \lonto \prod_{i \in S}A_i  $ is not injective, where $ A_i $ $ (i \in [n]) $ and $ \epsilon $ are as in items \emph{(ii)} and \emph{(iii)}, and $ \pi_S : \prod_{i =1}^n A_i \lonto \prod_{i \in S}A_i $ is the canonical projection.
\item $ \mathfrak{p}_i $ and $ \mathfrak{p}_j $ are incomparable under subset inclusion for all $ i, j \in [n] $ with $ i \neq  j $, where $ \mathfrak{p}_i $ are the prime ideals of $ A $ in item \emph{(iv)}.
\end{enumerate}
	  \end{theorem}

\begin{proof}
	The equivalence of items (ii) - (iv) follows from Proposition \ref{P.SV.goursat_fin_many}; it therefore remains to show the equivalence of items (i) and  (ii), as well as the equivalence of (a) - (c).

	\underline{(i) $ \Ra $  (ii).} Since $ A $ is local with rank at most $ n $,  $ \text{Spec}^{\text{min}}(A) = \{\mathfrak{p}_1, \dots, \mathfrak{p}_m\} $ for some $ m \in [n]  $, and since $ A $ is reduced, $ \text{Nil}(A) = \bigcap_{\mathfrak{p} \in \text{Spec}(A)}\mathfrak{p} =  \bigcap_{\mathfrak{p} \in \text{Spec}^{\text{min}}(A)}\mathfrak{p} = (0)$, and thus the canonical map $ A \lra \prod_{i = 1}^m A/\mathfrak{p}_i $  is injective. Since $ A $ is not a field and a local SV-ring, each factor ring $ A_i:= A/\mathfrak{p}_i $ is a non-trivial valuation domain, therefore (ii) follows by taking $ A_{m+1}= \dots = A_n := A_{m} $ and $ \epsilon: A \linto \prod_{i=1}^n A_i $ to be the canonical map $ a \mapsto (a/\mathfrak{p}_1, \dots,a/\mathfrak{p}_m, a/\mathfrak{p}_m, \dots,  a/\mathfrak{p}_m) $.

	\underline{(ii) $ \Ra $  (i).} Since each $ A_i $ is a domain and $ A  $ is isomorphic to the subring  $ \epsilon(A) \subseteq \prod_{i =1}^n A_i $, $ A  $ is reduced, so it remains to show that $ A $  is an SV-ring of rank at most $ n $. Since  $ A $ is local, it suffices to show by the implication (ii) $ \Ra $  (i) in  Theorem \ref{T.SV.equiv_SV-ring} that $ A $ has at most $ n $ minimal prime ideals $ \mathfrak{p}_1, \dots, \mathfrak{p}_n $, and that $ A/\mathfrak{p}_i $ is a  non-trivial valuation ring for all $ i \in [n]  $. For each $ i\in [n]  $, define $ \mathfrak{p}_i:= \text{ker}(\pi_i \circ \epsilon) $; then $ \mathfrak{p}_1, \dots, \mathfrak{p}_n $ are  prime ideals of $ A $ such that each $ A/\mathfrak{p}_i $ is a non-trivial valuation ring, and it follows by Lemma  \ref{L.SV.subdirect_min_prime_id.i}   that $ \text{Spec}^{\text{min}}(A) \subseteq \{\mathfrak{p}_1, \dots, \mathfrak{p}_n\}  $, as required. 

If any of the items (i) - (iv) holds, then the equivalence of items (a) - (c) follows by Lemma \ref{L.SV.subdirect_min_prime_id.ii}.
\end{proof}

\begin{remark}\label{R.SV.semi-local}
	Recall that a ring is \textit{semi-local} if it has finitely many maximal ideals; it is not difficult to see from the proofs of Proposition \ref{P.SV.goursat_fin_many} and of Theorem \ref{T.SV.equiv_red_loc_SV-ring_finite_rank} that the class of rings which are isomorphic to subdirect products of finitely many local domains (resp., valuation rings) is exactly the class of reduced semi-local rings (resp., reduced semi-local SV-rings) of finite rank.
\end{remark}

%% file: real_closed_SV-rings.tex
\subsection{Preliminaries on real closed rings}\label{SUBSEC.prelim_rcr}

Real closed rings can be defined in many equivalent ways; the following definition is the best suited for the purposes of this note (cf. \cite[1]{prestel.schwartz/mod_th_rcr}):

\begin{definition}\label{D.RCSVR.rcr}
	A ring $ A $ is a \textit{real closed ring} if  it satisfies the following conditions:
	\begin{enumerate}[\normalfont(i), ref=\ref{D.RCSVR.rcr} (\roman*)]
	\item $ A $ is reduced;
	\item the set of squares of $ A $ is the set of non-negative elements of a partial order $ \leq $ on  $ A $ and  $ (A, \leq) $ is an $ f $\textit{-ring}, i.e., $ (A, \leq)  $ is a partially ordered ring such that for every $ a, b \in A $ the supremum  $ a \vee  b$ and the infimum $ a \wedge b $ exist in $ A $, and $ c(a \wedge b) = ca \wedge cb $ for all  $ c \geq 0 $;
	\item\label{D.RCSVR.rcr.iii} for all $ a, b \in A $, if $ 0 \leq a \leq b $, then there exists $ c \in A $ such that $ bc = a^2 $; and
	\item  $ \text{qf}(A/\mathfrak{p}) $ is a real closed field and $ A/\mathfrak{p} $ is integrally closed for all $ \mathfrak{p} \in \text{Spec}(A) $.
\end{enumerate}
\end{definition}

It is clear from Definition \ref{D.RCSVR.rcr} that a field is a real closed ring if and only if it is a real closed field; arbitrary convex subrings of real closed fields are also real closed rings (Theorem \ref{T.RCSVR.equiv_rcvr}). The next theorem summarizes the main properties of real closed rings that will be made use of in the rest of the note:

\begin{theorem}\label{T.RCSVR.properties_rcr}
	\begin{enumerate}[\normalfont(I), ref=\ref{T.RCSVR.properties_rcr} (\Roman*)]
	\item\label{T.RCSVR.properties_rcr.I} 	The category of real closed rings together with ring homomorphisms is complete and cocomplete; in particular, direct and fibre products of real closed rings are real closed.
	\item\label{T.RCSVR.properties_rcr.II} Let $ A  $ be a real closed ring.
		\begin{enumerate}[\normalfont(i), ref=\ref{T.RCSVR.properties_rcr.II} (\roman*)]
		\item\label{T.RCSVR.properties_rcr.II.i} If $I \subseteq A $ is an ideal, then  $ A/I $ is real closed if and only if $ I  $ is radical.
		\item\label{T.RCSVR.properties_rcr.II.ii} If $ S\subseteq A$ is a multiplicative subset, then the localization $ S^{-1}A $ is real closed.
		\item\label{T.RCSVR.properties_rcr.II.iii} The poset $(\emph{Spec}(A), \subseteq) $ is a \emph{root system}, i.e., for all $ \mathfrak{p} \in \emph{Spec}(A) $, the principal up-set $ \mathfrak{p}^{\uparrow}:= \{ \mathfrak{q} \in \emph{Spec}(A) \mid \mathfrak{p} \subseteq \mathfrak{q} \}  $ is a chain.
		\item\label{T.RCSVR.properties_rcr.II.iv} If $ I, J \subseteq A$ are radical ideals, then $ I+J $ is a radical ideal. In particular:
			\begin{enumerate}[\normalfont(a), ref=\ref{T.RCSVR.properties_rcr.II.iv} (\alph*)]
				\item\label{T.RCSVR.properties_rcr.II.iv.a} If $ \mathfrak{p}, \mathfrak{q} \in \emph{Spec}(A) $ and $ 1 \notin \mathfrak{p} + \mathfrak{q} $, then $ \mathfrak{p} + \mathfrak{q} \in \emph{Spec}(A) $.
			\item\label{T.RCSVR.properties_rcr.II.iv.b} 	 The poset $ (\mathcal{I}^{\emph{rad}}(A), \subseteq) $ of radical ideals of $ A $ is a distributive lattice with join and meet operations given by sum and intersection of ideals, respectively.
				\end{enumerate}
		\end{enumerate}
	\end{enumerate}
\end{theorem}

\begin{proof}
	(I). By \cite[Section 12]{schwartz/madden.safr}, the category of real closed rings is a monoreflective subcategory of the category of reduced partially ordered rings; since the latter category is complete and cocomplete by \cite[Theorem 1.7]{schwartz/madden.safr}, so is the category of real closed rings by \cite[Proposition 2.3]{schwartz/madden.safr} and \cite[Proposition 2.7]{schwartz/madden.safr}.
	
	(II). (i)  is \cite[Theorem 4.5]{schwartz.basic}, (ii) is \cite[Proposition 12.6]{schwartz/madden.safr}, (iii) follows from \cite[Theorem 3.10]{schwartz.basic} and \cite[Theorem 13.1.9. (iii)]{dickmann/schwartz/tressl.specbook}, and (iv) is \cite[Corollary 15]{schwartz.rcr}. For (iv) (a), pick $ \mathfrak{p}, \mathfrak{q} \in \text{Spec}(A)$ such that $ 1 \notin \mathfrak{p} + \mathfrak{q} $. Since $ \mathfrak{p}+\mathfrak{q} $ is a radical ideal, $ \mathfrak{p} + \mathfrak{q}= \bigcap \{ \mathfrak{r} \in \text{Spec}(A)\mid \mathfrak{p} + \mathfrak{q} \subseteq \mathfrak{r}\}  $, and since $ 1 \notin \mathfrak{p} +\mathfrak{q} $, $ \mathcal{S}:= \{ \mathfrak{r} \in \text{Spec}(A)\mid \mathfrak{p} + \mathfrak{q} \subseteq \mathfrak{r}\}  \neq \emptyset $. If $ \mathfrak{r}_1, \mathfrak{r}_2\in \mathcal{S} $, then $ \mathfrak{p} \subseteq \mathfrak{r}_1, \mathfrak{r}_2 $, therefore $ \mathfrak{r}_1 $ and $ \mathfrak{r}_2 $ are comparable under subset inclusion  by (iii); it follows that $ \mathcal{S} $ is a chain in $ (\text{Spec}(A), \subseteq) $, therefore $ \mathfrak{p}+\mathfrak{q} = \bigcap_{\mathfrak{r}\in \mathcal{S}}\mathfrak{r} \in \text{Spec}(A)$. For (iv) (b) it remains to show that the lattice $  (\mathcal{I}^{\text{rad}}(A), +, \cap)  $ is distributive, so pick $ I, J, K \in \mathcal{I}^{\text{rad}}(A) $; the inclusion $ (I \cap J) + (I\cap K) \subseteq I \cap(J + K) $ is clear, and if $ a \in  I \cap(J + K)$, then $ a= j+k  $ for some  $ j \in J $ and  $ k \in K $, therefore  $ a^2 = aj + ak \in IJ+IK \subseteq (I \cap J) + (I\cap K)  $, and since $ (I \cap J) + (I\cap K)  \in \mathcal{I}^{\text{rad}}(A) $, $ a \in  (I \cap J) + (I\cap K)$ follows, as required.
\end{proof}

\subsubsection{Real closed valuation rings}

\begin{definition}\label{D.RCSVR.rcvr}
A ring $ A $ is a \textit{real closed valuation ring} if $ A $ is a real closed ring and a valuation ring.	
\end{definition}

\begin{theorem}\label{T.RCSVR.equiv_rcvr}
Let $ A$ be a ring. The following are equivalent:	
\begin{enumerate}[\normalfont(I)]
\item $ A $ is a real closed valuation ring.
\item $ \emph{qf}(A) $ is a real closed field and $ A $ is convex in $ \emph{qf}(A) $.
\item $ \emph{qf}(A) $ is a real closed field and $ A = \{ a \in \emph{qf}(A) \mid v(a) \geq 0\}  $, where  $ v: \emph{qf}(A)\lonto \Gamma$ is an \emph{order-compatible valuation} on $ \emph{qf}(A) $ \emph{(}i.e., for all $ a, b \in A $, $ 0 \leq a \leq b $ implies $ v(b) \leq v(a) $ \emph{)}.
	\item  $ A $ is a valuation ring, and both $ \emph{qf}(A) $  and $ A/\mathfrak{m}_A $ is are real closed fields.
\item $ A $ is a totally ordered domain which satisfies the intermediate value property for polynomials in one variable.
\item $ A $ is a totally ordered domain which satisfies the following conditions:
	\begin{enumerate}[\normalfont(i)]
	\item For all $ a, b \in A $, if $ 0 < a < b $, then  there exists $ c \in A $ such that $ bc=a $.
	\item Every positive element has a square root.
	\item Every monic polynomial of odd degree has a root.
	\end{enumerate}
\item $ A $ is a local real closed SV-ring of rank $ 1 $.
\end{enumerate}
\end{theorem}

\begin{proof}
	\underline{(I) $ \Ra $ (II).} It suffices to show that $ A  $ is convex in $ \text{qf}(A) $, so pick $ a \in A $ and  $  b \in \text{qf}(A)$ such that $ 0 < b < a $ and assume for contradiction that $ b \notin A $; then  $ b^{-1}\in A $, therefore $ 0 < 1< ab^{-1} $ in  $ A $ implies that there exists $ c \in A $ such that $ ab^{-1}c = 1^2=1  $ (Definition \ref{D.RCSVR.rcr.iii}), hence $ b= ac \in A $, a contradiction.

	\underline{(II) $ \Leftrightarrow $ (IV).} By \cite[Proposition 2.2.4]{knebusch/scheiderer.real_algebra} and  \cite[Theorem 2.5.1 (b)]{knebusch/scheiderer.real_algebra}.

	\underline{(IV) $ \Ra $  (I).} Pick $ \mathfrak{p} \in \text{Spec}(A) $. $ A $ is convex in $ \text{qf}(A) $ by the implication (IV) $ \Ra $  (II), therefore $ A_{\mathfrak{p}} $ is also convex in $ \text{qf}(A) $; in particular, $ A_{\mathfrak{p}}/\mathfrak{p}A_{\mathfrak{p}} $ is a real closed field by the implication (II) $ \Ra $ (IV), therefore $ A/\mathfrak{p} $ is a valuation ring (hence integrally closed) of the real closed field $ \text{qf}(A/\mathfrak{p}) = A_{\mathfrak{p}}/\mathfrak{p}A_{\mathfrak{p}}$, and thus it remains to show that the condition in item (iii) of Definition \ref{D.RCSVR.rcr} is satisfied. Pick $ a, b \in A $ such that  $ 0 \leq a \leq b $; then  $0 \leq a^2/b \leq b^2/b=b\in A $, and since  $ A $ is convex in $ \text{qf}(A) $, it follows that $ c:=a^2/b \in A $ is such that  $ bc=a^2$, as required.

	\underline{(II) $ \Leftrightarrow $  (III) $ \Leftrightarrow$ (V) $ \Leftrightarrow $ (VI)} See \cite[Theorem 1 and Lemma 4]{cherlin/dickmann.rcrII}.

	\underline{(I) $ \Leftrightarrow $  (VII).} By Lemma \ref{L.SV.val_ring_equiv}.
\end{proof}

\begin{example}\label{E.RCSVR.rcvr_germs}
	Let $ R $ be a real closed field and $ A  $ be the ring of continuous semi-algebraic functions $ R \lra R $. Pick  $ r \in R $ and define  $ \mathfrak{p}_{r^{-}}:= \{ f \in A \mid \exists \epsilon \in R^{>0}  \text{ such that } f_{\upharpoonright (r-\epsilon, r]} = 0\}  $; it is clear that $ \mathfrak{p}_{r^{-}} \in \text{Spec}(A) $, and it is claimed that $ A/\mathfrak{p}_{r^{-}} $ is a non-trivial real closed valuation ring. By the implication (II) $ \Ra $ (I) in Theorem \ref{T.RCSVR.equiv_rcvr} it suffices to show that $ A/\mathfrak{p}_{r^{-}} $ is isomorphic to a proper convex subring of a real closed field; to this end, order the function field $ R(t) $  by setting $ t $ to be a positive infinitesimal over $ R $, let  $ S $ be the real closure of $ R(t) $, and let $ B $ be the ring of continuous semi-algebraic functions $ S\lra S $. Every $ f \in A $ defines a unique  $ f_{S} \in B $ in the obvious way, and this yields a composite $ R $-algebra  homomorphism $ F: A \into B \onto S $ given by $ f \mapsto f_S(r-t) $; then $ \text{ker}(F) = \mathfrak{p}_{r^-} $ and the image of $ F $ is the convex hull of $ R $ in $ S $.
\end{example}

Therefore by Theorem \ref{T.RCSVR.properties_rcr.II.i} a real closed SV-ring is a real closed ring $ A $ such that $ A/\mathfrak{p} $ is a real closed valuation ring for all $ \mathfrak{p} \in \text{Spec}(A)  $. A particular class of real closed valuation rings  are those valuation  rings corresponding to the canonical valuation of real closed Hahn series fields:

\begin{definition}\label{D.RCSVR.hahn_field}
Let $ \emph{\textbf{k}} $ be a field and $ \Gamma $ be a totally ordered abelian group. Define $ \emph{\textbf{k}}((\Gamma)):= \emph{\textbf{k}}((x^{\Gamma})) $ to be the set of formal series $ a =\sum a_{\gamma}x^{\gamma} := \sum_{\gamma \in \Gamma} a_{\gamma}x^{\gamma} $ where $ \text{supp}(a) := \{\gamma \in \Gamma \mid a_{\gamma} \neq 0\} $ is a well-ordered subset of  $ \Gamma $.
\end{definition}

\begin{theorem}\label{T.RCSVR.Hahn}
	Let $ \textbf{k} $ be a field and $ \Gamma $ be a totally ordered abelian group.
	\begin{enumerate}[\normalfont(i)]
		\item The set $ \textbf{k}((\Gamma)) $ endowed with the operations of pointwise addition and Cauchy product of formal series 
		\[
		\sum a_{\gamma}x^{\gamma}+\sum b_{\gamma}x^{\gamma} := \sum (a_{\gamma} + b_{\gamma})x^{\gamma} 
		\] and 
	\[
			 \left(\sum a_{\gamma}x^{\gamma}\right)\left(\sum b_{\gamma}x^{\gamma} \right):= \sum_{\gamma \in \Gamma} \sum_{\delta + \epsilon= \gamma} (a_{\delta}b_{\epsilon})x^{\gamma}
	,\] 
	respectively, is a field called  \emph{Hahn series field}.
\item The map $ \nu: \textbf{k}((\Gamma)) \lonto \Gamma $ given by $ \nu(a) := \emph{min}(\emph{supp}(a))$ is a valuation with residue field $ \textbf{k} $; write $ \textbf{k}[[\Gamma]]:=  \textbf{k}((\Gamma^{\geq 0}))$ for its corresponding valuation ring.
\item If $ \textbf{k} $ is a totally ordered field, then $ \textbf{k}((\Gamma)) $ has the structure of a totally ordered field by setting $ a >0 $ if and only if $a_{\nu(a)} >0$ for all $ a \in \textbf{k}((\Gamma)) $; under this total order, $ \nu $ is an order-compatible valuation on $ \textbf{k}((\Gamma)) $. 

\item $ \textbf{k}((\Gamma)) $ is a real closed field if and only if $ \textbf{k} $ is real closed and $ \Gamma $ is  divisible; in particular, $ \textbf{k}[[\Gamma]] $ is a real closed valuation ring if and only if $ \textbf{k} $ is real closed and $ \Gamma $ is  divisible.
	\end{enumerate}
\end{theorem}

\begin{proof}
	Folklore, see for instance \cite[Exercise 3.5.6]{engler/prestel.valued_fields}, \cite[Section 3.5]{aschenbrenner/vdD/vdHoeven.asympt}, or \cite[Section 2]{dales/woodin.super}; the last statement in (iv) follows from the equivalence (I) $ \Leftrightarrow $  (III) in Theorem \ref{T.RCSVR.properties_rcr}.
\end{proof}

It is well-known that if $ V $ is a non-trivial real closed valuation ring, then there exists a local embedding $ V \linto \textbf{\emph{k}}[[\Gamma]] $, where $ \textbf{\emph{k}}:= V/\mathfrak{m}_V $ and $ \Gamma := \text{qf}(V)^{\times}/V^{\times}$ (\cite[62, Satz 21]{priess-crampe.ord_str} and Theorem \ref{T.APP.rcvf_hahn_embedd}), and $ \textbf{\emph{k}}[[\Gamma]] $ is a real closed valuation ring by the implication (I) $ \Ra $  (IV) in Theorem \ref{T.RCSVR.equiv_rcvr}; in fact, one can do slightly better:

\begin{proposition}\label{P.RCSVR.rcvr_embedd_nice}
	Let $ V\subseteq W $ be a local embedding of non-trivial real closed valuation rings, and set $ \textbf{k}:= V/\mathfrak{m}_V $, $ \textbf{l}:= W/\mathfrak{m}_W $, $ \Gamma:= \emph{qf}(V)^{\times} /V^{\times}$, and $ \Delta := \emph{qf}(W)^{\times}/W^{\times} $. There exist  local embeddings of non-trivial real closed valuation rings $ \epsilon_V: V \linto \textbf{k}[[\Gamma]] $ and $ \epsilon_W: W \linto \textbf{l}[[\Delta]] $ such that $ \epsilon_{W\upharpoonright V} = \epsilon_V $.
\end{proposition}

\begin{proof}
	Immediate from Theorem \ref{T.APP.rcvf_hahn_embedd_II}, since $ V \subseteq W $ being a local embedding implies that  $ (\text{qf}(V) ,V) \subseteq  (\text{qf}(W) ,W)$ is an embedding of real closed valued fields. 	
\end{proof}

Conclude this subsection with two a construction which gives rise to  non-trivial real closed valuation rings.

\begin{lemma}\label{lem.rcr_domain_not_rcvr}
	Let $ V  $ be a non-trivial real closed valuation ring, $ \lambda: V\lonto V/\mathfrak{m}_V=: \textbf{k} $ be the residue map, $ B \subseteq \textbf{k} $ be a subring. The ring $ A:= \lambda^{-1}(B) \subseteq V  $ is a real closed valuation ring if and only of $ B $ is a real closed valuation ring.
\end{lemma}

\begin{proof}
	Straightforward from Theorem \ref{T.RCSVR.properties_rcr.I} and the equivalence (I) $ \Leftrightarrow $  (II) in Theorem \ref{T.RCSVR.equiv_rcvr}, using also the fact that $\lambda $ is an order-preserving map and that $ \lambda^{-1}(B) \cong V \times_{\emph{\textbf{k}}} B $.
\end{proof}

\subsection{Branching ideals in local real closed rings of finite rank}\label{SUBSEC.branch}

In this subsection the notion of a branching ideal in a ring is introduced (Definition \ref{D.RCSVR.branch_id}) and several equivalent characterizations are given for the maximal ideal of a local real closed ring of finite rank to be a branching ideal (Proposition \ref{P.RCSVR.equiv_branching_max_id}); branching ideals are used to define rings of type $ (n,1) $ and of type $ (n,2) $ for $ n \in \N^{\geq 2} $ (Definition \ref{D.RCSVR.type}), whose model theory is studied in Section \ref{SEC.mod_th}. Recall throughout this subsection that if $ A  $ is a local real closed ring, then $ \mathfrak{p} + \mathfrak{q} \in \text{Spec}(A) $ for all $ \mathfrak{p}, \mathfrak{q} \in \text{Spec}(A) $, see Theorem \ref{T.RCSVR.properties_rcr.II.iv.a}.

\begin{definition}\label{D.RCSVR.branch_id}
	Let $ A  $ be a  ring. A prime ideal $ \mathfrak{q}  $ of $ A $	 is a \textit{branching ideal} if there exist  distinct $ \mathfrak{q}_1, \mathfrak{q}_2 \in \text{Spec}(A) $ such that $ \mathfrak{q}_1, \mathfrak{q}_2 \subsetneq \mathfrak{q} $ and $ \mathfrak{q} = \mathfrak{q}_1 +\mathfrak{q}_2 $.
\end{definition}

\begin{remark}\label{R.RCSVR.branch_0}
Let $ A $ be a real closed ring.
\begin{enumerate}[\normalfont(i), ref=\ref{R.RCSVR.branch_0} (\roman*)]
		\item 	If $ \mathfrak{q}\in\text{Spec}(A) $ is a branching ideal, then $ \text{rk}(A, \mathfrak{q}) \geq 2$ by Theorem \ref{T.RCSVR.properties_rcr.II.iii}, but the converse doesn't hold: consider the maximal ideal of $A:= V\times_{V/\mathfrak{p}} V $, where $ V  $ is a non-trivial real closed valuation ring of Krull dimension at least $ 2 $ and $ \mathfrak{p}  $ is a non-zero non-maximal prime ideal of $ V $, noting that $ A  $ is a real closed ring by items (I) and (II) (i) in Theorem \ref{T.RCSVR.properties_rcr}.
		\item\label{R.RCSVR.branch_0.ii}  $ A $ has at least one branching ideal if and only if $ \text{rk}(A) \geq 2 $. One implication is clear; conversely, if $ \mathfrak{q}_1, \mathfrak{q}_2 \in \text{Spec}(A)$ witness that $  \mathfrak{q} \in \text{Spec}(A) $ is a branching ideal, then any  $ \mathfrak{p}_1, \mathfrak{p}_2 \in \text{Spec}^{\text{min}}(A) $ such that $ \mathfrak{p}_1 \subseteq \mathfrak{q}_1 $ and $  \mathfrak{p}_2 \subseteq \mathfrak{q}_2 $ are distinct by Theorem \ref{T.RCSVR.properties_rcr.II.iii}, therefore $ \text{rk}(A) \geq \text{rk}(A, \mathfrak{q}) \geq 2 $. 
	\end{enumerate}
\end{remark}

\begin{lemma}\label{lemilla}
	Let $ A  $ be a ring    and let $ \mathfrak{q}_1, \mathfrak{q}_2  \in \emph{Spec}(A)$ be incomparable prime ideals in $ (\emph{Spec}(A),\subseteq) $. If $ \mathfrak{p}_1 \in \mathfrak{q}_1^{\downarrow}$ and $ \mathfrak{p}_2\in \mathfrak{q}_2^{\downarrow} $ are such that $ \mathfrak{p}_1^{\uparrow} $ and $ \mathfrak{p}_2^{\uparrow} $ are chains in $ (\emph{Spec}(A), \subseteq) $ and $ \mathfrak{p}_1 + \mathfrak{p}_2 \in \emph{Spec}(A) $, then $ \mathfrak{p}_1+\mathfrak{p}_2 = \mathfrak{q}_1 +\mathfrak{q}_2 $.
\end{lemma} 

\begin{proof}
	  The inclusion $ \mathfrak{p}_1 + \mathfrak{p}_2 \subseteq \mathfrak{q}_1 +\mathfrak{q}_2 $ is clear, and for the other inclusion it suffices to show that $ \mathfrak{q}_1  \subseteq \mathfrak{p}_1 +\mathfrak{p}_2$ and $ \mathfrak{q}_2 \subseteq \mathfrak{p}_1+\mathfrak{p}_2 $; in turn, it follows by the assumptions on $ \mathfrak{p}_1  $ and $ \mathfrak{p}_2 $ that it suffices to show that $ \mathfrak{p}_1+\mathfrak{p}_2 \not\subseteq \mathfrak{q}_1$ and $ \mathfrak{p}_1+\mathfrak{p}_2 \not\subseteq \mathfrak{q}_2 $. Assume for contradiction that $ \mathfrak{p}_1+\mathfrak{p}_2 \subseteq \mathfrak{q}_1$; then $ \mathfrak{q}_1$ and $\mathfrak{q}_2  $ are incomparable prime ideals in $ \mathfrak{p}_2^{\uparrow} $, contradicting that $ \mathfrak{p}_2^{\uparrow} $ is a chain, therefore $ \mathfrak{q}_1 \subseteq \mathfrak{p}_1+\mathfrak{p}_2 $. The proof of $ \mathfrak{q}_2 \subseteq \mathfrak{p}_1+\mathfrak{p}_2 $ is analogous.
\end{proof}

\begin{lemma}\label{L.RCSVR.finite_sum_prim_id_is_sum_of_two_prim_id}
	Let $ A $ be a local real closed ring and $ \mathfrak{p}_1, \dots, \mathfrak{p}_n \in \emph{Spec}(A) $  $ (n \in \N^{\geq 2}) $ be pairwise incomparable under subset inclusion. For all $ i \in [n] $  there exists  $  j \in [n]\setminus \{ i \}  $ such that $ \sum_{k=1}^n\mathfrak{p}_k = \mathfrak{p}_i + \mathfrak{p}_j $.	
\end{lemma}

\begin{proof}
	The proof is by induction on $ n \in \N^{\geq 2}$. The base case is clear, so assume that  the statement holds for some $ n \in \N^{> 2} $  and let $ \mathfrak{p}_1, \dots, \mathfrak{p}_{n+1} \in \text{Spec}(A) $ be pairwise incomparable under subset inclusion. Let $ i \in [n+1] $ be arbitrary and pick $ j_0 \in [n+1] \setminus \{ i \}  $.  By inductive hypothesis, there exists   $  j \in [n+1]\setminus \{ i, j_0 \}  $ such that  $  \sum_{k \in [n+1]\setminus \{ i \} }\mathfrak{p}_k = \mathfrak{p}_{j_0} + \mathfrak{p}_j  $, hence  $  \sum_{k=1}^{n+1}\mathfrak{p}_k = \mathfrak{p}_i + \mathfrak{p}_{j_0} +  \mathfrak{p}_{j} $; since $ A  $ is a local real closed ring,   either $ \mathfrak{p}_i + \mathfrak{p}_{j}\subseteq \mathfrak{p}_i + \mathfrak{p}_{j_0}  $ or  $ \mathfrak{p}_i + \mathfrak{p}_{j_0} \subseteq \mathfrak{p}_i + \mathfrak{p}_{j}  $ by items (iii) and (iv) (a) in Theorem \ref{T.RCSVR.properties_rcr}, hence either $  \sum_{k=1}^{n+1}\mathfrak{p}_k = \mathfrak{p}_i + \mathfrak{p}_{j_0} $ or $  \sum_{k=1}^{n+1}\mathfrak{p}_k = \mathfrak{p}_i +  \mathfrak{p}_{j} $, as required.
\end{proof}

\begin{proposition}\label{P.RCSVR.equiv_branching_max_id}
Let $ A $ be a local real closed ring of rank $ n \in \N^{\geq 2} $. The following are equivalent:	
\begin{enumerate}[\normalfont(i)]
\item $ \mathfrak{m}_A $ is a branching ideal.
\item For all $ \mathfrak{q}_1 \in \emph{Spec}(A)  $ there exists $ \mathfrak{q}_2 \in \emph{Spec}(A)\setminus \{ \mathfrak{q}_1 \}  $ such that  $ \mathfrak{q}_1, \mathfrak{q}_2 \subsetneq \mathfrak{m}_A $ and $ \mathfrak{m}_A=\mathfrak{q}_1 +\mathfrak{q}_2 $.
\item There exist distinct $ \mathfrak{p}_1, \mathfrak{p}_2 \in \emph{Spec}^{\emph{min}}(A) $   such that  $ \mathfrak{m}_A = \mathfrak{p}_1 +\mathfrak{p}_2$.
\item For all $ \mathfrak{p}_1 \in \emph{Spec}^{\text{\emph{min}}}(A) $ there exists $ \mathfrak{p}_2 \in  \emph{Spec}^{\text{\emph{min}}}(A)\setminus \{ \mathfrak{p}_1 \}  $ such that  $  \mathfrak{m}_A = \mathfrak{p}_1 + \mathfrak{p}_2 $.
\item Every non-unit of $ A $ is a sum of two zero divisors.
\item There exists a partition $ S_1 \ \dot{\cup} \ S_2 = \emph{Spec}^{\emph{min}}(A) $ such that $( \bigcap_{\mathfrak{p} \in S_1}\mathfrak{p}) + ( \bigcap_{\mathfrak{p} \in S_2}\mathfrak{p}) = \mathfrak{m}_A$.
\item There exist $ r, s \in [n] $ and local real closed rings  $ A_1 $ and $ A_2$ with isomorphic residue field $ K $ such that $ r+s=n $,  $ \emph{rk}(A_1)=r $, $ \emph{rk}(A_2)=s $, and $ A \cong A_1 \times_K A_2 $.
\end{enumerate}
\end{proposition}

\begin{proof} 
	Since $ A $ is local of rank $ n \in \N^{\geq 2} $, $ \text{Spec}^{\text{min}}(A) := \{ \mathfrak{p}_1, \dots, \mathfrak{p}_n  \}  $. 

	\underline{(i) $ \Leftrightarrow $  (ii).} One implication is clear, so suppose that (i) holds and assume for contradiction that there exists $\mathfrak{q} \in \text{Spec}(A)$ such that $ \mathfrak{q}+\mathfrak{q}' \subsetneq \mathfrak{m}_A  $ for all $ \mathfrak{q}'\in \text{Spec}(A) $ with $ \mathfrak{q}\neq \mathfrak{q}' $. By  (i), there exist distinct $ \mathfrak{q}_1, \mathfrak{q}_2\in\text{Spec}(A) $ such that $ \mathfrak{q}_1, \mathfrak{q}_2 \subsetneq \mathfrak{m}_A $ and $ \mathfrak{m}_A = \mathfrak{q}_1+ \mathfrak{q}_2 $; in particular, $ \mathfrak{q} \neq \mathfrak{q}_1 $ and $ \mathfrak{q} \neq \mathfrak{q}_2 $. Since $ A $ is a local real closed ring, either $ \mathfrak{q}+\mathfrak{q}_1\subseteq  \mathfrak{q}+\mathfrak{q}_2$ or $ \mathfrak{q}+\mathfrak{q}_2\subseteq  \mathfrak{q}+\mathfrak{q}_1$, and assuming without loss of generality the former, it follows that $ \mathfrak{m}_A = \mathfrak{q}_1+\mathfrak{q}_2 = \mathfrak{q}+ (\mathfrak{q}_1+\mathfrak{q}_2 )\subseteq \mathfrak{q}+\mathfrak{q}_2 \subsetneq \mathfrak{m}_A$, giving the required contradiction.

	\underline{(i) $ \Leftrightarrow $ (iii)} and \underline{(ii) $ \Leftrightarrow $ (iv).} Clear by Lemma \ref{lemilla} noting that if $ \mathfrak{q}, \mathfrak{q}' \in \text{Spec}(A) $ are distinct prime ideals such that $ \mathfrak{q}, \mathfrak{q}' \subsetneq \mathfrak{m}_A $ and $ \mathfrak{q}+\mathfrak{q}' = \mathfrak{m}_A $, then $ \mathfrak{q} $ and $ \mathfrak{q}' $ are incomparable under subset inclusion.

	\underline{(iii) $ \Ra $  (v).} Clear since $ \mathfrak{m}_A$ consists exactly of the non-units of $ A $ and since the set of zero divisors of  $ A $ is exactly $ \bigcup_{i =1}^n \mathfrak{p}_i $ as $ A $ is reduced.

	\underline{(v) $ \Ra $  (iii).} Since $ A $ is a reduced local ring, (v) is equivalent to the  statement that for all $ \epsilon \in \mathfrak{m}_A $ there exist $ i, j \in [n] $ such that $ \epsilon = b_i + b_j $ for some $ b_i \in \mathfrak{p}_i $ and $ b_j \in \mathfrak{p}_j  $, i.e., $ \mathfrak{m}_A \subseteq \bigcup_{i, j \in [n]}\mathfrak{p}_i +\mathfrak{p}_j $; by Lemma \ref{L.RCSVR.finite_sum_prim_id_is_sum_of_two_prim_id}, there exist $ i', j' \in [n] $ with $ i' \neq j' $ such that  $ \sum_{i=1}^n\mathfrak{p}_i = \mathfrak{p}_{i'}+\mathfrak{p}_{j'} $, therefore \[
	\bigcup_{i, j \in [n]}\mathfrak{p}_i +\mathfrak{p}_j  \subseteq  \sum_{i=1}^n\mathfrak{p}_i = \mathfrak{p}_{i'}+\mathfrak{p}_{j'} \subseteq \mathfrak{m}_A	\subseteq \bigcup_{i, j \in [n]}\mathfrak{p}_i +\mathfrak{p}_j
\] implies $ \mathfrak{m}_A =  \mathfrak{p}_{i'}+\mathfrak{p}_{j'} $, as required (note that $ \mathfrak{p}_{i'}, \mathfrak{p}_{j'}\subsetneq \mathfrak{m}_A $ since $ \text{rk}(A)\geq 2 $).

\underline{(iv) $ \Ra $  (vi).} Pick any $ \mathfrak{p}_i \in \text{Spec}^{\text{min}}(A)$ and define $ S_1:= \{\mathfrak{p} \in \text{Spec}^{\text{min}}(A) \mid \mathfrak{p}_i+\mathfrak{p} = \mathfrak{m}_A \}   $ and $ S_2:= \{\mathfrak{p} \in \text{Spec}^{\text{min}}(A) \mid \mathfrak{p}_i+\mathfrak{p} \neq  \mathfrak{m}_A \}   $. Note that $ \mathfrak{p}_i \in S_2 $, and also  $ S_1 \neq \emptyset $  by (iv), therefore $ S_1 \ \dot{\cup} \ S_2  = \text{Spec}^{\text{min}}(A)$ is a partition.

\noindent \textit{Claim.} If $ \mathfrak{p} \in S_1 $ and $ \mathfrak{q}\in S_2 $, then $ \mathfrak{p} + \mathfrak{q} = \mathfrak{m}_A $.

\noindent \textit{Proof of Claim.}  Since $ \mathfrak{q} \subseteq \mathfrak{p}+\mathfrak{q}, \mathfrak{p}_i + \mathfrak{q} $, either $ \mathfrak{p}+\mathfrak{q} \subseteq \mathfrak{p}_i + \mathfrak{q} $ or $ \mathfrak{p}_i +\mathfrak{q} \subseteq \mathfrak{p} + \mathfrak{q} $; in the former case it follows that $  \mathfrak{m}_A = \mathfrak{p}_i +(\mathfrak{p}+\mathfrak{q})\subseteq \mathfrak{p}_i+\mathfrak{q} $, a contradiction to $ \mathfrak{q} \in S_2 $, therefore  $ \mathfrak{p}_i +\mathfrak{q} \subseteq \mathfrak{p} + \mathfrak{q} $, hence $ \mathfrak{p}_i, \mathfrak{p} \subseteq \mathfrak{p}+\mathfrak{q} $, and thus $\mathfrak{m}_A =\mathfrak{p}_i + \mathfrak{p} \subseteq \mathfrak{p}+\mathfrak{q} $ from which the claim follows. \hfill $ \square_{\text{Claim}} $

\noindent Therefore 
\[
\left( \bigcap_{\mathfrak{p} \in S_1}\mathfrak{p}\right) + \left( \bigcap_{\mathfrak{p} \in S_2}\mathfrak{p}\right) \overset{(1)}{=} \bigcap_{\mathfrak{p} \in S_1,  \mathfrak{q}\in S_2}\mathfrak{p}+\mathfrak{q} \overset{(2)}= \mathfrak{m}_A
,\] 
where (1) follows from Theorem  \ref{T.RCSVR.properties_rcr.II.iv.b} and (2) follows from the claim above.

\underline{(vi) $ \Ra $  (vii).} Let $ S_1 \ \dot{\cup} \ S_2 = \text{Spec}^{\text{min}}(A) $ be a partition such that $( \bigcap_{\mathfrak{p} \in S_1}\mathfrak{p}) + ( \bigcap_{\mathfrak{p} \in S_2}\mathfrak{p}) = \mathfrak{m}_A$. Set $ r:= |S_1| $,  $ s:= |S_2|$,  $ A_1:= A/\bigcap_{\mathfrak{p} \in S_1}\mathfrak{p} $, and $ A_2:= A/\bigcap_{\mathfrak{p} \in S_2}\mathfrak{p} $; then $ A_1 $ and $ A_2 $ are local real closed rings of ranks $ r $ and $ s $ (respectively) with isomorphic residue field $ A/\mathfrak{m}_A=:K $ such that $ r+s=n $ and $ A \cong A_1 \times_K A_2 $, where the latter statement follows from Lemma \ref{L.SV.fibre}.

\underline{(vii) $ \Ra $ (i).} Clear by the description of the Zariski spectrum of the fibre product $ A_1\times_K A_2 $ as $ \text{Spec}(A_1\times_K A_2 ) \cong \text{Spec}(A_1)\amalg_{\text{Spec}(K)}\text{Spec}(A_2) $ and noting that neither $ A_1 $ nor $ A_2  $ are fields, see \cite[Section 12.5.7]{dickmann/schwartz/tressl.specbook} (for an algebraic proof, pick $ \mathfrak{r}_i \in \text{Spec}(A_i)\setminus \{ \mathfrak{m}_{A_i} \}  $ and set $ \mathfrak{q}_i:= \text{ker}(A_1\times_K A_2 \onto A_i/\mathfrak{r}_i) \in \text{Spec}(A_1 \times_K A_2) $; then $ \mathfrak{q}_1 = \mathfrak{r}_1 \times \mathfrak{m}_{A_2} $, $ \mathfrak{q}_2 = \mathfrak{m}_{A_1} \times \mathfrak{r}_{2} $, and $ \mathfrak{m}_A= \mathfrak{m}_{A_1}\times  \mathfrak{m}_{A_2}$ as subsets of $ A_1 \times_K  A_2 \subseteq A_1 \times A_2 $, hence $ \mathfrak{m}_A= \mathfrak{q}_1+\mathfrak{q}_2 $ is a branching ideal).
\end{proof}

\begin{remark}\label{R.RCSVR.branch_1}
	Let $ A $ be a real closed ring of finite rank and $ \mathfrak{q} \in \text{Spec}(A) $; clearly $ \mathfrak{q} $ is a branching ideal if and only if $ \mathfrak{m}_{A_{\mathfrak{q}}} = \mathfrak{q}A_{\mathfrak{q}} $ is a branching ideal of  $ A_{\mathfrak{q}} $, therefore Proposition \ref{P.RCSVR.equiv_branching_max_id} can be applied to give equivalent characterizations for an arbitrary prime ideal of $ A $ to be a branching ideal; in particular, $ \mathfrak{q} \in \text{Spec}(A) $ is a branching ideal if and only if there exist distinct $ \mathfrak{p}_1, \mathfrak{p}_2 \in \text{Spec}^{\text{min}}(A)$ such that $ \mathfrak{q}= \mathfrak{p}_1 +\mathfrak{p}_2 $ (this also follows immediately from Lemma \ref{lemilla}).
\end{remark}

\begin{remark}\label{R.RCSVR.branch_2}
	A local real closed ring  $ A $ of rank  $ n \in \N $ can have at most  $ n-1 $ branching ideals; in particular, the set of branching ideals of $ A $ is finite. If $ n=1 $, then  $ A $ is a domain (Lemma \ref{L.SV.ann_intersection_of_min_prime_id.ii.a}), therefore $ A $ has no branching ideals by   Remark \ref{R.RCSVR.branch_0.ii}; the statement for $ n \in \N^{\geq 2} $ follows by induction using Lemma \ref{L.SV.fibre} and Theorem \ref{T.RCSVR.properties_rcr.II.iii}.
\end{remark}

\begin{lemma}\label{L.RCSVR.embedd_loc_rcr_max_id_branch_implies_local}
	Let $ A $ and $ B $ be local real closed rings of rank $ n \in \N^{\geq 2} $  and $ f: A \linto B $ be an injective ring homomorphism. If $ \mathfrak{m}_A $ is a branching ideal, then $ f $ is a local map.
\end{lemma}

\begin{proof}
	By assumption and by the implication (i) $ \Ra $  (ii) in Proposition \ref{P.RCSVR.equiv_branching_max_id}, there exist two distinct minimal prime ideals $ \mathfrak{p}_1 , \mathfrak{p}_2 \subseteq A$ such that $ \mathfrak{m}_A = \mathfrak{p}_1 +\mathfrak{p}_2 $; by Corollary \ref{C.SV.min_prime_ann}, there exist $ \mathfrak{q}_1, \mathfrak{q}_2 \in \text{Spec}^{\text{min}}(B)  $ such that  $ f^{-1}(\mathfrak{q}_1) = \mathfrak{p}_1$  and $ f^{-1}(\mathfrak{q}_2) = \mathfrak{p}_2$, therefore\[
	\mathfrak{m}_A =\mathfrak{p}_1 + \mathfrak{p}_2 = f^{-1}(\mathfrak{q}_1)+f^{-1}(\mathfrak{q}_2)\subseteq f^{-1}(\mathfrak{q}_1 + \mathfrak{q}_2)\subseteq  f^{-1}(\mathfrak{m}_B) 
,\] hence $ f^{-1}(\mathfrak{m}_B)  = \mathfrak{m}_A$ by maximality of $ \mathfrak{m}_A $.
\end{proof}

\subsection{\texorpdfstring{Rings of type $ (n,1) $ and of type $ (n,2) $}{Rings of type (n,1) and of type (n,2)}}

The starting point of this subsection is a structure theorem for local real closed SV-rings of finite rank (Theorem \ref{T.RCSVR.equiv_loc_real_closed_SV-ring_finite_rank}) which is directly deduced from the structure theorem for reduced local SV-rings of finite rank (Theorem \ref{T.SV.equiv_SV-ring}); first, a lemma:

\begin{lemma}\label{lem.real_closed_fibre_prod} 
Let $ A $ and $ B $  be real closed rings. Suppose that there exist ideals $ I \subsetneq A $ and $ J \subsetneq B $ such that $ A/I \cong B/J $. The fibre product  $ C:= A \times_{B/J}B $ is real closed if and only if $ I $ and $ J $ are radical ideals of $ A $ and $ B $, respectively; in particular, if either $ A $ or $ B $ are real closed valuation rings, then $ C $ is real closed if and only if both  $ I $ and $ J $ are prime ideals of $ A  $ and $ B $, respectively.
\end{lemma}

\begin{proof}
	If $ I $ and $ J $ are  radical ideals, then  $ C $ is real closed by items (I) and (II) (i) in Theorem \ref{T.RCSVR.properties_rcr}. Conversely, suppose that $ C $ is real closed and let $ p_A: C \lonto A $ and  $ p_B : C \lonto B $ be the canonical surjections. Since both $ A  $ and $ B $ are real closed, $ \text{ker}(p_A) $ and $ \text{ker}(p_B) $ are radical ideals of $ C $ (Theorem \ref{T.RCSVR.properties_rcr.II.i}), and since $ C $ is real closed,  $ \text{ker}(p_A)\oplus \text{ker}(p_B) $ is also a radical ideal (Theorem \ref{T.RCSVR.properties_rcr.II.iv}), therefore $ I  $ and $ J $ are radical ideals by Remark \ref{R.SV.cofactor}; the last statement of the lemma follows from $ B/J \cong A/I $ and from the fact that radical ideals in valuation rings are prime ideals.
\end{proof}

\begin{theorem}\label{T.RCSVR.equiv_loc_real_closed_SV-ring_finite_rank}
	Let $ n \in  \N^{\geq 2}$ and $ A $ be a ring which is not a field. The following are equivalent:
\begin{enumerate}[\normalfont(i)]
	\item $ A $ is a local real closed SV-ring of rank  at most $ n $.
	\item $ A $ is a local real closed ring, and there exist non-trivial real closed valuation rings $ A_1, \dots, A_{n}  $ and an injective ring homomorphism $ \epsilon: A \linto \prod_{i=1}^n A_i $ such that $ \epsilon(A)  $ is  a subdirect product of $ \{ A_1, \dots, A_{n} \}   $. 
\item There exist non-trivial real closed valuation rings $ A_1, \dots, A_n $ and an injective ring homomorphism $ \epsilon: A \linto   \prod_{i=1}^n A_i  $, such that  $ (\pi^n_1 \circ \epsilon)(A)= A_1 $ \emph{(}\emph{Notation \ref{N.SV.pimj}}\emph{)}, and for all $ j \in [n-1]  $ there exist prime ideals $\mathfrak{i}_j\subsetneq (\pi^n_j \circ \epsilon)(A) $ and  $  \mathfrak{j}_{j+1}  \subsetneq A_{j+1} $ such that $ (\pi^n_j \circ \epsilon)(A)/\mathfrak{i}_j\cong A_{j+1}/\mathfrak{j}_{j+1}  $  and $$(\pi^n_{j+1} \circ \epsilon)(A) = (\pi^n_j \circ \epsilon)(A)\times_{A_{j+1}/\mathfrak{j}_{j+1}} A_{j+1};$$in particular,  $$\hspace{1.3cm}  A \cong (\pi^n_n \circ \epsilon)(A) = (((A_1 \times_{A_2/\mathfrak{j}_2} A_2) \times_{A_3/\mathfrak{j}_3}A_3) \dots \times_{A_{n-1}/\mathfrak{j}_{n-1} } A_{n-1})\times_{A_{n}/\mathfrak{j}_{n}}  A_{n} .$$ 
\item There exist prime ideals $ \mathfrak{p}_1, \dots, \mathfrak{p}_{n} \subseteq A $ such that  $  \bigcap_{i =1}^n \mathfrak{p}_i  = (0) $, $ \mathfrak{q}_j := (\bigcap_{i=1}^{j-1}\mathfrak{p}_i)+ \mathfrak{p}_{j} \in \emph{Spec}(A)$ for all $ j \in \{2, \dots, n\}  $, $  A/\mathfrak{p}_i $ is a non-trivial real closed valuation ring for all $ i\in [n]  $,  and the canonical ring homomorphism $ A \lra \prod_{i=1}^n A/\mathfrak{p}_i  $ given by $  a \mapsto (a/\mathfrak{p}_1, \dots, a/\mathfrak{p}_n) $ restricts to an isomorphism $$ \hspace{1.2cm} A \cong \left(\left(\left(\frac{A}{\mathfrak{p}_1} \times_{\frac{A}{\mathfrak{q}_2}}\frac{A}{\mathfrak{p}_2}\right) \times_{\frac{A}{\mathfrak{q}_3}}\frac{A}{\mathfrak{p}_3}\right) \dots \times_{\frac{A}{\mathfrak{q}_{n-1}}}\frac{A}{\mathfrak{p}_{n-1}}\right)\times_{\frac{A}{\mathfrak{q}_{n}}}\frac{A}{\mathfrak{p}_{n}} .$$ 
	 \end{enumerate}
	 Moreover, if any of the items \emph{(i) - (iv)} holds and $ \emph{rk}(A) =n$, then the prime ideals $ \mathfrak{p}_1, \dots, \mathfrak{p}_n  \subseteq A$ in item \emph{(iv)} are pairwise incomparable under subset inclusion and $ \emph{Spec}^{\emph{min}}(A) = \{ \mathfrak{p}_i \mid i \in [n] \} $.
\end{theorem}

\begin{proof}
	Note first that each of the items (i) - (iv) in the theorem implies the corresponding item in Theorem \ref{T.SV.equiv_red_loc_SV-ring_finite_rank}.

	\underline{(i) $ \Ra $  (ii).} By the implication (i) $ \Ra $  (ii) in Theorem \ref{T.SV.equiv_red_loc_SV-ring_finite_rank}, it suffices to show that the non-trivial valuation rings $A_1, \dots, A_n  $ in item (ii) of Theorem \ref{T.SV.equiv_red_loc_SV-ring_finite_rank} are real closed. Since  $ \pi_i\circ \epsilon : A \lonto A_i $ is surjective and  $ A_i $ is a  domain for each $ i \in [n]  $, $ \text{ker}(\pi_i \circ \epsilon) $ is a prime ideal of $ A $, hence radical, therefore $ A_i \cong A/\text{ker}(\pi_i \circ \epsilon) $ is real closed since $ A  $ is real closed (Theorem \ref{T.RCSVR.properties_rcr.II.i}).

	\underline{(ii) $ \Ra $  (iii).} By the implication  (ii) $ \Ra $  (iii) in  Theorem \ref{T.SV.equiv_red_loc_SV-ring_finite_rank}, it suffices to show that the ideals $ I_j \subsetneq (\pi^n_{ j} \circ \epsilon)(A)$ and $ J_{j+1} \subsetneq A_{j+1}$ in item (iii) of Theorem \ref{T.SV.equiv_red_loc_SV-ring_finite_rank} are prime for all $ j \in [n-1]  $. Pick $ j \in [n-1]  $; since $ (\pi^n_{j+1}\circ \epsilon)(A) $ is a subring of $ \prod_{i =1}^{j+1}A_i$, $ (\pi^n_{j+1}\circ \epsilon)(A) $ is reduced, and thus $ \text{ker}(\pi^n_{ j+1}\circ \epsilon) $ is a radical ideal of $ A $, and since $ A $ is real closed, it follows that $  (\pi^n_{j+1} \circ \epsilon)(A)$  is also real closed for all $ j \in \{1, \dots, n-1\} $ by Theorem \ref{T.RCSVR.properties_rcr.II.i}. Since $ \pi^n_{1}(A) = \pi_1(A) =A_1 $ is real closed, it now follows from $ (\pi_{j+1}^n \circ \epsilon )(A) \cong (\pi_{j}^n \circ \epsilon )(A) \times_{A_{j+1}/J_{j+1}} A_{j+1}$ and  from Lemma \ref{lem.real_closed_fibre_prod} that  $ I_j \subsetneq \pi^n_{j}(A)$ and $ J_{j+1} \subsetneq A_{j+1}$ are prime ideals for all $ j \in \{1, \dots, n-1\}  $.

	\underline{(iii) $ \Ra $  (iv).} By the implication  (iii) $ \Ra $  (iv) in  Theorem \ref{T.SV.equiv_red_loc_SV-ring_finite_rank}, it suffices to show that the non-trivial valuation ring $ A/\mathfrak{p}_i $ is real closed for all $ i \in [n]  $  and that $ G_j \in \text{Spec}(A)  $ for all $ j \in \{ 2, \dots, n \}  $, where $ \mathfrak{p}_i $ and $ G_j  $ are the ideals of $A $ in item  (iv) of Theorem \ref{T.SV.equiv_red_loc_SV-ring_finite_rank}. By (iii) and Lemma \ref{lem.real_closed_fibre_prod}, $ A $ is real closed, therefore $ A/\mathfrak{p}_i $ is real closed by Theorem \ref{T.RCSVR.properties_rcr.II.i}. Since $ A $ is real closed and the intersection of radical ideals is radical, $ G_j = (\bigcap_{i = 1}^{j-1} \mathfrak{p}_i) + \mathfrak{p}_j $  is a radical ideal of $ A $ by Theorem \ref{T.RCSVR.properties_rcr.II.iv}, hence $ G_j/\mathfrak{p}_j$ is a radical ideal of the valuation ring $ A/\mathfrak{p}_j $, therefore it is  prime, from which $ G_j \in \text{Spec}(A) $ follows.

	\underline{(iv) $ \Ra $  (i).} By the implication of (iv) $ \Ra $ (i) in Theorem \ref{T.SV.equiv_red_loc_SV-ring_finite_rank}, it suffices to argue that $ A $ is real closed; but this follows from Lemma \ref{lem.real_closed_fibre_prod}  together with the fact that each $ A/\mathfrak{p}_i  $ is real closed and that each of the ideals $ \mathfrak{q}_j $ are prime.   

	Finally, if any of the items (i) - (iv) holds and $ \mathfrak{p}_1, \dots, \mathfrak{p}_n\subseteq A $ are the prime ideals in item (iv), then $ A $ is a local ring and a subdirect product of the domains  $ \{ A/\mathfrak{p}_1, \dots, A/\mathfrak{p}_n \}  $, therefore if $ \text{rk}(A) =n $, then $ \text{Spec}^{\text{min}}(A)= \{ \mathfrak{p}_i \mid i \in [n] \}  $ follows by Lemma \ref{L.SV.subdirect_min_prime_id.ii}.
\end{proof}

\begin{remark}\label{R.RCSVR.branch_fibre}
	Let $ A  $ be a local real closed SV-ring of rank $ n\in \N^{\geq 2} $ and write $ \text{Spec}^{\text{min}}(A) = \{ \mathfrak{p}_i \mid i \in [n] \} $; the ideals $ \mathfrak{q}_j := (\bigcap_{i=1}^{j-1}\mathfrak{p}_i)+ \mathfrak{p}_{j} \in \text{Spec}(A)$ ($ j \in [n-1] $) in item (iv) of Theorem \ref{T.RCSVR.equiv_loc_real_closed_SV-ring_finite_rank} are exactly the branching ideals of $ A $  (note that it could be the case that $ \mathfrak{q}_i = \mathfrak{q}_j  $ for $ i \neq  j  $). Pick $ j \in [n-1] $; then $ \mathfrak{q}_j = (\bigcap_{i=1}^{j-1}\mathfrak{p}_i)+ \mathfrak{p}_{j} = \bigcap_{i=1}^{j-1}(\mathfrak{p}_i + \mathfrak{p}_j)  = \mathfrak{p}_{k} +\mathfrak{p}_j$ for some $ k \in [j-1] $ by items (II) (iii), (II) (iv) (a), and (II) (iv) (b) in Theorem \ref{T.RCSVR.properties_rcr}, therefore $ \mathfrak{q}_j $ is a branching ideal. Conversely, suppose that $ \mathfrak{q} \in \text{Spec}(A) $  is a branching ideal; By Remark \ref{R.RCSVR.branch_1} there exist distinct $ i , j \in [n] $ such that $ i \neq j $ and $ \mathfrak{q}= \mathfrak{p}_i +\mathfrak{p}_j $.  Let $ i \in [n] $ be minimal with the property that there exists $ j_0 \in [n]  $ such that $ i < j_0 $ and  $ \mathfrak{p}_i + \mathfrak{p}_{j_0}= \mathfrak{q} $, and let $ j \in [n] $ be minimal with the property that $  i <j   $ and $ \mathfrak{p}_i + \mathfrak{p}_j = \mathfrak{q} $. Assume for contradiction that there exists $ k \in [j-1] $ such that $ \mathfrak{p}_k + \mathfrak{p}_j \subsetneq \mathfrak{p}_i +\mathfrak{p}_j $; note in particular that $ \mathfrak{p}_i +\mathfrak{p}_k \subseteq \mathfrak{p}_i+\mathfrak{p}_j $. Then $ \mathfrak{p}_i + \mathfrak{p}_k \not\subseteq \mathfrak{p}_j + \mathfrak{p}_k $ as otherwise $ \mathfrak{p}_i \subseteq  \mathfrak{p}_i + \mathfrak{p}_k \subseteq \mathfrak{p}_j + \mathfrak{p}_k $ implies $ \mathfrak{p}_i + \mathfrak{p}_j \subseteq  \mathfrak{p}_j + \mathfrak{p}_k  \subsetneq \mathfrak{p}_i + \mathfrak{p}_j $, therefore $ \mathfrak{p}_j  + \mathfrak{p}_k \subseteq \mathfrak{p}_i + \mathfrak{p}_k $, and thus $ \mathfrak{p}_i + \mathfrak{p}_j \subseteq \mathfrak{p}_i + \mathfrak{p}_k \subseteq \mathfrak{p}_i + \mathfrak{p}_j $, hence $ \mathfrak{p}_i + \mathfrak{p}_k = \mathfrak{p}_i + \mathfrak{p}_j = \mathfrak{q} $, a contradiction to minimality of $ j $. Therefore  $ \mathfrak{p}_i + \mathfrak{p}_j \subseteq \mathfrak{p}_k + \mathfrak{p}_j$ for all $ k \in [j-1] $, and thus $ \mathfrak{q}_j = (\bigcap_{k \in [j-1]}\mathfrak{p}_k)+ \mathfrak{p}_j = \bigcap_{k \in [j-1]}(\mathfrak{p}_k + \mathfrak{p}_j) = \mathfrak{p}_i +\mathfrak{p}_j=\mathfrak{q}$, as required.
\end{remark}

\begin{remark}\label{R.RCSVR.str_thm_loc_rcr_fin_rk}
	It is clear from the proof of Theorem \ref{T.RCSVR.equiv_loc_real_closed_SV-ring_finite_rank} that a similar structure theorem holds for local real closed rings of finite rank: just replace every occurrence of \enquote{non-trivial real closed valuation ring} by \enquote{non-trivial real closed domain}.
\end{remark}

\begin{remark}
	Any real closed ring has \textit{bounded inversion} (\cite[Example 2.11, Proposition 12.4. (b)]{schwartz/madden.safr}), therefore a real closed domain is a real closed valuation ring if and only if it is \textit{$ 1 $-convex}  by \cite[Lemma 2.2]{larson.SV_and_related_f-rings_and_spaces}; in particular, local real closed SV-rings of finite rank are \textit{finitely $ 1 $-convex $ f $-rings} in the sense of Larson, see \cite{larson.finitely_1-conv_f-rings}. 
\end{remark}

Let $ A $ be a local real closed SV-ring of rank $ n \in \N^{\geq 2} $; by the implication (i) $ \Ra $  (iv) in Theorem \ref{T.RCSVR.equiv_loc_real_closed_SV-ring_finite_rank}, $ A $ is isomorphic to an finite iterated fibre product of non-trivial real closed valuation rings $ A/\mathfrak{p}_i  $ along surjective homomorphisms  $ A/\mathfrak{p}_i \lonto A/\mathfrak{q}_j $ onto real closed valuation rings $ A/\mathfrak{q}_j  $, and the simplest such rings are the ones for which $ A/\mathfrak{q}_i = A/\mathfrak{q}_j  $ for all $ i , j \in [n] $, i.e., local real closed SV-rings of rank $ n $ with exactly one branching ideal (Remark \ref{R.RCSVR.branch_fibre}). The focus of the remaining part of this section is on rings of this latter class.

\begin{notation}\label{N.RCSVR.I-fold_fibr_prod}
Let $ \{ A_i \}_{i \in I} $ be a non-empty family of rings such that there exists a ring $ B $ and surjective ring homomorphisms $ f_i : A_i \lonto B $ for all  $ i \in I $. Let \[
	{\prod}_{B, i \in I}A_i:= \left\{(a_i)_{i\in I} \in \prod_{i\in I} A_i \mid f_i(a_i)=f_j(a_j)\text{ for all } i, j  \in I \right\} 
\] be the \textit{$ I $-fold fibre product of $ \{ A_i \}_{i \in I}  $ over $ B $ \emph{(}along $ \{ f_i \}_{i \in I}  $\emph{)}}. If  there exists a ring $ A $  such that $ A_i = A $ for all $ i \in I $, then set $ {\prod}_B^I A:= {\prod}_{B, i \in I}A_i $; if moreover $ I = [n] $ for some $ n \in \N $, then set $ {\prod}_B^n A:=  {\prod}_B^{[n]} A  $.
\end{notation}

\begin{lemma}\label{L.RCSVR.equiv_loc_rcsvr_rk_n_exactly_one_branching_ideal}
Let $ A $ be a local real closed SV-ring of rank $ n \in \N^{\geq 2} $ and write $ \emph{Spec}^{\emph{min}}(A) = \{ \mathfrak{p}_i \mid i \in [n] \}  $. The following are equivalent:
\begin{enumerate}[\normalfont(i), ref=\ref{L.RCSVR.equiv_loc_rcsvr_rk_n_exactly_one_branching_ideal} (\roman*)]
\item $ A $ has exactly one branching ideal.  
\item $ \mathfrak{p}_i+ 	\mathfrak{p}_j = \mathfrak{p}_k +\mathfrak{p}_{\ell} $ for all $ i, j , k , \ell \in [n] $ such that $ i \neq j $ and $ k \neq \ell $.  
\item\label{L.RCSVR.equiv_loc_rcsvr_rk_n_exactly_one_branching_ideal.iii}	There exist non-trivial real closed valuation rings $ A_1, \dots, A_n $	and surjective ring homomorphisms $ f_i: A_i \lonto C $ \emph{(}$ i \in [n] $\emph{)} onto a real closed valuation ring  $ C $ such that $ \emph{ker}(f_i)\neq (0) $ for all $ i \in [n]	$ and $ A \cong  {\prod}_{C, i \in [n]}A_i$.
\end{enumerate}
In particular, if any of the conditions \emph{(i) - (iii)} holds and $ \mathfrak{b}_A \in \emph{Spec}(A) $ is the unique branching ideal of $ A $, then
 \begin{enumerate}[\normalfont(a), , ref=\ref{L.RCSVR.equiv_loc_rcsvr_rk_n_exactly_one_branching_ideal} (\alph*)]
\item\label{L.RCSVR.equiv_loc_rcsvr_rk_n_exactly_one_branching_ideal.a} 	$ \mathfrak{b}_A = \mathfrak{p}_i + \mathfrak{p}_j$ for all $ i, j \in [n] $ such that $ i \neq j $, and
\item\label{L.RCSVR.equiv_loc_rcsvr_rk_n_exactly_one_branching_ideal.b} the canonical embedding $ A \linto \prod_{i \in [n]}A/\mathfrak{p}_i $ given by $ a \mapsto (a/\mathfrak{p}_i)_{i \in [n]} $ corestricts to an isomorphism $ A \cong  {\prod}_{A/\mathfrak{b}_A, i \in [n]}A/\mathfrak{p}_i $
\end{enumerate}
\end{lemma}

\begin{proof} 
	\underline {(i) $ \Ra $  (ii).} Clear by Remark \ref{R.RCSVR.branch_1}.

	\underline{(ii) $ \Ra $  (iii).}  Since $ A $ is local of rank $ n  $,	$ \text{Spec}^{\text{min}}(A):= \{ \mathfrak{p}_1, \dots, \mathfrak{p}_n \}  $. Define  $ A_i := A/\mathfrak{p}_i $ and $ \mathfrak{q}_j := (\bigcap_{i=1}^{j-1}\mathfrak{p}_i)+ \mathfrak{p}_{j} $ for all $ j \in \{2, \dots, n\}  $. If $ j\in \{ 2, \dots, n \}  $, then  $ \mathfrak{q}_j = \bigcap_{i=1}^{j-1}(\mathfrak{p}_i+ \mathfrak{p}_{j} )$ by Theorem \ref{T.RCSVR.properties_rcr.II.iv.b}, and since $ \mathfrak{p}_i+ \mathfrak{p}_{j}  \in \text{Spec}(A)$ for all $ i \in [j-1] $ by Theorem \ref{T.RCSVR.properties_rcr.II.iv.a}, there exists $ j' \in [j-1] $  such that $ \mathfrak{q}_j = \bigcap_{i=1}^{j-1}(\mathfrak{p}_i+ \mathfrak{p}_{j} )=\mathfrak{p}_{j'}+\mathfrak{p}_j  $ by Theorem \ref{T.RCSVR.properties_rcr.II.iii}, therefore it follows by (ii) that $ \mathfrak{q}_j =\mathfrak{q}_i $ for all $ i, j \in \{ 2, \dots, n \}  $. Define $ \mathfrak{q}:= \mathfrak{q}_j  $ for some (equivalently, all) $ j \in \{ 2, \dots, n \}   $; then $ A \cong {\prod}_{A/\mathfrak{q}, i \in [n]}A/\mathfrak{p}_i $ by the implication (i) $ \Ra $  (iv) and the moreover part in Theorem \ref{T.RCSVR.equiv_loc_real_closed_SV-ring_finite_rank}, from which (iii) follows.

	\underline{(iii) $ \Ra $  (i).} $ \text{Spec}(C) $ is a closed subspace of $ \text{Spec}(A) $ with unique generic point $\mathfrak{q} = \text{ker}(A \onto C)$, therefore follows from the description in \cite[Section 12.5.7]{dickmann/schwartz/tressl.specbook} of $ \text{Spec}(A) $ that the unique branching
ideal of $ A $ is $ \mathfrak{q}$.  

The last statement in the lemma follows from the proof of the equivalences (I) - (iii).
\end{proof}

\begin{definition}\label{D.RCSVR.type}
Let $ A $ be a ring and $ n \in \N^{\geq 2} $.
\begin{enumerate}[\normalfont(i)]
	\item $ A $ is of \textit{type $ (n,1) $}  if $ A $ is a local real closed SV-ring of rank $ n $ with  exactly one branching ideal $ \mathfrak{b}_A $, which moreover is maximal.
	\item $ A $ is of \textit{type $ (n,2) $} if $ A $ is a  local real closed SV-ring of rank $ n $ with exactly one branching ideal $ \mathfrak{b}_A $, which moreover is not maximal.
\end{enumerate}
\end{definition}

\begin{example}\label{ex.germs}
	Let $ R $ be a real closed field, $ X \subseteq R^m $ be a $ 1  $-dimensional semi-algebraic subset, and $ A$ be the ring of continuous semi-algebraic functions $ X \lra R $. Pick $ x \in X $ and suppose that there are  $ n \in \N^{\geq 2} $ half-branches of $ X $ at $ x $  (\cite[Definition 9.5.2]{bcr}); then the ring of germs of functions $ f \in A$ at  $ x $ is a ring of type $ (n,1) $. This has been noted in \cite[Example 2.6 (c)]{schwartz.SV}, but another proof of this fact will be given here; for notational simplicity consider the case $m =1 $ and  $ X= R $, so that each  $ r \in R $ has exactly two  half-branches, denoted by $ r^{-} $ and $ r^{+} $. First note the ring of germs of functions $ f \in A $ at $ r \in R  $ is exactly the localization $ A_{\mathfrak{m}_r} $, where $ \mathfrak{m}_r $ is the maximal ideal of all functions $ g \in A $ such that $ g(r)= 0 $. The kernel $ I_r $ of the localization map $ A \lonto A_{\mathfrak{m}_r} $ consists of the functions that vanish on an open neighbourhood of $ r $; in particular, $ f \in A $ belongs to $ I_r $ if and only if  $ f $ vanishes on the two half-branches $ r^- $ and $ r^+ $ of  $ r $, i.e., $  I_r = \mathfrak{p}_{r^-} \cap  \mathfrak{p}_{r^+}$, where \[
		\mathfrak{p}_{r^-}:= \{ f \in A \mid \exists \epsilon \in R^{>0} \text{ such that } f_{\upharpoonright (r-\epsilon, r]}= 0 \} \text{ and } \mathfrak{p}_{r^+}:= \{ f \in A \mid \exists \epsilon \in R^{>0} \text{ such that } f_{\upharpoonright [r, r+\epsilon)}= 0 \}.
	\] Note also that if $ f \in \mathfrak{m}_r$, then the functions $ f_1, f_2 : R \lra R $ defined by \[
	f_1(x):= \begin{cases}
		0 &\text{ if } x\leq r\\
		f &\text{ if } r\leq x
		\end{cases} \ \ \text{ and } \ \ f_2(x):= \begin{cases}
		f &\text{ if } x\leq r\\
		0 &\text{ if } r\leq x
		\end{cases}
	\] are continuous and semi-algebraic, $ f_1 \in \mathfrak{p}_{r^-} $, $ f_2 \in \mathfrak{p}_{r^+} $, and $ f=f_1 + f_2 $, therefore  $ \mathfrak{m}_r = \mathfrak{p}_{r^-} + \mathfrak{p}_{r^+}$, and \[
	A_{\mathfrak{m}_r} \cong A/ (\mathfrak{p}_{r^-} \cap  \mathfrak{p}_{r^+}) \cong A/\mathfrak{p}_{r^-} \times_{A/\mathfrak{m}_r} A/\mathfrak{p}_{r^+},
\] where the second isomorphism follows from Lemma \ref{L.SV.fibre}; conclude that  $ A_{\mathfrak{m}_r} $ is a ring of type $ (2, 1) $ by  Example \ref{E.RCSVR.rcvr_germs}.
\end{example}

\subsubsection{\texorpdfstring{Embeddings of rings of type $ (n,1) $ and of type $ (n,2) $}{Embeddings of rings of type (n,1) and of type (n,2)}} This section concludes with embedding statements about rings of type $ (n,1) $ and of type $ (n,2) $ which have a key impact on the model theory of these rings; in particular, Lemmas \ref{L.RCSVR.1_square} and  \ref{L.RCSVR.2_square} (which can be thought of as a higher-rank version of Proposition \ref{P.RCSVR.rcvr_embedd_nice}) are essential for the  model completeness proofs in Subsection \ref{SUBSEC.mod_compl}. Start with the following and almost trivial:

\begin{lemma}\label{L.RCSVR.canonical_embedd}
	Let $ A  $ be a local real closed SV-ring of rank $ n \in \N^{\geq 2} $. There exists a ring $ A' $ of type $ (n,1) $ and a local embedding $ A \linto A' $.
\end{lemma}

\begin{proof}
	Set $ \text{Spec}^{\text{min}}(A) = \{ \mathfrak{p}_i \mid i \in [n]\}   $. The canonical map $ A \linto \prod_{i \in [n]}A/\mathfrak{p}_{i} $ given by $ a \mapsto (a/\mathfrak{p}_1, \dots, a/\mathfrak{p}_n) $ is an embedding, and  each $A/\mathfrak{p}_i $ is a non-trivial real closed valuation ring with residue field $ A/\mathfrak{m}_A $, from which follows that $ A':=\prod_{A/\mathfrak{m}_A, i \in [n]}A/\mathfrak{p}_i $ is a ring of type $ (n,1) $ and the embedding $ A \linto \prod_{i \in [n]}A_i $ corestricts to a local embedding $ A \linto A' $.
\end{proof}

In fact, Lemma \ref{L.RCSVR.canonical_embedd} still holds even if $ A $ is not an SV-ring; first, a lemma:

\begin{lemma}\label{L.conv_hull_rcd}
	Let $ A $ be a real closed domain.
	\begin{enumerate}[\normalfont(i), ref=\ref{L.conv_hull_rcd} (\roman*)]
\item 		$A= \emph{qf}(A) $ if and  only if $ A  $ is cofinal in $ \emph{qf}(A) $.
\item\label{L.conv_hull_rcd.ii} Suppose that $ A $ is non-trivial, i.e., $ A \neq \emph{qf}(A) $. The convex hull $ V_A$ of $ A  $ in $ \emph{qf}(A) $ is a non-trivial real closed valuation ring such that $ A \cap \mathfrak{m}_{V_A}= \mathfrak{m}_A $.
\end{enumerate}
\end{lemma}

\begin{proof}
	(i). Suppose that $ A $ is cofinal in $ \text{qf}(A) $ and pick $ a \in A^{>0} $. Then there exists $ b \in A $ such that $ 0 < a^{-1} < b $, therefore $ 0< 1< ab $ and thus there exists $ c \in A $  such that $ abc=1^2 = 1 $ (Definition \ref{D.RCSVR.rcr.iii}), from which $ a^{-1} =bc\in A $ follows.

	(ii). $ V_A $ is a non-trivial real closed valuation ring because it is a proper convex subring of $ \text{qf}(A) $ by item (i). Clearly $ A \cap \mathfrak{m}_{V_A} \subseteq \mathfrak{m}_A $. Conversely, pick $ a \in \mathfrak{m}_A $ such that $ a >0 $ and assume for contradiction that $ a\notin \mathfrak{m}_{V_A} $; then $ a^{-1} \in V_A$, therefore there exists $ b \in A $ such that  $ 0 < a^{-1}< b $, and thus arguing as in the proof of item  (i) it follows that $ a^{-1}\in A $, yielding the required contradiction.
\end{proof}

\begin{proposition}\label{P.RCSVR.embedd_loc_rcr_fin_rank}
	Let $ A  $ be a local real closed ring of rank $ n \in \N^{\geq 2} $. There exists a ring $ B $ of type $ (n,1) $ and a local embedding $ A \linto B $.
\end{proposition}

\begin{proof}
 Set $ \text{Spec}^{\text{min}}(A) = \{ \mathfrak{p}_i\mid i \in [n] \}  $; each $ A/\mathfrak{p}_i $ is a non-trivial real closed domain with residue field $ A/\mathfrak{m}_A $, therefore the canonical embedding $ A \linto \prod_{i \in [n]} A/\mathfrak{p}_i $ given by $ a \mapsto (a/\mathfrak{p}_i)_{i \in [n]} $ corestricts to a local embedding $ A \linto \prod_{A/\mathfrak{m}_A, i \in [n]} A/\mathfrak{p}_i =:B_0 $. For each $ i \in [n] $, the convex hull $ V_i $ of $ A/\mathfrak{p}_i $ in $ \text{qf}(A/\mathfrak{p}_i) $ is a non-trivial real closed valuation ring by Lemma \ref{L.conv_hull_rcd.ii}. Let $ \textbf{\emph{k}}_i:= V_i/\mathfrak{m}_{V_i} $ and $ \Gamma_i:= \text{qf}(V_i)^{\times}/V_i^{\times} $ for all $ i \in [n] $; note that since $ A \cap \mathfrak{m}_{V_i} = \mathfrak{m}_A $ (Lemma \ref{L.conv_hull_rcd.ii}), the residue field $ A/\mathfrak{m}_A $ embeds into $ \textbf{\emph{k}}_i $ for all $  i \in [n] $. Let $ \emph{\textbf{k}} $ be a real closed field amalgamating all the  $  \emph{\textbf{k}}_i$  over $ A/\mathfrak{m}_A $,  let $ \Gamma $ be a divisible  totally ordered abelian group into which all the $ \Gamma_i $ embed, and let $ \epsilon_i: V_i \linto \emph{\textbf{k}}_i[[\Gamma_i]] $ be a local embedding for all $ i \in [n] $ (these exist by Proposition \ref{P.RCSVR.rcvr_embedd_nice}); all this data fits into commutative diagrams

\noindent\adjustbox{center}{
\begin{tikzpicture}[node distance=2cm, auto]
   \node (lm)  {$ A/\mathfrak{m}_A $};
 \node (llb) [left=1.9cm of lm] {$ A/\mathfrak{p}_i $};
  \node (rrb) [right=1.9cm of lm] {$ A/\mathfrak{p}_j $};
  \node (llt) [above of=llb] {$ V_i $};
  \node (rrt) [above of=rrb] {$ V_j $};
  \node (mm) [above=3.45cm of lm] {$ \textbf{\emph{k}} $};
  \node (lt) [right of=llt] {$\emph{\textbf{k}}_i$};
  \node (rt) [left of=rrt] {$\emph{\textbf{k}}_j$}; 
  \node (lltt) [above of=llt] {$ \emph{\textbf{k}}_i[[\Gamma_i]] $};
  \node (rrtt) [above of=rrt] {$ \emph{\textbf{k}}_j[[\Gamma_j]] $};
  \node (llttt) [above of=lltt] {$ \emph{\textbf{k}}_i[[\Gamma]] $};
  \node (rrttt) [above of=rrtt] {$ \emph{\textbf{k}}_j[[\Gamma]] $};
  \node (tm) [above of=mm] {$ \textbf{\emph{k}}[[\Gamma]] $};
  \draw[{Hooks[right]}->] (lm) to node {} (lt);
  \draw[{Hooks[right]}->] (lm) to node [swap] {} (rt);
  \draw[{Hooks[right]}->] (llb) to node {} (llt);
 \draw[{Hooks[right]}->] (rrb) to node {} (rrt);
  \draw[{Hooks[right]}->] (lt) to node {} (mm);
  \draw[{Hooks[right]}->] (rt) to node [swap] {} (mm);
  \draw[{Hooks[right]}->] (llt) to node {$ \epsilon_i $} (lltt);
  \draw[{Hooks[right]}->] (rrt) to node [swap] {$ \epsilon_j $} (rrtt);
  \draw[{Hooks[right]}->] (lltt) to node {} (llttt);
  \draw[{Hooks[right]}->] (rrtt) to node [swap] {} (rrttt);
  \draw[{Hooks[right]}->] (llttt) to node {} (tm);
  \draw[{Hooks[right]}->] (rrttt) to node  {} (tm); 
 \draw[->>] (llb) to node {} (lm);
 \draw[->>] (rrb) to node {} (lm);
\draw[->>] (llt) to node {} (lt);
 \draw[->>] (rrt) to node {} (rt);
\draw[->>] (lltt) to node {} (lt);
 \draw[->>] (rrtt) to node {} (rt);
 \draw[->>] (tm) to node {} (mm);
\end{tikzpicture}
}
for all $ i, j \in [n] $. Commutativity of the diagrams above for all $ i , j \in [n] $ implies that the resulting composite local embeddings  \[
A/\mathfrak{p}_i \linto V_i \overset{\epsilon_i}{\linto} \emph{\textbf{k}}_i[[\Gamma_i]]\into \emph{\textbf{k}}_i[[\Gamma]] \into \emph{\textbf{k}}[[\Gamma]] 
\]
for all $ i \in [n] $ yield a local embedding $ B_0 \linto {\prod}_{\textbf{\emph{k}}}^n \textbf{\emph{k}}[[\Gamma]]=:B $, therefore  the composite map $ A \into B_0 \into B $ is a local embedding of $ A $ into the ring $ B $ of type $ (n,1) $, as required.
\end{proof}

\begin{convention}\label{C.RCSVR.factors}
	Let $ j \in [2]$,  and let $ A $ and $ B $ be rings of type $ (n, j)  $ such that $ A \subseteq B $; write also $ \text{Spec}^{\text{min}}(A):= \{ \mathfrak{p}_i \mid i \in [n] \}  $ and $ \text{Spec}^{\text{min}}(B):= \{ \mathfrak{q}_i \mid i \in [n] \}  $. By Corollary \ref{C.SV.min_prime_ann} it can be assumed that  $ \mathfrak{q}_i \cap A=\mathfrak{p}_i $ for all $ i \in [n] $, and thus 	by the implication (i) $ \Ra $  (iii) and item (b) in Lemma \ref{L.RCSVR.equiv_loc_rcsvr_rk_n_exactly_one_branching_ideal} it can be assumed that $ A=  {\prod}_{C, i \in [n]}A_i $ and $  B=  {\prod}_{D, i \in [n]}B_i  $, where $ A_i:= A/\mathfrak{p}_i $ and $ B_i:= B/\mathfrak{q}_i $ for all $ i\in [n] $, $ C:= A/\mathfrak{b}_A $, and $ D:= B/\mathfrak{b}_B $; in particular, the embedding $ A \subseteq B $ induces embeddings $ A_i \subseteq B_i $ for all  $ i\in [n] $.
\end{convention}

\begin{lemma}\label{L.RCSVR.embedd_1n_local_factors}
	Let $ A  $ and $ B $ be rings of type $ (n,1) $. Any embedding $ A \subseteq B $ is local and it induces local embeddings  $ A_i \subseteq B_i  $ for all  $ i \in [n] $ \emph{(Convention \ref{C.RCSVR.factors})}. 
\end{lemma}

\begin{proof}
	That any embedding  $ A \subseteq B $ is local follows from Lemma \ref{L.RCSVR.embedd_loc_rcr_max_id_branch_implies_local}; since $ A \subseteq B $ is local, so are each of the embeddings $ A_i \subseteq B_i $ (Remark \ref{R.SV.local_embedd_factor}).
\end{proof}

Lemma \ref{L.RCSVR.embedd_1n_local_factors} says that any embedding $ A \subseteq B $ of rings of type $ (n,1) $ sends the unique branching ideal $ \mathfrak{b}_B$ of $ B $ to the unique branching ideal $ \mathfrak{b}_A $ of $ A $, i.e., $ \mathfrak{b}_B \cap A =\mathfrak{b}_A $, and since the branching ideals of these rings are exactly the maximal ideals, this means that  $ A\subseteq B $ is a local embedding; the next example shows that this property of arbitrary embeddings of rings of type $ (n,1) $ does not carry over to rings to type $ (n,2) $:

\begin{example}\label{E.RCSVR.non_branch_pres_embedd_type_2,2}
	Let $  V$ be a real closed valuation ring of Krull dimension $  4 $ such that $ (\text{Spec}(V), \subseteq) $ is the chain $ (0) \subsetneq \mathfrak{p}  \subsetneq \mathfrak{q} \subsetneq \mathfrak{r} \subsetneq \mathfrak{m}_V$; define $ A:= V \times_{V/\mathfrak{p}} V $, $ B:=V_{\mathfrak{r}} \times_{V_{\mathfrak{r}}/\mathfrak{p}V_{\mathfrak{r}}} V_{\mathfrak{r}} $ and $ C:= V_{\mathfrak{r}}\times_{V_{\mathfrak{r}}/\mathfrak{q}V_{\mathfrak{r}}} V_{\mathfrak{r}} $. Then $ A \subseteq B $ and  $ B \subseteq C $ are embeddings of rings of type  $ (2,2) $ such that the embedding $ A \subseteq B $ is not local but  $ \mathfrak{b}_B \cap A = \mathfrak{b}_A $, and the embedding $ B \subseteq C $ is local but $ \mathfrak{b}_C \cap B \supsetneq \mathfrak{b}_B $; in particular, the embedding $ A \subseteq C $ is not local and  $ \mathfrak{b}_C \cap A \supsetneq \mathfrak{b}_A $. The diagram below represents the composite spectral map $ \text{Spec}(C) \rightarrow \text{Spec}(B) \rightarrow \text{Spec}(A) $ induced by the composite embedding $ A \subseteq B \subseteq C $

\noindent\adjustbox{center}{
\begin{tikzpicture}
[
    level distance=10mm,
    sibling distance=17mm
]
\node  (r_1) at (0,0)   {$ \mathfrak{m}_C $}
	child {node  {}edge from parent [draw=none]
		child {node (c_1) {$ \mathfrak{b}_C $}edge from parent [draw=none]
			child {node (dd_1) {$ \bullet $} [sibling distance=7mm]
			child {node  {$ \text{Ann}_C(a_1) $}}
			child {node  {}edge from parent [draw=none]}}
			child {node  (dd_2) {$ \bullet $}[sibling distance=7mm]
			child {node  {}edge from parent [draw=none]}
			child {node  {$ \text{Ann}_C(a_2) $}}}
	}};

\node  (r_2) at (5,0)   {$ \mathfrak{m}_B $}
	child {node  {}edge from parent [draw=none]
		child {node  (c_2) {$ \bullet $}edge from parent [draw=none]
			child {node (D_2)  {$ \mathfrak{b}_B$} 
			child {node  {$ \text{Ann}_B(a_1) $}}
			child {node  {$ \text{Ann}_B(a_2) $}}
}}};
\node at (10,0)   {$ \mathfrak{m}_A $}
	child {node  (e_1) {$ \bullet $}
		child {node  (e_2) {$ \bullet $}
			child {node  (e_3) {$\mathfrak{b}_A$}
			child {node  {$ \text{Ann}_A(a_1) $}}
			child {node  {$ \text{Ann}_A(a_2) $}}
}}};
\draw[->, dashed] (r_1) -- (r_2);
\draw[->, dashed] (c_1) -- (c_2);
\draw[->, dashed] (dd_1) to [bend right=17] (D_2);
\draw[->, dashed] (dd_2) to [bend left=13] (D_2);
\draw[-] (r_1) -- (c_1);
\draw[-] (r_2) -- (c_2);
\draw[->, dashed] (r_2) --  (e_1);
\draw[->, dashed] (c_2) --  (e_2);
\draw[->, dashed] (D_2) --  (e_3);
\end{tikzpicture}	
}
where $ a_1, a_2 \in A$ are any non-zero orthogonal elements (Corollary \ref{C.SV.min_prime_ann}); an analogous construction shows that for all $ n \in \N^{\geq 2} $ there exists  embeddings $ A \subseteq B  $ and  $ B \subseteq C $ of rings of type $ (n, 2) $  such that $ A \subseteq B $ is not a local embedding and such that $ \mathfrak{b}_C \cap B \supsetneq \mathfrak{b}_B $.
\end{example}

Embeddings $ A \subseteq B $ of rings of type $ (n,2) $ which satisfy  $ \mathfrak{m}_B \cap A = \mathfrak{m}_A$ and $ \mathfrak{b}_B \cap A = \mathfrak{b}_B $ play an important role in the model-theoretic analysis of this class of rings; to this end, make the following: 

\begin{definition}\label{D.RCSVR.good}
	An embedding  $ A \subseteq B $ of rings of type $ (n,2) $ is \textit{good} if	$ \mathfrak{m}_B \cap A = \mathfrak{m}_A$ and $ \mathfrak{b}_B \cap A = \mathfrak{b}_B $.
\end{definition}

\begin{remark}\label{R.RCSVR.good}
	Let $ A $ and $ B $ be rings of type $ (n,j  ) $ such that $ A \subseteq B $ and write $ A=  {\prod}_{C, i \in [n]}A_i $ and $  B=  {\prod}_{D, i \in [n]}B_i  $ (Convention \ref{C.RCSVR.factors}). If $ j =1 $, then $ C $ and $ D  $ are real closed fields and the embedding $ A \subseteq B $ induces an embedding  $ C \subseteq D $ by Lemma \ref{L.RCSVR.embedd_1n_local_factors}; if $ j =2 $, then $ C $ and $ D $ is are non-trivial real closed valuation rings and the embedding $ A \subseteq B $ is good if and only if it induces a local embedding $ C \subseteq D $.  
\end{remark}

\begin{definition}\label{D.RCSVR.sharp_embedd}
	Let $f: A \linto B $ be an injective ring homomorphism and  $ \mathfrak{p} \in \text{Spec}(B) $. Say that $ f $ is  \textit{sharp at $ \mathfrak{p} $} if the induced embedding $ A/f^{-1}(\mathfrak{p}) \linto B/\mathfrak{p}   $ is an isomorphism. 
\end{definition}

\begin{definition}[cf. Definition 6 in \cite{larson.finitely_1-conv_f-rings}]\label{D.RCVSR.homogeneous}
	Let $ j \in [2] $. A ring $ A $ of type $ (n,j) $ is \textit{homogeneous}  if there exists a non-trivial real closed valuation ring $ V $ and a surjective ring homomorphism $ f: V \lonto C $ onto a real closed valuation ring  $ C $ such that  $ A \cong {\prod}_B^n V $ (Notation \ref{N.RCSVR.I-fold_fibr_prod}).
\end{definition}

\begin{example}
	Let $ V $ and $ W  $ be non-trivial real closed valuation rings with isomorphic residue field $ \emph{\textbf{k}} $. If there exists a local embedding $ \epsilon: V \linto W $, then  $ \epsilon $ induces a local embedding  $ f: V\times_{\emph{\textbf{k}}}W  \linto W \times_{\emph{\textbf{k}}} W =:B$, therefore $ f $ is an embedding of a ring of type $ (2,1) $ into a  homogeneous ring of type $ (2,1) $, and $ f $ is sharp at $ \mathfrak{m}_B $. This construction cannot be done in general. More precisely, let $ \emph{\textbf{k}} $ be a real closed field, $ \Gamma  $ be a countable non-archimedean totally ordered divisible abelian group, and define $ A:= \textbf{\emph{k}}[[\Gamma]]\times_{\emph{\textbf{k}}} \textbf{\emph{k}}[[\R]] $; there doesn't exist any embedding $A \linto  \textbf{\emph{k}}[[\Gamma]]\times_{\emph{\textbf{k}}} \textbf{\emph{k}}[[\Gamma]]$ nor any embedding $A \linto  \textbf{\emph{k}}[[\R]]\times_{\emph{\textbf{k}}} \textbf{\emph{k}}[[\R]]$. Otherwise, such embeddings would induce either a local embedding $ \textbf{\emph{k}}[[\R]] \linto   \textbf{\emph{k}}[[\Gamma]]$ or a local embedding $ \textbf{\emph{k}}[[\Gamma]] \linto   \textbf{\emph{k}}[[\R]]$ (Lemma \ref{L.RCSVR.embedd_1n_local_factors}); the latter embeddings give rise to embeddings of divisible totally ordered abelian groups $ \R \linto \Gamma $ or $ \Gamma \linto  \R$, which is impossible by choice of $ \Gamma $. 
\end{example}

\begin{lemma}\label{L.RCSVR.1_square}
	Let $ A$ and $ B $ be rings of type $ (n,1) $  such that $ A \subseteq B $. There exist rings $ A' $ and $ B' $ of type $ (n, 1) $  and embeddings  $ \epsilon_A: A \linto A' $ and $ \epsilon_B: B \linto B' $ such that $ A' \subseteq B' $,  $ \epsilon_A $ is sharp at  $ \mathfrak{b}_{A'} \ (= \mathfrak{m}_{A'})  $, $ \epsilon_B $ is sharp at  $ \mathfrak{b}_{B'} \ (= \mathfrak{m}_{B'})  $, $ A' $ and $ B' $ are homogeneous, and  $ \epsilon_{B\upharpoonright A} = \epsilon_A $.
\end{lemma}

\begin{proof}
	Write $ A= {\prod}_{\emph{\textbf{k}}, i \in [n]}A_i$ and $ B= {\prod}_{\emph{\textbf{l}}, i \in [n]}B_i$, where $ A_i $ and $ B_i $ are non-trivial real closed valuation rings with residue fields $ \textbf{\emph{k}} $ and \textbf{\emph{l}} (respectively) for all $ i \in [n] $, see Convention \ref{C.RCSVR.factors}; set also $ \Gamma_i:= \text{qf}(A_i)^{\times}/A_i^{\times} $ and $ \Delta_i:=  \text{qf}(B_i)^{\times}/B_i^{\times}  $, noting that the local embedding  $ A_i \subseteq B_i $ (Lemma \ref{L.RCSVR.embedd_1n_local_factors})  induces a local embedding $ \emph{\textbf{k}}[[\Gamma_i]]\subseteq \textbf{\emph{l}}[[\Delta_i]] $ for all $ i \in [n] $. By Proposition \ref{P.RCSVR.rcvr_embedd_nice} there exist local embeddings $  \eta_i : A_i \linto \textbf{\emph{k}}[[\Gamma_i]]$ and $ \delta_i : B_i \linto \textbf{\emph{l}}[[\Delta_i]]$ such that $ \delta_{i \upharpoonright A_i} = \eta_{i} $ for all $ i \in [n] $; by embedding all the divisible $ o $-groups $\Delta_i  $ into a divisible $ o $-group $ \Delta $, it follows that the local embeddings $ \eta_i $ and  $ \delta_i $ induce local embeddings  $ \eta_i': A_i \linto \emph{\textbf{k}} [[\Delta]]$ and $ \delta_i': B_i \linto \emph{\textbf{l}}[[\Delta]] $ such that $ \delta'_{i\upharpoonright A_i} = \eta'_i $. Define $ A':= {\prod}_{\emph{\textbf{k}}}^n \emph{\textbf{k}}[[\Delta]] $ (Notation \ref{N.RCSVR.I-fold_fibr_prod}), $ B':= {\prod}_{\emph{\textbf{l}}}^n \emph{\textbf{l}}[[\Delta]] $, $ \epsilon_A $ to be given by $ \epsilon_A(a):= (\eta_1'(a_1), \dots, \eta_n'(a_n)) $, and $ \epsilon_B $ to be given by  $ \epsilon_B(b):= (\delta_1'(b_1), \dots, \delta_n'(b_n)) $; it is clear by construction that this choice of data satisfies the requirements in the statement of the lemma.
\end{proof}

\begin{lemma}\label{L.RCSVR.2_square}
	Let $ A$ and $ B $ be rings of type $ (n, 2) $ such that $ A \subseteq B $ and such that $ A \subseteq B $ is a good embedding \emph{(Definition \ref{D.RCSVR.good})}. There exist rings  $ A' $ and $ B'  $ of type $ (n, 2) $, and good embeddings  $ \epsilon_A: A \linto A' $ and $ \epsilon_B: B \linto B' $, such that $ A' \subseteq B' $, the embedding $ A' \subseteq B' $ is good,  $ \epsilon_A $ is sharp at  $ \mathfrak{b}_{A'}   $, $ \epsilon_B $ is sharp at  $ \mathfrak{b}_{B'}   $, $ A' $ and $ B' $ are homogeneous, and  $ \epsilon_{B\upharpoonright A} = \epsilon_A $.
\end{lemma}

\begin{proof}
	Write $ A= {\prod}_{C, i \in [n]}A_i$ and $ B= {\prod}_{D, i \in [n]}B_i$, where $ A_i $ and $ B_i $ are non-trivial real closed valuation rings having non-trivial real closed valuation rings $ C $ and  $ D $ as homomorphic images (respectively) for all $ i \in [n] $, see Convention \ref{C.RCSVR.factors}; set also $ \mathfrak{b}_{A_i}:= \text{ker}(A_i \onto C) $  and $ \mathfrak{b}_{B_i}:= \text{ker}(B_i \onto D) $, noting that neither $ \mathfrak{b}_{A_i} $ nor $ \mathfrak{b}_{B_i} $ are the zero ideal by the implication (i) $ \Ra $  (iii) in Lemma \ref{L.RCSVR.equiv_loc_rcsvr_rk_n_exactly_one_branching_ideal}. For each $ i \in [n] $, the localization  $ (A_i)_{\mathfrak{b}_{A_i}} $ is a non-trivial real closed valuation ring properly containing $ A_i $ with residue field  $ (A_i)_{\mathfrak{b}_{A_i}}/\mathfrak{b}_{A_i}(A_i)_{\mathfrak{b}_{A_i}}=  (A_i)_{\mathfrak{b}_{A_i}}/\mathfrak{b}_{A_i}= \text{qf}(A_i/\mathfrak{b}_{A_i})= \text{qf}(C)$, therefore $ \widehat{A}:=\prod_{\text{qf}(C), i \in [n]}(A_{i})_{\mathfrak{b}_{A_i}} $ is a ring of type  $ (n,1) $  such that $ A \subseteq  \widehat{A}$; similarly,  $ \widehat{B} := {\prod}_{\text{qf}(D), i \in [n]}(B_i)_{\mathfrak{b}_{B_i}}$ is a ring of type $ (n,1) $ such that $ B \subseteq \widehat{B} $, and since $ A \subseteq B $ a good embedding, it induces an embedding $ \widehat{A} \subseteq \widehat{B} $ making the obvious square commutes (cf. Remark \ref{R.RCSVR.good}). 

	By Lemma \ref{L.RCSVR.1_square}, there exist rings $ \widehat{A}' $ and $ \widehat{B}' $ of type $ (n,1) $ and embeddings $ \epsilon_{\widehat{A}} : \widehat{A} \linto  \widehat{A}'$ and $ \epsilon_{\widehat{B}} : \widehat{B} \linto  \widehat{B}'$ such that $ \widehat{A}' \subseteq \widehat{B}' $, $ \epsilon_{\widehat{A}}$ is sharp at $ \mathfrak{b}_{\widehat{A}'} \ (= \mathfrak{m}_{\widehat{A}'}) $, $ \epsilon_{\widehat{B}}$ is sharp at $ \mathfrak{b}_{\widehat{B}'} \ (= \mathfrak{m}_{\widehat{B}'}) $, $ \widehat{A}' $ and $ \widehat{B}' $ are homogeneous, and $ \epsilon_{\widehat{B}' \upharpoonright \widehat{A}'}= \epsilon_{\widehat{A}'} $; in particular, there exist non-trivial real closed valuation rings $ V $ and $ W $ with residue fields $ \text{qf}(C) $ and $ \text{qf}(D) $ (respectively)  such that $ V \subseteq W $ and making the diagram

\noindent\adjustbox{center}{
\begin{tikzpicture}[node distance=2cm, auto]
  \node (lb) {$C$};
  \node (lt) [above of=lb] {$\text{qf}(C)$};
  \node (rb) [right of=lb] {$D$};
  \node (rt) [above of=rb] {$\text{qf}(D)$};
  \node (llb) [left=2cm of lb] {$ A_i $};
  \node (llt) [above of = llb] {$ (A_i)_{\mathfrak{b}_{A_i}} $};
  \node (rrb) [right=2cm of rb] {$ B_i $};
  \node (rrt) [above of=rrb] {$ (B_i)_{\mathfrak{b}_{B_i}} $};
  \node (lmt) [above of=llt, right=0.3cm of llt] {$ V $};
   \node (rmt) [above of=rrt, left=0.3cm of rrt] {$ W $};
  \draw[{Hooks[right]}->] (lb) to node {} (lt);
  \draw[{Hooks[right]}->] (lb) to node [swap] {} (rb);
  \draw[{Hooks[right]}->] (rb) to node [swap] {} (rt);
  \draw[{Hooks[right]}->] (lt) to node {} (rt);
  \draw[{Hooks[right]}->] (llb) to node {} (llt);
 \draw[{Hooks[right]}->] (rrb) to node [swap] {} (rrt);
 \draw[{Hooks[right]}->] (llt) to node {$ \epsilon_{\widehat{A}, i} $} (lmt);
 \draw[{Hooks[right]}->] (rrt) to node [swap] {$ \epsilon_{\widehat{B}, i} $} (rmt);
 \draw[{Hooks[right]}->] (lmt) to node {} (rmt);
 \draw[->>] (llb) to node {} (lb);
 \draw[->>] (rrb) to node {} (rb);
\draw[->>] (llt) to node {} (lt);
 \draw[->>] (rrt) to node {} (rt);
\draw[->>] (lmt) to node  {$ \lambda_{V} $} (lt);
 \draw[->>] (rmt) to node [swap] {$ \lambda_W $} (rt);
\end{tikzpicture}
}
\noindent commute for all $ i \in [n] $, where $ \lambda_V $ and $ \lambda_W $ are the residue field maps and $ \epsilon_{\widehat{A}, i}  $ and $ \epsilon_{\widehat{B}, i}  $ are the embeddings induced by $ \epsilon_{\widehat{A}}  $ and $ \epsilon_{\widehat{B}}  $ (respectively). Define $ V':= \lambda_V^{-1}(C) $ and $ W':= \lambda_W^{-1}(D) $; then $ V'$ and $ W' $ are non-trivial real closed valuation rings (Lemma \ref{lem.rcr_domain_not_rcvr}), and using the fact that  $ \lambda_V^{-1}(C) \cong V \times_{\text{qf}(C)} C $ and $ \lambda_W^{-1}(D) \cong W \times_{\text{qf}(D)} D $, it follows from the universal property of the pullback that the diagram

\noindent\adjustbox{center}{
\begin{tikzpicture}[node distance=2cm, auto]
  \node (lt) {$C$};
  \node (rt) [right of=lt] {$D$};
  \node (llt) [left = 2cm of lt] {$ A_i $};
  \node (rrt) [right = 2cm of rt] {$  B_i $};
  \node (lmt) [above of=llt, right=0.6cm of llt] {$ V' $};
   \node (rmt) [above of=rrt, left=0.6cm of rrt] {$ W' $};
  \draw[{Hooks[right]}->] (lt) to node {} (rt);
 \draw[{Hooks[right]}->] (llt) to node {$ \epsilon_{A, i} $} (lmt);
 \draw[{Hooks[left]}->] (rrt) to node [swap] {$ \epsilon_{B, i} $} (rmt);
 \draw[{Hooks[right]}->] (lmt) to node {} (rmt);
\draw[->>] (llt) to node {} (lt);
 \draw[->>] (rrt) to node {} (rt);
\draw[->>] (lmt) to node  {} (lt);
 \draw[->>] (rmt) to node [swap] {} (rt);
\end{tikzpicture}
}

\noindent commutes for all $ i \in [n] $, where $ \epsilon_{A, i} $ and $ \epsilon_{B,i} $ are the restrictions of $ \epsilon_{\widehat{A}, i} $ and $ \epsilon_{\widehat{B}, i}  $ to $ A_i $ and $ B_i $ (respectively), and $  V' \subseteq W'$ is induced by $ V \subseteq W $. Define $ A':= {\prod}_{C}^n V' $, $ B':= {\prod}_{D}^n W' $, $ \epsilon_A: A \linto A' $ to be given by $ \epsilon_A(a):= (\epsilon_{A,1}(a_1), \dots, \epsilon_{A,n}(a_n)) $, and $ \epsilon_B: B \linto B' $ to be given by $ \epsilon_B(b):= (\epsilon_{B,1}(b_1), \dots, \epsilon_{B,n}(b_n)) $;  it is clear by construction that this choice of data satisfies the requirements in the statement of the lemma.
\end{proof}

%% file: model_theory.tex
Throughout this section fix $ n \in \N^{\geq 2} $ and let $ \mathscr{L}:= \{ +, -, \cdot, 0, 1 \}  $ be the language of rings; unless stated otherwise, all model-theoretic statements are assumed to be phrased with respect to the language $ \mathscr{L} $.  If $S_1, \dots, S_m$ are new predicate or function symbols, write $ \mathscr{L}(S_1, \dots, S_m):= \mathscr{L} \ \dot{\cup} \ \{S_1, \dots, S_m\} $, and if an $ \mathscr{L} $-structure $ A $ can be expanded to an $  \mathscr{L}(S_1, \dots, S_m) $-structure, write  $ A(S_1(A), \dots, S_m(A)) $ for the resulting expansion.

\subsection{\texorpdfstring{The theories $ T_n $, $ T_{n,1} $, and $ T_{n,2} $}{The theories Tₙ, Tₙ,₁ and Tₙ,₂}}

\begin{lemma}\label{L.MOD_TH.rk_n_first-order}
		There exists an $ \mathscr{L} $-sentence $ \phi_{\emph{rk}=n} $ such that for all reduced local rings $ A $, $ A \models \phi_{\emph{rk}= n} $ if and only if $ A $ has rank $ n $. 
\end{lemma}

\begin{proof}
By Lemma \ref{L.SV.ann_intersection_of_min_prime_id.ii.a} one may define $ \phi_{\text{rk}=n} $ to be the $ \mathscr{L} $-sentence expressing the following statement about $ A $: \enquote{there exist non-zero pairwise orthogonal  elements $ a_1, \dots, a_n \in A $ such that if  $ b\in A $ is a non-zero element distinct from all the $ a_i $, then  $ b $ is not orthogonal to some $ a_i$}.
\end{proof}

\begin{lemma}\label{L.MOD_TH.SV_first-order}
	There exists an $ \mathscr{L} $-sentence $ \phi_{\emph{SV},n} $ such that for all reduced local rings $ A $ of rank $ n  $, $ A \models \phi_{\emph{SV}, n} $ if and only if $ A $ is an SV-ring; moreover, $ \phi_{\emph{SV},n} $ can be chosen to be a universal sentence in the language $ \mathscr{L}(\emph{div}) $, where $ \emph{div} $ is a binary predicate interpreted as divisibility.
\end{lemma}

\begin{proof}
	Let $ a_1, \dots, a_n \in A$ be any non-zero pairwise orthogonal elements, so that $ \text{Spec}^{\text{min}}(A)= \{ \text{Ann}_A(a_i) \mid i \in [n] \}  $ by Lemma \ref{L.SV.ann_intersection_of_min_prime_id.ii.b}. By the implication (i) $ \Ra $  (ii) in Theorem \ref{T.SV.equiv_SV-ring}, $ A  $ is an SV-ring if and only if for all $ b, c \in A $ and for all $ i \in [n] $,  $ b/\text{Ann}(a_i) $ divides  $ c/\text{Ann}(a_i) $ or $ c/\text{Ann}(a_i) $  divides $ b/\text{Ann}(a_i) $ in $ A/\text{Ann}_A(a_i) $; moreover, 
	\[	A/\text{Ann}_A(a_i)  \models \text{div}( b/\text{Ann}(a_i) ,  c/\text{Ann}(a_i) )\iff A \models \exists x [(bx-c)a_i = 0] \iff  A \models \text{div}(ba_i, c a_i)
	,\] therefore one may define $ \phi_{\text{SV},n} $ to be the $ \mathscr{L} $-sentence expressing the following statement about $ A  $: \enquote{for all non-zero pairwise orthogonal  $ a_1, \dots, a_n\in A $ and for all  $ b, c \in A $, either  $ ba_i $ divides $ ca_i $ or $ ca_i $ divides $ ba_i $}.	
\end{proof}

\begin{definition}\label{D.MOD_TH.T_n}
	Let $ T_n $ to be the $ \mathscr{L} $-theory of local real closed SV-rings of rank $ n $; explicitly, an axiomatization for $ T_n $ consists of the $ \mathscr{L} $-axioms for local real closed rings (\cite[5, 9]{prestel.schwartz/mod_th_rcr}) together with the sentences $ \phi_{\text{rk}=n} $ and $ \phi_{\text{SV},n} $ defined in Lemmas \ref{L.MOD_TH.rk_n_first-order} and \ref{L.MOD_TH.SV_first-order}, respectively.
\end{definition}

\begin{definition}\label{D.MODTH.phi_br,n}
Let $ \phi_{\text{br}, n} $ to be the $ \mathscr{L}$-sentence expressing the following statement about a ring $ A $: \enquote{$ \text{Ann}(a_i)+ \text{Ann}(a_j) = \text{Ann}(a_k)+ \text{Ann}(a_{\ell})$ for all pairwise orthogonal non-zero elements $ a_1, \dots, a_n \in A $ and for all $ i, j , k ,\ell \in [n] $ such that $ i \neq j $ and $ k \neq \ell $.}
\end{definition}

\begin{lemma}\label{L.MODTH.phi_br,n}
The following are equivalent for all local real closed rings  $ A $ of rank $ n$:
\begin{enumerate}[\normalfont(i)]
\item 	$ A \models \phi_{\emph{br}, n}$.
\item  $ A  $ has exactly one branching ideal.
\end{enumerate}
\end{lemma}

\begin{proof}
By the equivalence (i) $ \Leftrightarrow $  (ii) in Lemma \ref{L.RCSVR.equiv_loc_rcsvr_rk_n_exactly_one_branching_ideal} together with Lemma \ref{L.SV.ann_intersection_of_min_prime_id.ii.b}.
\end{proof}

\begin{definition}
\begin{enumerate}[\normalfont(I)]
	\item Let $ T_{n,1} $ be the $ \mathscr{L} $-theory $ T_n $ together with  $ \phi_{\text{br},n} $ and the $ \mathscr{L} $-sentence expressing the following statement about a ring: \enquote{every unit is a sum of two zero divisors}.
\item Let $ T_{n,2} $ be the $ \mathscr{L} $-theory  $ T_n $ together with $ \phi_{\text{br},n} $ and the $ \mathscr{L} $-sentence expressing the following statement about a ring: \enquote{there exists a unit which is not a sum of two zero divisors}.
\end{enumerate}
\end{definition}

\begin{lemma}
$ A \models T_{n,1} $ if and only if $ A $ is a ring of type $ (n, 1) $ and 	$ A \models T_{n,2} $ if and only if $ A $ is a ring of type $ (n, 2) $.
\end{lemma}

\begin{proof}
	By Lemma \ref{L.MODTH.phi_br,n}	together with the equivalence (i) $ \Leftrightarrow $  (v) in Proposition \ref{P.RCSVR.equiv_branching_max_id}.
\end{proof}

\begin{remark}\label{R.MODTH.def_branch_id}
	\begin{enumerate}[\normalfont(i), ref=\ref{R.MODTH.def_branch_id} (\roman*)]
		\item\label{R.MODTH.def_branch_id.i} 	By \cite[5, 9]{prestel.schwartz/mod_th_rcr}, the $ \mathscr{L} $-theory of real closed rings has a recursive axiomatization; in particular, both of the theories $ T_{n,1} $ and $ T_{n,2} $ also have recursive axiomatizations. 
		\item	By Lemma \ref{L.SV.ann_intersection_of_min_prime_id.ii.b}, every minimal prime ideal is parametrically definable in models of $ T_{n,1} $ and in models of $ T_{n,2} $, but minimal prime ideals cannot be defined without parameters in models of either of these theories: for instance,  if $ V $ is a non-trivial real closed valuation ring with residue field $ \emph{\textbf{k}} $, then $ V\times_{\emph{\textbf{k}}} V \models T_{2,1} $ and the map $ V\times_{\emph{\textbf{k}}} V \lra  V\times_{\emph{\textbf{k}}} V $ given by $ (a_1, a_2) \mapsto (a_2, a_1) $ is an automorphism which swaps the two minimal prime ideals.
		\item\label{R.MODTH.def_branch_id.ii} If $ A \models T_{n,1} $, then $ \mathfrak{b}_A = \mathfrak{m}_A $ and thus the branching ideal $ \mathfrak{b}_A $ is definable without parameters. If  $ A \models T_{n,2} $, then $ \mathfrak{b}_A \subsetneq \mathfrak{m}_A $ and the branching ideal $ \mathfrak{b}_A  $ is also definable  without parameters by the formula expressing the following about an element $ a \in A $: \enquote{$ a \in \text{Ann}(a_i)+\text{Ann}(a_j) $ for all non-zero pairwise orthogonal   $ a_1,\dots, a_n \in A$ and for all $ i , j \in [n] $ such that $ i \neq j $} (see  Lemma \ref{L.SV.ann_intersection_of_min_prime_id.ii.b} and Lemma \ref{L.RCSVR.equiv_loc_rcsvr_rk_n_exactly_one_branching_ideal.a}).
	\end{enumerate}
\end{remark}

\subsection{Model completeness}\label{SUBSEC.mod_compl}

The core of this section consists of Theorems \ref{T.MODTH.T_1n_mod_compl} and \ref{T.MODTH.T_2n_mod_compl}. The first key idea for these model completeness results is that if $ A \subseteq B $ is an embedding  (resp. a good embedding, see Definition \ref{D.RCSVR.good}) of rings of type $ (n, 1) $ (resp. of rings of type $ (n,2) $), then the corresponding embeddings $ A_i \subseteq B_i $  of non-trivial real closed valuation rings (see Convention \ref{C.RCSVR.factors}) are elementary in suitable expansions of the language of rings (Proposition \ref{P.MODTH.RCVR_mod_compl}); the second key idea is using the first key idea together with Lemma \ref{L.MODTH.qfformula_factors} to show that under some hypotheses, $ A $ is existentially closed in $ B $  (Lemma \ref{L.MODTH.1_embedd_imm} and Lemma \ref{L.MODTH.2_embedd_imm}), and then put everything together using Lemma \ref{L.RCSVR.1_square} and Lemma \ref{L.RCSVR.2_square}. First, some preliminaries are needed:

\begin{lemma}\label{L.MOD_TH.embedd_rcvr}
If $ V  $ and  $ W $ are valuation rings such that $ V \subseteq W $, then $ V \subseteq W $ as  $ \mathscr{L}(\emph{\texttt{m}}) $-structures if and only if $ V \subseteq W $ as $ \mathscr{L}( \emph{div}) $-structures, where $ \emph{\texttt{m}} $ is a unary predicate interpreted as the maximal ideal and $ \emph{div} $ is a binary  predicate interpreted as the divisibility relation.
\end{lemma}

\begin{proof}
	If $V \subseteq W $ as $ \mathscr{L}(\text{div}) $-structures and $ a \in \mathfrak{m}_V $, then $ V \models \neg \text{div}(a, 1) $, hence $ W \models \neg \text{div}(a, 1) $ and thus $ V \cap \mathfrak{m}_W = \mathfrak{m}_V $. Suppose now that $ V \subseteq W $ as $ \mathscr{L}(\texttt{m}) $-structures and let $ a, b \in V $ be such that  $ W \models \text{div}(a, b) $, so that there exists $ c \in W $ such that  $ ac =b $, and assume for contradiction that  $ V \models \neg \text{div}(a,b) $; in particular, $ b \neq 0 $ and $ c \notin V $. Since $ V  $ is a valuation ring, $ V \models \text{div}(b,a) $, therefore there exists $ d\in V $ such that  $ bd=a $; in particular,  $ ac=bdc=b $, hence  $b(1-dc)=0  $, therefore $ 1=dc $  and thus  $ c^{-1}= d\in V $, so $ c^{-1} \in \mathfrak{m}_V = \mathfrak{m}_W \cap V$, a contradiction to $ c, c^{-1}\in W $.
\end{proof}

\begin{remark}\label{R.MODTH.div_embedd_local}
	The proof of Lemma	\ref{L.MOD_TH.embedd_rcvr} also shows that if $ A $ and $ B $ are local rings such that $ A \subseteq B $ as  $ \mathscr{L}(\text{div}) $-structures, then $ A\subseteq B $ as  $ \mathscr{L}(\texttt{m}) $-structures, i.e., $ A\subseteq B $ is a local embedding.
\end{remark}

In \cite{cherlin/dickmann.rcrII} the model theory of non-trivial real closed rings is studied in the languages $ \mathscr{L}(\leq) $ and $ \mathscr{L}(\leq, \text{div}) $; in particular, since the class of non-trivial real closed valuation rings is elementary in the language $ \mathscr{L}(\leq) $ and the total order  $ x \leq y $ in structures of this class is defined by the existential formula $ \exists z [z= (y-x)^2] $ (Theorem \ref{T.RCSVR.equiv_rcvr}), the class of non-trivial real closed valuation rings is also elementary in the language of rings.

\begin{proposition}\label{P.MODTH.RCVR_mod_compl}
Let $ \sf{RCVR} $ be the $ \mathscr{L} $-theory of non-trivial real closed valuation rings, and $ \emph{\texttt{p}} $ and $ \emph{\texttt{m}} $  be unary predicate symbols.
\begin{enumerate}[\normalfont(i), ref=\ref{P.MODTH.RCVR_mod_compl} (\roman*)]
	\item\label{P.MODTH.RCVR_mod_compl.i} Define 	$ \sf{RCVR}(\emph{\texttt{m}}) $ to be the $ \mathscr{L}(\emph{\texttt{m}}) $-theory $ \sf{RCVR} $ together with the sentence expressing  that $ \emph{\texttt{m}} $ is the set of non-units. Then $ \sf{RCVR}(\emph{\texttt{m}}) $ is complete and model complete.
	\item\label{P.MODTH.RCVR_mod_compl.ii} Define 	$ \sf{RCVR}(\emph{\texttt{b}},\emph{\texttt{m}}) $ to be the $ \mathscr{L}(\emph{\texttt{b}}, \emph{\texttt{m}}) $-theory $ \sf{RCVR}(\emph{\texttt{m}}) $ together with the sentence expressing that $ \emph{\texttt{b}} $ is a non-zero prime ideal properly contained in $ \emph{\texttt{m}} $. Then $ \sf{RCVR}(\emph{\texttt{b}}, \emph{\texttt{m}}) $ is complete and model complete.
\end{enumerate}
\end{proposition}

\begin{proof}
	(i) follows from \cite[Theorems 4A and 4B]{cherlin/dickmann.rcrII} together with Lemma \ref{L.MOD_TH.embedd_rcvr}; (ii) follows from \cite[Corollary 6.3]{tressl.heirs}, since every model $ (V, \mathfrak{b}, \mathfrak{m}_V) \models \sf{RCVR}(\texttt{b},\texttt{m}) $ is definable in  $ (\text{qf}(V), V, V_{\mathfrak{b}}) \models \sf{RCF}_{\text{convex}, 2} $.
\end{proof}

\begin{lemma}\label{L.MODTH.amalg_rcvrs}
	Let $ V_1 $ and $ V_2  $ be non-trivial real closed valuation rings regarded as $ \mathscr{L}(\leq, \emph{\texttt{m}}) $-structures, where $ \leq  $ is a binary predicate interpreted as the total order relation and $ \emph{\texttt{m}} $ is a unary predicate interpreted as the maximal ideal. If $ A $ is a totally ordered valuation ring such that $A\subseteq V_1, V_2  $ as  $ \mathscr{L}(\leq, \emph{\texttt{m}}) $-structures, then there exists a non-trivial real closed valuation ring $W  $ amalgamating $ V_1 $ and $ V_2 $ over $ A  $ as $ \mathscr{L}(\leq, \emph{\texttt{m}}) $-structures.
\end{lemma}

\begin{proof}
	By Lemma \ref{L.MOD_TH.embedd_rcvr}, $ A \subseteq V_1, V_2 $ as  $ \mathscr{L}(\leq, \text{div}) $-structures, therefore by the Cherlin-Dickmann theorem \cite[Section 2]{cherlin/dickmann.rcrII} and \cite[Proposition 3.5.19]{chang/keisler.model_theory} there exists a non-trivial real closed valuation ring $ W $ amalgamating $ V_1 $ and $ V_2 $ over $ A  $ as $ \mathscr{L}(\leq, \text{div}) $-structures; conclude by appealing to Lemma \ref{L.MOD_TH.embedd_rcvr} again.
\end{proof}

\begin{notation}\label{N.MODTH.tuple_factors}
	Let $ \{ A_i \}_{i \in I} $ be a non-empty family of rings such that there exists a ring $ B $ and surjective ring homomorphisms $ f_i : A_i \lonto B $ for all  $ i \in I $, and set $ A:= {\prod}_{B, i \in I}A_i $ (see Notation \ref{N.RCSVR.I-fold_fibr_prod}); in particular, $ A  $ is a subdirect product of $ \prod_{i \in I}A_i $ with canonical projection maps $ p_i : A \lonto A_i $ for all $ i \in I $.
	\begin{enumerate}[\normalfont(i)]
		\item If $ a\in A $, then write  $ a_i := p_i(a) $, and if  $ r \in \N^{\geq 2} $ and  $ \overline{a} \in A^r $, then write $ \overline{a}_i:= (a_{1 i}, \dots, a_{r i}) \in A_i^r$ for all $ i \in I $.
		\item If $ F(\overline{x}) \in A[\overline{x}] $, then write $ F_i(\overline{x}) $ for the polynomial in $ A_i[\overline{x}] $ obtained by replacing each coefficient  $ a \in A  $ appearing in $ F(\overline{x}) $ by $ a_i $.
	\end{enumerate}
\end{notation}

\begin{remark}\label{R.MODTH.formula_factors}
	Let $ \{ A_i \}_{i \in I} $ and $ B $ be as in 	Notation \ref{N.MODTH.tuple_factors} and set $ A:=  {\prod}_{B, i \in I}A_i $. If $ F(\overline{x}) \in A[x_1, \dots, x_r] $ and $ \overline{a} \in  A^r$, then $ F(\overline{a})_i = F_i(\overline{a}_i) $ for all $ i \in I $, therefore $ A \models F(\overline{a}) =0$ if and only if $A_i\models F_i(\overline{a}_i) =0$ for all $ i \in I $.
\end{remark}

\begin{lemma}\label{L.MODTH.qfformula_factors}	
	Let $ \{ A_i \}_{i \in I} $ and $ B $ be as in 	\emph{Notation \ref{N.MODTH.tuple_factors}} and set $ A:=  {\prod}_{B, i \in I}A_i $. Let $ \phi(x_1, \dots, x_r) $ be a quantifier-free $ \mathscr{L}$-formula with parameters from $S \subseteq A $, and $ \overline{a} \in A^r $ be such that $ A \models \phi(\overline{a}) $. There exist quantifier-free $ \mathscr{L}$-formulas $ \phi_{\overline{a}, i}(x_1, \dots, x_r) $ with parameters from $ S_i:= p_{i} (S) \subseteq A_i $ $( i \in I )$ such that 
\begin{enumerate}[\normalfont(i)]
\item $  A \models \phi(\overline{a})$ if and only if $ A_i \models \phi_{\overline{a}, i} (\overline{a}_i) $ for all $ i \in I $, and
\item if  $ \overline{a'} \in A^r $ is such that $ A_i \models \phi_{\overline{a}, i} (\overline{a'}_i) $ for all $ i \in I $, then $ A \models \phi(\overline{a'}) $.
\end{enumerate}
\end{lemma}

\begin{proof}
	Since $ \phi(x_1, \dots, x_r) $ is quantifier-free and $ A \models \phi(\overline{a}) $, it can be assumed that $ \phi $  is of the form
	\[	
		\bdwedge_{\lambda \in \Lambda} F^+_{\lambda}(\overline{x})= 0   \ \& \ F^-_{\lambda}(\overline{x})\neq 0   
	,\] 
	\noindent where $ \Lambda $ is a finite index set   and $ F_{\lambda}^{\pm}\in S[x_1, \dots, x_r] $ for all $ \lambda \in \Lambda $.    For each $ i \in  I $,  define  $ \phi_{\overline{a}, i}(x_1, \dots, x_r) $ to be the $ \mathscr{L}$-formula (with parameters from $ S_i $)
\[
	\bdwedge_{\lambda \in \Lambda} F^+_{\lambda i}(\overline{x})= 0   \ \& \  \bdwedge_{\lambda \in \Lambda}\{ F^-_{\lambda i}(\overline{x})\neq 0 \mid A_i \models F^-_{\lambda i}(\overline{a}_i) \neq 0 \}   
;\] note that  for each $ i \in I $, if $ A_i \models F^-_{\lambda i}(\overline{a}_i) = 0 $ for all $ \lambda \in \Lambda $, then $ \phi_{\overline{a}, i}(\overline{x})$ is logically equivalent to $ \bdwedge_{\lambda \in \Lambda} F^+_{\lambda i}(\overline{x})= 0   $. Items (i) and (ii) in the statement of the lemma now follow by Remark \ref{R.MODTH.formula_factors} and by construction of the formulas $ \phi_{\overline{a}, i}(x_1, \dots, x_r) $.
\end{proof}

\begin{remark}\label{rem.fmlas_in_rcsvr_and_factors_R_1} 
	The converse of item (ii) in Lemma \ref{L.MODTH.qfformula_factors} doesn't hold in general. For example, let  $ A:= A_1 \times_{\emph{\textbf{k}}}  A_2 $, where $ A_1$ and $ A_2 $ are non-trivial local domains with residue field $ \emph{\textbf{k}} $,  $ \phi(x) $ be the $ \mathscr{L} $-formula $ cx \neq0 $ with parameter $ c $, where  $ c:=(c_1, c_2)\in A \subseteq A_1\times A_2 $ is  such that $ c_i \in \mathfrak{m}_{A_i} \setminus \{0\}  $, and  define $ a := (c_1, 0) \in A $; clearly $ A \models \phi(a) $. By the construction in the proof of Lemma  \ref{L.MODTH.qfformula_factors}, $ \phi_{a,1}(x) $ is the $ \mathscr{L}$-formula $ c_1x \neq 0 $ with parameter  $ c_1 $,  $ \phi_{a,2}(x)  $ is an empty conjunct (hence logically equivalent to $ x = x $), and clearly $ A \models \phi(a) $ if and only if $ A_1 \models \phi_{a, 1}(a_1) $ and $ A_2 \models \phi_{a, 1}(a_2) $; but $ A \models \phi(b) $ and $ A_1 \nvDash \phi_{a,1}(b_1) $ for $ b := (0, c_2)\in A $.
\end{remark}

\subsubsection{\texorpdfstring{Model completeness for $ T_{n,1} $}{Model completeness for Tₙ,₁}}

For the next lemma recall that any embedding $ A \subseteq B $ of models of $ T_{n,1} $ is a local embedding and it induces local embeddings $ A_i \subseteq B_i $ for all  $ i \in [n] $, see Convention \ref{C.RCSVR.factors} and Lemma \ref{L.RCSVR.embedd_1n_local_factors}; moreover, if $ A \subseteq B $ is sharp at $ \mathfrak{b}_B \ (= \mathfrak{m}_B) $, then $(A/\mathfrak{m}_A\cong)\ A_i/\mathfrak{m}_{A_i} \cong B_i/\mathfrak{m}_{B_i} \ (\cong B/\mathfrak{m}_B) $ for all $ i \in [n] $, see Definition  \ref{D.RCSVR.sharp_embedd}.

\begin{lemma}\label{L.MODTH.1_embedd_imm}
	Let $ A, B \models T_{n,1} $ be such that $ A \subseteq B $.
	\begin{enumerate}[\normalfont(I), ref=\ref{L.MODTH.1_embedd_imm} (\Roman*)]
\item\label{L.MODTH.1_embedd_imm.I} 		If $ A \subseteq B $ is sharp at  $ \mathfrak{b}_B \ (= \mathfrak{m}_B)$, then $ A $ is existentially closed in $ B $. 
\item\label{L.MODTH.1_embedd_imm.II} If both $ A $ and $ B $ are homogeneous \emph{(Definition \ref{D.RCVSR.homogeneous})}, then $ A $ is existentially closed in $ B $.
\end{enumerate}
\end{lemma}

\begin{proof}
	Let $ \phi(x_1,\dots, x_r) $ be a quantifier-free $ \mathscr{L} $-formula with parameters from $ A$ and let  $ \overline{b} \in B^r $ be such that $ B \models \phi(\overline{b}) $; by Lemma \ref{L.MODTH.qfformula_factors} there exist quantifier-free $ \mathscr{L}$-formulas $ \phi_{\overline{b}, i}(x_1, \dots, x_r)$ with parameters from $  A_i  $ $ (i \in [n]) $ such that  

\begin{enumerate}[(i)]
\item $  B \models \phi(\overline{b})$ if and only if $ B_i\models \phi_{\overline{b}, i} (\overline{b}_i) $ for all $ i \in [n] $, and
\item if $ \overline{a} \in B^r $ is such that $ B_i \models \phi_{\overline{b}, i} (\overline{a}_i) $ for all $ i \in [n] $, then $ B \models \phi(\overline{a}) $.
\end{enumerate}

	(I). For each $ i \in [n] $ and  each  $ j \in [r] $, pick  $ c_{ji}\in A_i $ such that $ c_{ji}/\mathfrak{m}_{A_i}  $ is the image of $ b_{ji}/\mathfrak{m}_{B_i} $ under the  isomorphism $A_i/\mathfrak{m}_{A_i}   \cong B_i/\mathfrak{m}_{B_i}$; since $ A, B \models T_{n,1} $ and $ \overline{b} \in B^r $, it follows by choice of $ c_{ji}\in A_i $ that $ c_j:=(c_{j1},\dots, c_{jn}) \in A$ for all $ j \in [r] $,  and thus $ \overline{c}:= (c_1, \dots, c_r) \in A^r $. Again by choice of $ c_{j i} \in A_i $ and by item (i) above,
\begin{align*}
	(B_i, \mathfrak{m}_{B_i}) \models \phi_{\overline{b}, i}(\overline{b}_i) \ \& \ \bdwedge_{j \in [r]}\texttt{m}(b_{ji}-c_{ji});
\end{align*}
 since  $ A_i \subseteq B_i $ is a local embedding for all $ i \in [n] $, by Proposition \ref{P.MODTH.RCVR_mod_compl.i} there exist $ a_{ji}  \in A_i $ ($ j \in [r] $) such that 
\begin{align*}
		(A_i, \mathfrak{m}_{A_i}) \models \phi_{\overline{b}, i}(\overline{a}_i) \ \& \ 
&\bdwedge_{j \in [r]}\texttt{m}(a_{ji}-c_{ji})
\end{align*}
for all $ i \in [n] $, where $ \overline{a}_i:= (a_{1i}, \dots a_{ri}) $. Once again by choice of $ c_{ji} \in A_i $, $ a_j := (a_{j1},\dots, a_{jn}) \in A$ for all $ j \in [r] $, and thus $ \overline{a}:=(a_1, \dots, a_r)\in A^r \subseteq B^r$; since $ A_i \models \phi_{\overline{b}, i}(\overline{a}_i) $ and $ \phi_{\overline{b}, i}(\overline{x}) $ is quantifier-free, $ B_i \models \phi_{\overline{b}, i}(\overline{a_i}) $ for $ i \in [n] $, so by item (ii) above it follows that $ B \models \phi(\overline{a}) $, and since $ A $ is a substructure of  $  B$ and $ \phi(\overline{x}) $ is quantifier-free, $ A \models \phi(\overline{a}) $ follows.

(II). Since $ A $ and $ B $ are homogeneous, it can be assumed that there exist  non-trivial real closed valuation rings $ V  $ and $ W $ such that $ A_i = V $ and  $ B_i= W $ for all  $ i \in [n] $; in particular, $ b_{ji} \in W $ for all $ i \in [n] $ and $ j \in [r] $.  By item (i) above,  
 \begin{align*}
	 (W, \mathfrak{m}_W)  \models \bdwedge_{i \in [n]}\phi_{\overline{b}, i}(b_{1i},\dots, b_{ri}) \ \& \ \bdwedge_{j \in [r], \ i, i' \in [n]}\texttt{m}(b_{ji}-b_{ji'}),
 \end{align*}
\noindent  and since  $ V \subseteq W$ is a local embedding, by Proposition \ref{P.MODTH.RCVR_mod_compl.i} there exist $ a_{ji}  \in V$ ($ i \in [n] $, $ j \in [r] $) such that 
\begin{align*}
	 (V, \mathfrak{m}_V)  \models \bdwedge_{i \in [n]}\phi_{\overline{b}, i}(a_{1i},\dots, a_{ri}) \ \& \ \bdwedge_{j \in [r], \ i, i' \in [n]}\texttt{m}(a_{ji}-a_{ji'}). 
\end{align*}
 It follows that for each  $j \in [r] $,  $ a_j:= (a_{j1}, \dots, a_{jn}) \in A$ and thus $ \overline{a}:= (a_1, \dots, a_r)\in A^r \subseteq B^r $; since $ V \models \phi_{\overline{b}, i}(\overline{a}_i) $ and $ \phi_{\overline{b}, i}(\overline{x}) $ is quantifier-free, $ W \models \phi_{\overline{b}, i}(\overline{a}_i) $ for $ i \in [n] $, so by item (ii) above it follows that $ B \models \phi(\overline{a}) $, and since $ A $ is a substructure of  $  B$ and $ \phi(\overline{x}) $ is quantifier-free, $ A \models \phi(\overline{a}) $ follows.
\end{proof}

\begin{remark}
	The proof of Lemma \ref{L.MODTH.1_embedd_imm.I} can be used \textit{mutatis mutandis} to show that given arbitary collections of non-trivial real closed valuation rings  $  \{ V_i \}_{i \in I}  $ and $ \{ W_i \}_{i \in I} $ with $ V_i/\mathfrak{m}_{V_i} \cong W_i/\mathfrak{m}_{W_i} =: \emph{\textbf{k}} $  and $ V_i \subseteq W_i $ for all  $ i \in I $, then $ {\prod}_{\emph{\textbf{k}}, i \in I}V_i $ is existentially closed in $ {\prod}_{\emph{\textbf{k}}, i \in I}W_i $.
\end{remark}

\begin{theorem}\label{T.MODTH.T_1n_mod_compl}
	$ T_{n,1} $ is model complete.
\end{theorem}

\begin{proof}
	Combine Lemmas \ref{L.MODTH.1_embedd_imm} and \ref{L.RCSVR.1_square}.
\end{proof}

\subsubsection{\texorpdfstring{Model completeness for $ T_{n,2} $}{Model completeness for Tₙ,₂}}

The main difference between embeddings of models of $ T_{n,1} $ and embeddings of models of $ T_{n,2} $ is that every embedding $ A \subseteq B $ of models of $ T_{n,1} $ is local (hence also $ \mathfrak{b}_B \cap A = \mathfrak{b}_A $), but this is not the case for models of $ T_{n,2} $ (i.e., not every embedding of models of $ T_{n,2} $ is a good embedding, Example \ref{E.RCSVR.non_branch_pres_embedd_type_2,2} and Definition \ref{D.RCSVR.good}); this fact, together with the next lemma implies that $ T_{n,2} $ is not model complete in the language of rings:

\begin{lemma}\label{L.MODTH.ec_embedd_implies_banch}
	Let $ A $ and $ B $ be local real closed rings of rank $ n  $ such that $ A \subseteq B $. If $ A$ is existentially closed in $ B $, then the embedding $ A \subseteq B $ is local and $ \mathfrak{q} \cap A $ is a branching ideal of $ A $ for every   branching ideal $\mathfrak{q} \in \emph{Spec}(B)  $.
\end{lemma}

\begin{proof}
That $ A\subseteq B $ must be a local embedding is clear. Let now $ a_1, \dots, a_n \in A $ be non-zero pairwise orthogonal elements, so that  $ \text{Spec}^{\text{min}}(A) = \{ \text{Ann}_A(a_i) \mid i \in [n]\}  $ and $ \text{Spec}^{\text{min}}(B) = \{ \text{Ann}_B(a_i) \mid i \in [n]\}  $ (Corollary \ref{C.SV.min_prime_ann}). Pick a branching ideal $ \mathfrak{q}  \in \text{Spec}(B) $; by Remark \ref{R.RCSVR.branch_1} there exist $i, j \in [n]  $ with $ i \neq j  $ such that $ \mathfrak{q} = \text{Ann}_B(a_i) + \text{Ann}_B(a_j) $, therefore $ \text{Ann}_A(a_i) + \text{Ann}_A(a_j) \subseteq \mathfrak{q} \cap A $. Pick now $ b \in  \mathfrak{q} \cap A$; then 	
	\[	B \models \exists x y [xa_i=0 \ \& \ ya_j=0 \ \& \ b = x+y],
	\] and since $ a_i, a_j, b \in A $ and  $ A $ is existentially closed in $ B $, it follows that 
	\[
	A \models \exists  x y [xa_i=0 \ \& \ ya_j=0 \ \& \ b = x+y],
\] therefore $ b \in  \text{Ann}_A(a_i) + \text{Ann}_A(a_j) $ and thus $ \mathfrak{q} \cap A = \text{Ann}_A(a_i) + \text{Ann}_A(a_j) \in \text{Spec}(A)$ is a branching ideal.
\end{proof}

Example \ref{E.RCSVR.non_branch_pres_embedd_type_2,2} shows that there exist embeddings of models of $ T_{n, 2} $ which are not local, and also that there exist embeddings of models of $ T_{n,2} $ which don't map the branching ideal to the branching ideal; in view of Lemma \ref{L.MODTH.ec_embedd_implies_banch}, to obtain model completeness for $ T_{n,2} $ one must enlarge the language of rings is such a way that every embedding of models of $ T_{n,2} $ in the resulting language is a good embedding.

\begin{definition}
\begin{enumerate}[\normalfont(i)]
	\item  Let $  \texttt{b}$ and $ \texttt{m}$ are two unary predicates and define $ \mathscr{L}^*:= \mathscr{L} ( \texttt{b}, \texttt{m} )  $.
	\item Let $ T^*_{n,2} $ be the $ \mathscr{L}^* $-theory $ T_{n,2} $ together with the sentence expressing that $ \texttt{b} $ and $ \texttt{m}$ are interpreted as the branching ideal (Remark \ref{R.MODTH.def_branch_id.ii}) and as the set of non-units, respectively.
\end{enumerate}
\end{definition}

\begin{remark}\label{R.MODTH.factors}
	Let  $ A:=  {\prod}_{C, i \in [n]}A_i \models T^*_{n,2}$, so that $ A_1, \dots, A_n$ and $ C $ are non-trivial real closed valuation rings  such that for each $ i \in [n] $ there exists a surjective ring homomorphism $ A_i \lonto C $ onto a non-trivial real closed valuation ring $ C $.
	\begin{enumerate}[\normalfont(i), ref=\ref{R.MODTH.factors} (\roman*)]
		\item 	For each $ i \in [n] $, $ A_i $ is regarded as an $ \mathscr{L}^* $-structure in the canonical way, that is, $ \texttt{b}(A_i) = \mathfrak{b}_{A_i}:=\text{ker}(A_i \onto C) $ and $ \texttt{m}(A_i) := \mathfrak{m}_{A_i} $; in particular, the projection map $ A \lonto A_i $ is an  $ \mathscr{L}^* $-homomorphism.
		\item\label{R.MODTH.factors.ii} Note that $ \mathfrak{b}_A = \mathfrak{b}_{A_1} \times \dots \times \mathfrak{b}_{A_n}$ and $ \mathfrak{m}_A= \mathfrak{m}_{A_1} \times \dots \times \mathfrak{m}_{A_n}$ when $ \mathfrak{b}_A $ and $ \mathfrak{m}_A $ are regarded as subsets of $ \prod_{i=1}^n A_i $; in particular, if $ F(\overline{x}) \in A[x_1, \dots, x_r] $ and $ \overline{a} \in  A^r$, then $ A \models \texttt{b}(F(\overline{a})) $ if and only if $A_i\models \texttt{b}(F_i(\overline{a}_i) )$ for all $ i \in [n] $, and $ A \models \texttt{m}(F(\overline{a})) $ if and only if $A_i\models \texttt{m}(F_i(\overline{a}_i) )$ for all $ i \in [n] $ (cf. Remark \ref{R.MODTH.formula_factors}).
	\end{enumerate}
\end{remark}

\begin{lemma}\label{L.MODTH.2_qfformula_factors}	
	Let  $ A:=  {\prod}_{C, i \in [n]}A_i \models T^*_{n,2}$ and  $ \phi(x_1, \dots, x_r) $ be a quantifier-free $ \mathscr{L}^*$-formula with parameters from $S \subseteq A $, and $ \overline{a} \in A^r $ be such that $ A \models \phi(\overline{a}) $. There exist quantifier-free $ \mathscr{L}^*$-formulas $ \phi_{\overline{a}, i}(x_1, \dots, x_r) $ with parameters from $ S_i:= p_{i} (S) \subseteq A_i $ $( i \in [n] )$ such that 
\begin{enumerate}[\normalfont(i)]
\item $  A \models \phi(\overline{a})$ if and only if $ A_i \models \phi_{\overline{a}, i} (\overline{a}_i) $ for all $ i \in [n] $, and
\item if  $ \overline{b} \in A^r $ is such that $ A_i \models \phi_{\overline{a}, i} (\overline{b}_i) $ for all $ i \in [n] $, then $ A \models \phi(\overline{b}) $.
\end{enumerate}
\end{lemma}

\begin{proof}
Analogous to the proof of Lemma \ref{L.MODTH.qfformula_factors}	using Remark \ref{R.MODTH.factors.ii}.
\end{proof}

\begin{lemma}\label{L.MODTH.2_embedd_imm}
Let $ A, B \models T^*_{n,2} $ be such that $ A \subseteq B $ as $ \mathscr{L}^* $-structures.
\begin{enumerate}[\normalfont(I), ref=\ref{L.MODTH.2_embedd_imm} (\Roman*)]
\item\label{L.MODTH.2_embedd_imm.I} 		If $ A \subseteq B $ is sharp at  $ \mathfrak{b}_B $, then $ A $ is existentially closed in $ B $ as an $ \mathscr{L}^* $-structure. 
\item If both $ A $ and $ B $ are homogeneous, then $ A $ is existentially closed in $ B $ as an $ \mathscr{L}^* $-structure.
\end{enumerate}
\end{lemma}

\begin{proof}
Analogous to the proof of Lemma \ref{L.MODTH.1_embedd_imm} using Lemma \ref{L.MODTH.2_qfformula_factors} and Proposition \ref{P.MODTH.RCVR_mod_compl.ii}.
\end{proof}

\begin{theorem}\label{T.MODTH.T_2n_mod_compl}
	$ T^*_{n,2} $ is model complete.
\end{theorem}

\begin{proof}
Combine Lemmas \ref{L.MODTH.2_embedd_imm} and 	\ref{L.RCSVR.2_square}.
\end{proof}

\subsection{Consequences of model completeness}

\subsubsection{\texorpdfstring{The model companion of local real closed (SV-) rings of rank $ n $}{The model companion of local real closed (SV-) rings of rank n}}

For the next result, recall that the class of local real closed rings of rank $ n  $ is elementary in the language of rings $ \mathscr{L} $; explicitly, an axiomatization for this class of rings is given by the axioms for local real closed rings (\cite{prestel.schwartz/mod_th_rcr}) together with the $ \mathscr{L} $-sentence $ \phi_{\text{rk}=n} $ defined in Lemma \ref{L.MOD_TH.rk_n_first-order}.

\begin{corollary}\label{C.MODTH.T_1n_mod_companinon}
$ T_{n,1} $ is the model companion of $ T_n $ and also of the $ \mathscr{L} $-theory of local real closed rings of rank $ n $.
\end{corollary}

\begin{proof}
Combine Theorem \ref{T.MODTH.T_1n_mod_compl} together with Lemma \ref{L.RCSVR.canonical_embedd} and Proposition \ref{P.RCSVR.embedd_loc_rcr_fin_rank}.
\end{proof}

\begin{example}
	$ T_n $ does not have the amalgamation property; in particular,  $ T_{n,1} $ is not the model completion of $ T_{n} $ (\cite[Proposition 3.5.18]{chang/keisler.model_theory}). Indeed, let $ V $ be a non-trivial real closed valuation ring of Krull dimension at least 2 and with residue field $ \emph{\textbf{k}} $, and let $ \mathfrak{p} $ be a non-zero non-maximal prime ideal of $ V $. Define $ A:= {\prod}^n_{V/\mathfrak{p}}V $, $ B:=  {\prod}^n_{V_{\mathfrak{p}}/\mathfrak{p}V_{\mathfrak{p}}}V_{\mathfrak{p}}$,  and $C:= {\prod}^n_{\emph{\textbf{k}}}V$; note in particular that $ A \models T_{n,2} $, so that the branching ideal $ \mathfrak{b}_A $ of $ A  $ is property contained in $ \mathfrak{m}_A $, and also $ B,C \models T_{n,1} $. Clearly $ A \subseteq B, C $, and if  $ T_n $ has the amalgamation property, then by Corollary \ref{C.MODTH.T_1n_mod_companinon} there exists $ D \models T_{n,1} $ amalgamating $ B $ and $ C  $ over $ A $; by Lemma \ref{L.RCSVR.embedd_1n_local_factors}, $ B \cap \mathfrak{m}_D = \mathfrak{m}_B $ and $ C \cap \mathfrak{m}_D = \mathfrak{m}_C $, but $ A \cap \mathfrak{m}_B = \text{ker}(A\onto V/\mathfrak{p}) = \mathfrak{b}_A$ and $ A \cap \mathfrak{m}_C = \text{ker}(A\onto \emph{\textbf{k}})=\mathfrak{m}_A$, therefore $ \mathfrak{b}_A = A \cap \mathfrak{m}_D =\mathfrak{m}_A$, a contradiction to $ \mathfrak{b}_A \neq \mathfrak{m}_A $.
\end{example}

\begin{lemma}\label{L.MODTH.amalg_T_n_into_T_n_1}
 Let $ A \models T_n $ and $ B, C \models T_{n,1} $ be such that $ A \subseteq B,C $. If $ A \subseteq B $ and  $ A \subseteq C $ are local embeddings, then there exists  $ D\models T_{n,1}  $ amalgamating $ B $ and $ C $ over $ A $.
\end{lemma}

\begin{proof}
	Write $ \text{Spec}^{\text{min}}(A) = \{ \mathfrak{p}_{A, i}  \mid i \in [n]\}  $, $ \text{Spec}^{\text{min}}(B) = \{ \mathfrak{p}_{B, i}  \mid i \in [n]\}  $,  and $ \text{Spec}^{\text{min}}(C) = \{ \mathfrak{p}_{C, i}  \mid i \in [n]\}  $, and assume without loss of generality that $ \mathfrak{p}_{B, i} \cap A =\mathfrak{p}_{A, i} =   \mathfrak{p}_{C, i} \cap A$ for all $ i \in [n] $ (Corollary \ref{C.SV.min_prime_ann}). Since $ A\subseteq B $ and $ A \subseteq C $ are local embeddings, $ A/\mathfrak{p}_{A,i} \subseteq B/\mathfrak{p}_{B,i}$ and $ A/\mathfrak{p}_{A,i} \subseteq C/\mathfrak{p}_{C,i}$ are local embeddings for all $ i \in [n] $ (Remark \ref{R.SV.local_embedd_factor}), therefore by Lemma \ref{L.MODTH.amalg_rcvrs} there exist non-trivial real closed valuation rings $ V_i $ amalgamating   $ B/\mathfrak{p}_{B,i} $ and $ C/\mathfrak{p}_{C,i} $ over $ A/\mathfrak{p}_{A,i} $  as $ \mathscr{L}(\texttt{m}) $-structures.  Since $ \sf{RCVR}(\texttt{m}) $ is complete (Proposition \ref{P.MODTH.RCVR_mod_compl.i}), there exists a  non-trivial real closed valuation ring $ V $ with residue field $ \textbf{\emph{k}} $ such that $ (V_i, \mathfrak{m}_{V_i})\subseteq (V, \mathfrak{m}_V) $ for all  $ i \in [n] $, therefore it follows that the canonical composite embeddings \[
	A \subseteq B \into \prod_{i \in [n]}B/\mathfrak{p}_{B, i}\subseteq\prod_{i \in [n]} V_i \subseteq\prod_{i \in [n]} V	\ \ \text{ and  } \ \ 	A \subseteq C \into \prod_{i \in [n]}C/\mathfrak{p}_{C,i}\subseteq\prod_{i \in [n]} V_i \subseteq\prod_{i \in [n]} V\] 
	corestrict to composite embeddings $ A \subseteq B \into \prod_{\emph{\textbf{k}}}^n V $ and $ A \subseteq C \into \prod_{\emph{\textbf{k}}}^n V $ (respectively) witnessing that $ \prod_{\emph{\textbf{k}}}^n V \models T_{n,1} $ amalgamates $ B  $ and $ C $ over $ A $.
\end{proof}

\begin{proposition}
	Let $ T_{n,1}(\emph{\texttt{m}}) $ and $ T_{n}(\emph{\texttt{m}}) $ be the $ \mathscr{L}(\emph{\texttt{m}}) $-theories $ T_{n,1} $ and $ T_{n} $ together with the sentence expressing that $ \emph{\texttt{m}} $ is the set of non-units, respectively.
	\begin{enumerate}[\normalfont(i)]
	\item The models of $ T_{n}(\emph{\texttt{m}}) $  have prime extensions in the class of models of $ T_{n,1}(\emph{\texttt{m}}) $.
	\item $ T_{n,1}(\emph{\texttt{m}}) $ is the model completion of $ T_{n}(\emph{\texttt{m}}) $.
	\end{enumerate}
\end{proposition}

\begin{proof}
	Note first that since $ T_{n,1}(\texttt{m}) $ is an extension by definitions of  $ T_{n,1} $ and this latter theory is model complete, $ T_{n,1}(\texttt{m}) $ is also model complete. 

	(i). Let  $ A \models T_{n}(\texttt{m}) $ and set $ \text{Spec}^{\text{min}}(A)= \{ \mathfrak{p}_i \mid i \in [n] \}  $. By Lemma \ref{L.RCSVR.canonical_embedd}, $ A $ embeds into $ A':=  \prod_{A/\mathfrak{m}_A, i \in [n]}A/\mathfrak{p}_i \models T_{n,1}(\texttt{m})$ as an $ \mathscr{L}(\texttt{m}) $-structure, and it is claimed that $ A' $ is the prime extension of $ A $ in the class of models of $ T_{n,1}(\texttt{m}) $. Let  $ B \models T_{n,1}(\texttt{m}) $ and suppose that $ A \subseteq B $ as  $ \mathscr{L}(\texttt{m}) $-structures; set $ \text{Spec}^{\text{min}}(B)= \{ \mathfrak{q}_i \mid i \in [n] \}  $, so that the canonical embedding $  B \linto \prod_{i \in [n]}B/\mathfrak{q}_{i} $  corestricts to an isomorphism $  B \cong \prod_{B/\mathfrak{m}_B, i \in [n]}B/\mathfrak{q}_{i} $ (Lemma \ref{L.RCSVR.equiv_loc_rcsvr_rk_n_exactly_one_branching_ideal.b}); since the embedding $ A \subseteq B $ is local, it induces local embeddings  $ A/\mathfrak{p}_i \subseteq B/\mathfrak{q}_i $ for all $ i \in [n] $, from which it follows that the induced embedding $ \prod_{i \in [n]}A/\mathfrak{p}_i \linto \prod_{i \in [n]}B/\mathfrak{q}_i $ restricts to an $ \mathscr{L}(\texttt{m}) $-embedding $ A' \linto B $ over  $ A $.

	(ii). By Corollary \ref{C.MODTH.T_1n_mod_companinon} and \cite[Proposition 3.5.18]{chang/keisler.model_theory} it suffices to show that  if $ B, C \models T_{n,1}(\texttt{m}) $ contain a common $ \mathscr{L}(\texttt{m}) $-substructure $ A \subseteq B,C $ such that  $ A \models T_{n}(\texttt{m}) $, then there exists $ D \models T_{n,1 } (\texttt{m})$ amalgamating $ B $ and $ C $ over $ A $; the existence of such $ D $ follows by Lemma \ref{L.MODTH.amalg_T_n_into_T_n_1}.
\end{proof}

\begin{remark}
	Let $ T_{n}'  $ be the $ \mathscr{L} $-theory $ T_n $ together with the sentence expressing that every non-unit is a sum of two zero divisors. By the equivalence  (i) $ \Leftrightarrow $  (v) in Proposition \ref{P.RCSVR.equiv_branching_max_id}, models of $ T_{n}' $ are local real closed SV-rings of rank  $ n $ whose maximal ideal is a branching ideal; clearly any model of $ T_{n, 1} $ is a model of  $ T_{n}' $, therefore $ T_{n,1} $ is also the model companion of $ T_{n}' $ by Corollary \ref{C.MODTH.T_1n_mod_companinon}. But in this case, $ T_{n,1} $ is even the model completion of $ T_{n}' $ in the language of rings: indeed, if $ B, C \models T_{n,1} $ contain a common $ \mathscr{L} $-substructure $ A \subseteq B,C $ such that  $ A \models T_{n}'$, then by Lemma \ref{L.RCSVR.embedd_loc_rcr_max_id_branch_implies_local} both $ A \subseteq B $ and  $ A \subseteq C $ are local embeddings, therefore by Lemma \ref{L.MODTH.amalg_T_n_into_T_n_1} there exists $ D \models T_{n,1} $ amalgamating $ B $ and $ C  $ over $ A $.
\end{remark}

\subsubsection{Completeness, decidability, and NIP}
	
\begin{corollary}\label{C.MOD_TH.complete}
	The theories $ T_{n,1} $ and $ T_{n,2} $ are complete.
\end{corollary}

\begin{proof}
	By Theorems \ref{T.MODTH.T_1n_mod_compl} and \ref{T.MODTH.T_2n_mod_compl} to prove completeness it suffices to show that each of the theories $ T_{n,1} $ and $ T^*_{n,2} $ have the joint embedding property; note that $ T^*_{n,2} $ is an extension by definitions of  $ T_{n,2} $, so completeness of  $ T^*_{n,2} $ entails completeness of $ T_{n,2} $.  Let $ A, B \models T_{n,1}  $; by Lemma \ref{L.RCSVR.1_square}, it can be assumed that both $ A  $ and $ B $ are homogeneous, so that there exist non-trivial real closed valuation rings $ V $ and $ V' $ with residue fields $ \emph{\textbf{k}}  $ and $ \emph{\textbf{k}}' $, respectively, such that $ A = {\prod}_{\emph{\textbf{k}}}^n V $ and $ B= {\prod}_{\emph{\textbf{k}}'}^{n}V' $. Since $ \sf{RCVR}(\texttt{m}) $ is complete (Proposition \ref{P.MODTH.RCVR_mod_compl.i}), there exists a non-trivial real closed valuation ring $ W $ with residue field $ \emph{\textbf{l}} $ and local embeddings $ V, V' \subseteq W $, from which it follows that both $ A $ and $ B $ embed into $ {\prod}_{\emph{\textbf{l}}}^n W \models T_{n,1}$; the joint embedding property for $ T^*_{n,2} $ follows  in a similar manner appealing to Lemma \ref{L.RCSVR.2_square} and Proposition \ref{P.MODTH.RCVR_mod_compl.ii}. 
\end{proof}

\begin{corollary}
	The theories $ T_{n,1} $ and $ T_{n,2} $ are decidable.
\end{corollary}

\begin{proof}
	Since $ T_{n,1} $ and $ T_{n,2} $  are complete by Corollary \ref{C.MOD_TH.complete} and recursively axiomatizable (Remark \ref{R.MODTH.def_branch_id.i}), they are decidable.
\end{proof}

For the last result of this subsection, recall that if $ A $ is an NIP $  \mathscr{L}_1$-structure (\cite{simon.guide})   and $ B $  is an $ \mathscr{L}_2 $-structure which is interpretable in $ A $  (\cite[Section 5.3]{hodges}), then $ B$ is also NIP.

\begin{corollary}
The theories $ T_{n,1} $ and $ T_{n,2} $ are NIP.
\end{corollary}

\begin{proof}
	Since IP is preserved under elementary equivalence, it suffices to show by Corollary \ref{C.MOD_TH.complete} that $ {\prod}_{\emph{\textbf{k}}}^n V \models T_{n,1}$ and  $ {\prod}_{V/\mathfrak{p}}^n V\models T_{n,2} $ are NIP, where $ V  $ is a non-trivial real closed valuation ring of Krull dimension at least 2 with residue field $ \emph{\textbf{k}} $ and $ \mathfrak{p} $ is a non-zero  non-maximal prime ideal of $ V $. Since weakly o-minimal theories are NIP (\cite[Appendix A.1.3]{simon.guide}), $ \sf{RCVR}(\leq, \text{div}) $ is NIP by \cite[Corollary 1.6 (c)]{dickmann.elim_covr}, therefore so is $ V \models \sf{RCVR} $.   By \cite[Proposition 3.23]{simon.guide}, the Shelah expansion $ V^{\text{Sh}} $ is also NIP, and since prime ideals of rings are externally definable  (\cite[Fact 4.1]{delbée2024classificationdpminimalintegraldomains}), it follows that both structures $ (V,\mathfrak{m}_V) $ and $ (V, \mathfrak{p}) $ are NIP; since $ {\prod}_{\emph{\textbf{k}}}^n V $ is interpretable in  $ (V, \mathfrak{m}_V) $ and  $ {\prod}_{V/\mathfrak{p}}^n V $ is interpretable in  $ (V, \mathfrak{p}) $, it follows that both $ {\prod}_{\emph{\textbf{k}}}^n V $ and $ {\prod}_{V/\mathfrak{p}}^n V $ are NIP, as required.
\end{proof}

\subsection{\texorpdfstring{Quantifier elimination for $ T_{n,1} $}{Quantifier elimination for Tₙ,₁}}

For this last section recall that any real closed ring $ A $ is an $ f $-ring  $ (A, \vee, \wedge, \leq) $, and since $ a \geq 0 $ in  $ A $ if and only if $ a  $ is a square, the partial order $ \leq $ on $ A $ is definable, and thus so are the lattice operations $ \vee $ and $ \wedge $. In particular, any $ \mathscr{L} $-theory $ T $ of real closed rings can be regarded as an  $\mathscr{L}(\vee, \wedge, \leq)$-theory via extension by definitions; moreover, if $ A \subseteq B$ is an $ \mathscr{L} $-embedding of real closed rings, then $ A \subseteq B $ is also an  $ \mathscr{L}(\vee, \wedge, \leq) $-embedding (the proof of this fact is contained in \cite[11]{prestel.schwartz/mod_th_rcr}). The next example, which is inspired by the example in \cite[19]{prestel.schwartz/mod_th_rcr}, shows failure of quantifier elimination for the theory $ T_{n,1} $ in various languages. 

\begin{example}
	If  $ \mathscr{L}' $ is any language such that $ \mathscr{L} \subseteq \mathscr{L}' \subseteq \mathscr{L}(\vee, \wedge, \leq, \text{div}) $, then $ T_{n,1} $ does not have quantifier elimination in $ \mathscr{L}' $; for notational simplicity only the case $ n=2 $ will be considered here, but the construction can be easily adapted for arbitrary $ n \in  \N^{\geq 2}$.  Assume  for contradiction that $ T_{2,1} $ has quantifier elimination in $ \mathscr{L}' $ and let $ V  $ be a non-trivial real closed valuation ring of Krull dimension at least $ 2 $ with residue field $ \emph{\textbf{k}} $. Let $ \mathfrak{p} $ be a non-zero non-maximal prime ideal of $ V $,  and define $ B:= V\times_{\emph{\textbf{k}}} V $ and $ C:= V\times_{\emph{\textbf{k}}} V/\mathfrak{p} $; then the maps $ \epsilon_B: V \lra B $ and $ \epsilon_C: V \lra C $ given by  $ \epsilon_B(v):= (v, v) $ and  $ \epsilon_C(v, v/\mathfrak{p})$ are $ \mathscr{L}' $-embeddings, and since $ T_{2,1} $ is model complete (Theorem \ref{T.MODTH.T_1n_mod_compl}) there exist $ D \models T_{2,1}$ and $ \mathscr{L}' $-embeddings $ f_B: B \linto D $ and $ f_C:C \linto D $ such that the diagram

\noindent\adjustbox{center}{
\begin{tikzpicture}[node distance=2cm, auto]
  \node (lb) {$V$};
  \node (lt) [above of=lb] {$B$};
  \node (rb) [right of=lb] {$C$};
  \node (rt) [above of=rb] {$D$};
  \draw[{Hooks[right]}->] (lb) to node {$\epsilon_B$} (lt);
  \draw[{Hooks[right]}->] (lb) to node [swap] {$\epsilon_C$} (rb);
  \draw[{Hooks[right]}->] (rb) to node [swap] {$f_C$} (rt);
  \draw[{Hooks[right]}->] (lt) to node {$f_B$} (rt);
\end{tikzpicture}
}
	commutes (\cite[Proposition 3.5.19]{chang/keisler.model_theory}). Pick $ \eta\in \mathfrak{m}_V\setminus \mathfrak{p} $; then $ b:= (0, \eta) \in B$ and $ b':= (\eta, 0)\in B $ are  non-zero orthogonal elements, and $ c:= (0, \eta/\mathfrak{p}) \in C$ and $ c':= (\eta, 0/\mathfrak{p})\in C $ are non-zero orthogonal elements, therefore $ \text{Spec}^{\text{min}}(B) = \{ \text{Ann}_B(b), \text{Ann}_B(b') \}  $, $ \text{Spec}^{\text{min}}(C) = \{ \text{Ann}_C(c), \text{Ann}_C(c') \}  $ and \[
 \text{Spec}^{\text{min}}(D) = \{ \text{Ann}_D(b), \text{Ann}_D(b') \}= \{ \text{Ann}_D(c), \text{Ann}_D(c') \} 
\] 
by Corollary \ref{C.SV.min_prime_ann}. But then \[
	\text{Spec}(f_B \circ \epsilon_B)(\text{Ann}_D(b)) =\text{Spec}(f_B \circ \epsilon_B)(\text{Ann}_D(b')) =\text{Spec}(f_C \circ \epsilon_C)(\text{Ann}_D(c')) = (0)
\] and $ \text{Spec}(f_C \circ \epsilon_C)(\text{Ann}_D(c)) = \mathfrak{p} $, therefore the square of Zariski spectra induced by the commutative square above does not commute, giving the required contradiction.
\end{example}

The example above shows that $ T_{n,1} $ fails to have quantifier elimination in any language $ \mathscr{L}' $ such that $ \mathscr{L}\subseteq \mathscr{L}' \subseteq \mathscr{L}(\vee, \wedge, \leq, \text{div})  $ due to the fact that any $ \mathscr{L}' $-structure $ A $ such  that $ A \models T_{n,1} $  has $ \mathscr{L}' $-substructures of smaller rank; enlarging the language to force all substructures to have rank $ n $ turns out to be sufficient to obtain quantifier elimination.

\begin{definition}
	Let $ e_1, \dots, e_n$ be new constant symbols and let $\mathscr{L}^{\dag}:= \mathscr{L}(\vee, \wedge, \leq, \text{div}, e_1, \dots, e_n) $. Define $ T^{\dag}_{n,1}$ to be the canonical extension by definitions of $ T_{n,1} $ to $ \mathscr{L}(\vee, \wedge, \leq, \text{div}) $ together with the statement expressing that $ e_1 , \dots, e_n$ are non-zero and pairwise orthogonal elements.
\end{definition}

\begin{theorem}
	$ T_{n,1}^{\dag} $ has quantifier elimination.
\end{theorem}

\begin{proof}
	Let $ A, B, C\models (T_{n,1}^{\dag})_{\forall} $ be such that $ A\subseteq B, C$ as $ \mathscr{L}^{\dag} $-structures. Since $ T_{n,1} $ is model complete by Theorem \ref{T.MODTH.T_1n_mod_compl}, so is $ T_{n,1}^{\dag} $, therefore it suffices to show that there exists $ D \models T_{n,1}^{\dag} $ amalgamating $ B $ and $ C  $ over $ A  $ as $ \mathscr{L}^{\dag} $-structures; moreover, since $  B, C\models (T_{n,1}^{\dag})_{\forall}  $, it can be assumed that $ B, C \models T_{n,1}^{\dag} $.

	\noindent \textit{Claim 1}. $ A $  is a reduced local SV-$ f $-ring of rank $ n $ (\cite[Definition 4.1]{schwartz.SV}).

	\noindent \textit{Proof of Claim 1}. $ A $ is a reduced $ f $-ring by the equivalence  (i) $ \Leftrightarrow  $ (ii) in \cite[Theoreme 9.1.2]{bkw.groupes}, and $ A$ is local because $ B $ is local and this property can be expressed by the universal $ \mathscr{L}^{\dag} $-sentence $\forall x[\text{div}(x,1) \text{ or }\text{div}(1-x,1)] $. Since $ A $ is  reduced and local, its rank is the supremum of the lengths of sequences of non-zero pairwise orthogonal elements (Lemma \ref{L.SV.ann_intersection_of_min_prime_id.ii.a}), and as $ A $ is an $ \mathscr{L}^{\dag} $-substructure of $ B\models T_{n,1}^{\dag} $, it follows that $ \text{rk}(A)=n $; finally, $ A $ is an SV-ring by  Lemma \ref{L.MOD_TH.SV_first-order}, and $ A$ has bounded inversion because $ B $ has bounded inversion (\cite[Proposition 12.4]{schwartz/madden.safr}) and this property can be expressed by the universal $ \mathscr{L}^{\dag} $-sentence $ \forall x[1 \leq x \ra \text{div}(x, 1)] $, therefore $ A $ is an SV-$ f $-ring by \cite[Proposition 3.2]{delzell/madden.lattice_ordered_rings} and \cite[Proposition 3.3]{schwartz.convex_subr}.  \hfill $ \square_{\text{Claim 1}} $

	Set $ \mathfrak{p}_{A,i}:= \text{Ann}_A(a_i) $, $ \mathfrak{p}_{B,i}:= \text{Ann}_B(a_i) $, and $ \mathfrak{p}_{C,i}:= \text{Ann}_C(a_i) $ for all $ i \in [n] $, where each $ a_i \in A$ is the interpretation of the constant symbol $ e_i \in \mathscr{L}^{\dag} $, so that $ \text{Spec}^{\text{min}}(A)=\{\mathfrak{p}_{A,i}\mid i\in[n] \}$, $ \text{Spec}^{\text{min}}(B)=\{\mathfrak{p}_{B,i}\mid i\in[n] \}$, and $ \text{Spec}^{\text{min}}(C)=\{\mathfrak{p}_{C,i}\mid i\in[n] \}$ by Corollary \ref{C.SV.min_prime_ann}.  Since $ \mathfrak{p}_{A, i} $ is a minimal prime ideal of the reduced $ f $-ring $ A $, $ \mathfrak{p}_{A, i} $ an irreducible $ \ell $-ideal (\cite[Sections 8.4 and 8.5; Theoreme 9.3.2]{bkw.groupes}), therefore $ A/\mathfrak{p}_{A,i}  $  is a totally ordered domain  and the residue map $ A\lonto A/\mathfrak{p}_{A,i} $ a homomorphism of lattice-ordered rings (see also \cite[31-33]{schwartz/madden.safr}).

	\noindent \textit{Claim 2.} The ring embeddings $ A/\mathfrak{p}_{A,i} \subseteq B/\mathfrak{p}_{B, i} , C/\mathfrak{p}_{C, i} $ are $ \mathscr{L}(\leq, \texttt{m}) $-embeddings for all $ i \in [n] $.

	\noindent \textit{Proof of Claim 2.} The $ \mathscr{L}^{\dag} $-embeddings $ A\subseteq B, C $ are local embeddings (Remark \ref{R.MODTH.div_embedd_local}), therefore they induce local embeddings  $ A/\mathfrak{p}_{A,i}\subseteq B/\mathfrak{p}_{B, i}, C/\mathfrak{p}_{C,i}$ for all $ i \in [n] $ (Remark \ref{R.SV.local_embedd_factor}), and thus it remains to show that these latter embeddings are $ \mathscr{L}(\leq) $-embeddings. Pick $ a \in A $; then  $ a \vee 0 \in A \subseteq B$, $ (a \vee 0)/\mathfrak{p}_{A,i} = (a/\mathfrak{p}_{A,i}) \vee (0/\mathfrak{p}_{A,i}) $, and $ (a \vee 0)/\mathfrak{p}_{B,i} = (a/\mathfrak{p}_{B,i}) \vee (0/\mathfrak{p}_{B,i}) $, from which it follows that $ a/\mathfrak{p}_{A,i} \geq 0/\mathfrak{p}_{A,i}$ if and only if $ a/\mathfrak{p}_{B,i} \geq 0/\mathfrak{p}_{B,i}$, as required. \hfill $ \square_{\text{Claim 2}} $

 By Claim 1 and by the above, $ A/\mathfrak{p}_{A,i} $ is a totally ordered valuation ring for all $ i \in [n] $, and since $ B, C \models T_{n, 1}^{\dag} $, $ B/\mathfrak{p}_{B,i} $ and $ C/\mathfrak{p}_{C,i} $  are non-trivial real closed valuation rings, therefore by Claim 2 and Lemma \ref{L.MODTH.amalg_rcvrs} there exist non-trivial real closed valuation rings $ V_i $ and local embeddings $ B/\mathfrak{p}_{B, i}, C/\mathfrak{p}_{C, i} \subseteq V_i $ such that the diagram

\noindent\adjustbox{center}{
 \begin{tikzpicture}[node distance=2cm, auto]
	\node (lb) {$A/\mathfrak{p}_{A, i} $};
  \node (lt) [above of=lb] {$B/\mathfrak{p}_{B, i}$};
  \node (rb) [right of=lb] {$C/\mathfrak{p}_{C,i}$};
  \node (rt) [above of=rb] {$V_i$};
  \draw[{Hooks[right]}->] (lb) to node {} (lt);
  \draw[{Hooks[right]}->] (lb) to node [swap] {} (rb);
  \draw[{Hooks[right]}->] (rb) to node [swap] {} (rt);
  \draw[{Hooks[right]}->] (lt) to node {} (rt);
\end{tikzpicture}	
}
 
\noindent commutes for all $ i \in [n] $. Amalgamate all $ V_i $ into a single non-trivial real closed valuation ring $ V $ with residue field $ \emph{\textbf{k}} $ in such a way that all the embeddings $ V_i \subseteq V  $ are local and define $ D_0:= {\prod}^n_{\emph{\textbf{k}}} V \models T_{n, 1}$. Note that $ a_i \in \mathfrak{p}_{A, j} $ if and only if $ j = i $ for all $ i , j \in [n] $ and the images of $ a_i/\mathfrak{p}_{A, i} $ and $ a_i/\mathfrak{p}_{C, i} $ in $ V $ coincide;  it follows that the ring $ D_0 $  can be expanded to a model $D \models T_{n, 1}^{\dag} $ in such a way that  the composite local embeddings $ B/\mathfrak{p}_{B,i} \subseteq V_i \subseteq V $ and $ C/\mathfrak{p}_{C,i} \subseteq V_i \subseteq V $ induce $ \mathscr{L}^{\dag} $-embeddings $ B \subseteq D	  $ and $ C \subseteq D$ (note that the ring embeddings $ B, C \subseteq D $ are  $ \mathscr{L}(\text{div}) $-embeddings by Theorem \ref{T.MODTH.T_1n_mod_compl})  making the diagram

\noindent\adjustbox{center}{
\begin{tikzpicture}[node distance=2cm, auto]
	\node (lb) {$A$};
  \node (lt) [above of=lb] {$B$};
  \node (rb) [right of=lb] {$C$};
  \node (rt) [above of=rb] {$D$};
  \draw[{Hooks[right]}->] (lb) to node {} (lt);
  \draw[{Hooks[right]}->] (lb) to node [swap] {} (rb);
  \draw[{Hooks[right]}->] (rb) to node [swap] {} (rt);
  \draw[{Hooks[right]}->] (lt) to node {} (rt);
\end{tikzpicture}
}
\noindent commute, concluding thus the proof.
\end{proof}

%% file: concluding_remarks.tex
The aim of this section is twofold. To start with, Subsection \ref{SUBSEC.diff} contains the key difficulties that arise in trying to carry the same the model-theoretic analysis as done in Section \ref{SEC.mod_th} to those local real closed SV-rings of finite rank which have  two or more branching ideals; then, Subsection \ref{SUBSEC.branching} introduces the notion of the branching spectrum of a local real closed ring of finite rank in order to connect the results of the previous section with  the model theory of real closed rings with radical relations as developed in \cite{prestel.schwartz/mod_th_rcr} (see also \cite{guier.elim}) with a view towards overcoming the obstacles described Subsection \ref{SUBSEC.diff} as well as  laying down the path towards a possible proof of Conjecture \ref{conj_1}.  

As in Section \ref{SEC.mod_th}, fix $ n \in \N^{\geq 2} $ for what remains;  all model-theoretic statements are again assumed to be phrased with respect to the language of rings $ \mathscr{L} $ unless stated otherwise. 

\subsection{Difficulties in the presence of more than one branching ideal}\label{SUBSEC.diff}

A recurring theme  in the study of real closed rings is that their Zariski spectrum serves as a rough measure of their complexity, and this has already been seen in the previous two sections with the differences between rings of type $ (n,1) $ and  of type $ (n,2) $.  Note that by Remarks \ref{R.RCSVR.branch_0.ii} and \ref{R.RCSVR.branch_2}, a local real closed ring of rank $ 2 $ has exactly one branching ideal, and if $ A $ is a local real closed ring of rank $ 3 $, then all the possible configurations of its minimal prime ideals, branching ideals, and its maximal ideal can be summarized in the following  four Hasse diagrams, where the dotted lines indicate that there could be more prime ideals between the two nodes that each dotted line connects:

\noindent\adjustbox{center}{
\begin{tikzpicture}
[
    level distance=10mm,
    sibling distance=13mm
]
\node (root1) at (0,0) {$ \mathfrak{m}_A$}
child {node  {}edge from parent [draw=none]
		child {node  {}edge from parent [draw=none]
			child {node {$ \mathfrak{p}_1$}edge from parent [draw=none]
}
			child {node {$ \mathfrak{p}_2 $}edge from parent [draw=none]
}}
		child {node {}edge from parent [draw=none]
			child {node {}edge from parent [draw=none]}
			child {node {$ \mathfrak{p}_3  $}edge from parent [draw=none]}}
};
\draw[dot diameter=1pt, dot spacing=2pt, dots] (root1) -- (root1-1-1-1);
\draw[dot diameter=1pt, dot spacing=2pt, dots] (root1) -- (root1-1-1-2);
\draw[dot diameter=1pt, dot spacing=2pt, dots] (root1) -- (root1-1-2-2);
\node at (4,0) {$ \mathfrak{m}_A$}
child {node (child2) {$ \mathfrak{b}_A $} edge from parent [dot diameter=1pt, dot spacing=2pt, dots]
		child {node  {}edge from parent [draw=none]
			child {node {$ \mathfrak{p}_1 $}edge from parent [draw=none]
}
			child {node {$ \mathfrak{p}_2$}edge from parent [draw=none]
}}
		child {node {}edge from parent [draw=none]
			child {node {}edge from parent [draw=none]}
			child {node {$ \mathfrak{p}_3 $}edge from parent [draw=none]}}
};
\draw[dot diameter=1pt, dot spacing=2pt, dots] (child2) -- (child2-1-1); 
\draw[dot diameter=1pt, dot spacing=2pt, dots] (child2) -- (child2-1-2);
\draw[dot diameter=1pt, dot spacing=2pt, dots] (child2) -- (child2-2-2);
\node (root3) at (8,0) {$ \mathfrak{m}_A $}
	child {node {}edge from parent [draw=none]
		child {node (child3) {$ \mathfrak{q}$}edge from parent [draw=none]
			child {node {$ \mathfrak{p}_1 $} edge from parent [dot diameter=1pt, dot spacing=2pt, dots]
}
			child {node {$ \mathfrak{p}_2 $} edge from parent [dot diameter=1pt, dot spacing=2pt, dots]}}
		child {node {}edge from parent [draw=none]
			child {node {}edge from parent [draw=none]}
			child {node {$\mathfrak{p}_3 $}edge from parent [draw=none]}}
};
\draw[dot diameter=1pt, dot spacing=2pt, dots] (root3) -- (root3-1-2-2);
\draw[dot diameter=1pt, dot spacing=2pt, dots] (root3) -- (child3);
\node  at (12,0) {$ \mathfrak{m}_A $}
	child {node (child4) {$ \mathfrak{q}_2 $}edge from parent [dot diameter=1pt, dot spacing=2pt, dots]
		child {node {$ \mathfrak{q}_1 $}edge from parent [dot diameter=1pt, dot spacing=2pt, dots]
			child {node {$ \mathfrak{p}_1 $}edge from parent [dot diameter=1pt, dot spacing=2pt, dots]}
			child {node {$ \mathfrak{p}_2 $}edge from parent [dot diameter=1pt, dot spacing=2pt, dots]}}
		child {node {}edge from parent [draw=none]
			child {node {}edge from parent [draw=none]}
			child {node {$ \mathfrak{p}_3 $}edge from parent [draw=none]}}
};
\draw[dot diameter=1pt, dot spacing=2pt, dots] (child4) -- (child4-2-2);
\end{tikzpicture}	
}

If $ A $ is furthermore an SV-ring, then the fist diagram indicates that $ A $ is a ring of type $ (3,1) $, and the second diagram indicates that $ A $ is a ring of type $ (3,2) $. The fact that in the two diagrams on the left there is exactly one branching ideal creates a situation of \enquote{symmetry} that has been heavily exploited in Section \ref{SEC.mod_th}, namely, the model-theoretic analysis of rings of type $ (n,j ) $ ($ j \in [2] $) is done in terms of the model theory of each of the residue domains $ A/\mathfrak{p}_1, \dots, A/\mathfrak{p}_n $ ($ \text{Spec}^{\text{min}}(A) = \{ \mathfrak{p}_i \mid i \in [n] \}  $), where each of these domains are all canonically regarded as structures in the same language ($ \mathscr{L}(\texttt{m}) $ if $ j =1 $ and $ \mathscr{L}(\texttt{b},\texttt{m}) $ if $ j =2$, see Subsection \ref{SUBSEC.mod_compl}) and can therefore be easily compared and manipulated as models of a suitable theory ($ \sf{RCVR}(\texttt{m}) $ if $ j =1 $ and $ \sf{RCVR}(\texttt{b},\texttt{m}) $ if $ j =2$). 

This \enquote{symmetry} can break in the presence of more than one branching ideal. In particular, two residue domains $ A/\mathfrak{p}_i $ and $ A/\mathfrak{p}_j  $ may carry information about a different number of prime ideals of $ A $, and thus capturing this information via unary predicates leads to different languages (e.g., in the case of the third diagram, $ A/\mathfrak{p}_2 $ can be canonically regarded as an $ \mathscr{L}(\texttt{b}, \texttt{m}) $-structure, but $ A/\mathfrak{p}_3 $ may not, although the latter can be canonically regarded as an $ \mathscr{L}(\texttt{m}) $-structure);  the fact that the residue domains $ A/\mathfrak{p}_i $ may not be all canonically regarded as structures in the same language in the way described above creates an inherent difficulty in the model-theoretic analysis of local real closed SV-rings of finite rank in terms of the domains $ A/\mathfrak{p}_i $ when there is more than one branching ideal.

There are also differences in terms of definability in the presence of more than one branching ideal. One fact about local real closed rings $ A $ of finite rank with one branching ideal $ \mathfrak{b}_A $ that was used in the proof of completeness of $ T_{n,2} $ (Corollary \ref{C.MOD_TH.complete})  is that $ \mathfrak{b}_A $ is definable in $ A $ without parameters (Remark \ref{R.MODTH.def_branch_id.ii}); the next example shows that branching ideals are generally not definable without parameters in local real closed rings of finite rank which have more than one branching ideal:

\begin{example}\label{ex.not_def}
	Let $ V $ be a real closed domain of Krull dimension  $ 2 $ and let $ \mathfrak{p} $ be a non-zero non-maximal prime ideal. Define $ A_1 = A_2:= V\times_{V/\mathfrak{p}} V$, noting that $ A_1  $ is a local real closed ring of rank  $ 2 $  with unique branching ideal $ \text{ker}(A_1 \onto V/\mathfrak{p})  $ and with residue field $V/\mathfrak{m}_V =:\textbf{\emph{k}} $. The ring $ A:=A_1\times_{\emph{\textbf{k}}} A_2 \subseteq A_1 \times A_2$ is a local real closed ring of rank $ 4 $ with exactly $ 3 $ branching ideals, namely $\mathfrak{q}_1:= \text{ker}(A\onto A_1\onto V/\mathfrak{p}) $, $ \mathfrak{q}_2:=\text{ker}(A\onto A_2\onto V/\mathfrak{p}) $, and $ \mathfrak{m}_A=\text{ker}(A\onto \emph{\textbf{k}}) $;  if $ \text{Spec}^{\text{min}}(A)= \{ \mathfrak{p}_1, \dots, \mathfrak{p}_4\} $, then the Hasse diagram of $ (\text{Spec}(A), \subseteq) $ is

\noindent\adjustbox{center}{
\begin{tikzpicture}
[
    level distance=10mm,
    sibling distance=15mm,
]
\node     {$ \mathfrak{m}_A $} [sibling distance=20mm]
		child {node {$ \mathfrak{q}_1 $} [sibling distance=13mm]
			child {node {$ \mathfrak{p}_1$}}
			child {node {$ \mathfrak{p}_2$}}}
		child {node {$ \mathfrak{q}_2 $}[sibling distance=13mm]
			child {node {$ \mathfrak{p}_3 $}}
	child {node {$ \mathfrak{p}_4$}}};
\end{tikzpicture}	
}

\noindent and the map $ (a_1, a_2) \mapsto (a_2, a_1) $ is an automorphism of $ A \subseteq A_1\times A_2$ which swaps $ \mathfrak{q}_1 $ and $ \mathfrak{q}_2 $.
\end{example}

Finally, recall that given a real closed valuation ring $ C $, there always exists a non-trivial real closed valuation ring $ W $ with homomorphic image  $ C $ yielding the homogeneous ring  $ {\prod}_{C}^nW $ of type  $ (n,j ) $ (if $ C $ is a field, then $ j=1 $, otherwise $  j=2$), see Definition \ref{D.RCVSR.homogeneous}; loosely speaking, for any fixed \enquote{cofactor} $ C $ one can find a \enquote{factor} $ W $ yielding a ring of type  $ (n,1) $ or of type $ (n,2) $ having the property that all its residue domains modulo minimal prime ideals are abstractly isomorphic to $ W $. The model-theoretic relevance of this construction is that one can make the property of a tuple $ \overline{a} \in W^n $ being an element of $ {\prod}_{C}^nW  $ \enquote{internal to $ W $}, in the sense that  $ \overline{a} \in  {\prod}_{C}^nW $ if and only if $  a_{i}-a_{j} \in \text{ker}(W \onto C)$ for all $ i, j \in [n] $; this was crucially used in the model completeness proof of $ T_{n,1} $, see Lemma \ref{L.MODTH.1_embedd_imm.II}. Although it is easy to see that this construction can be done for all local real closed SV-rings of rank $ 3 $, the next example shows that this is not any more the case in the presence of two incomparable branching ideals (the precise statement is Claim 2 in the example below):

\begin{example}
	Let $ \emph{\textbf{k}} $ be a real closed field and $ \Gamma $ be a divisible  totally ordered abelian group without a smallest non-zero convex subgroup (for instance, $ \Gamma $ can be taken to be $\Q^{\N}$ ordered lexicographically); in particular, $V_1:= \emph{\textbf{k}}[[\Gamma]] $ is a non-trivial real closed valuation ring without  a largest non-maximal prime ideal, i.e., $ \mathfrak{m}_{V_1} $ does not have an immediate predecessor in $(\text{Spec}(V_1), \subseteq)  $. Let $ V_2:= \emph{\textbf{k}}[[\Q]] $, noting that $ \text{Spec}(V_2)= \{(0), \mathfrak{m}_{V_2} \}  $.

\noindent \textit{Claim 1.} There does not exist a surjective ring homomorphism $ f: V_1 \lonto V_2$ nor a surjective ring homomorphism $ g: V_2\lonto V_1$.

\noindent \textit{Proof of Claim 1.} Assume for contradiction that either $ f $ or $ g $ as in the statement of the claim exist. Note that since  $ f $ and $ g $ are surjective,  either $ \text{Spec}(V_2) $ is a final segment in $ (\text{Spec}(V_1), \subseteq) $ or  $ \text{Spec}(V_1) $ is a final segment in  $ (\text{Spec}(V_2), \subseteq)$, respectively; but this is impossible by choice of $ V_1 $ and $ V_2 $. \hfill $ \square_{\text{Claim }1} $

\noindent \textit{Claim 2.} There does not exist a non-trivial real closed valuation ring $ W $ having both $ V_1 $ and $ V_2  $ as homomorphic images; in particular, there is no local real closed SV-ring of rank $ 4 $ of the form $ (W\times_{V_1}W) \times_\textbf{\emph{k}} (W\times_{V_2}W)$.

\noindent \textit{Proof of Claim 2.} Assume for contradiction that there exist surjective ring homomorphisms $ f_1: W \lonto V_1 $ and  $ f_2: W\lonto V_2 $; since  $ W  $ is a valuation ring, either $ \text{ker}(f_1) \subseteq \text{ker}(f_2) $ or $ \text{ker}(f_2) \subseteq \text{ker}(f_1) $, therefore either $ V_2 $ is a homomorphic image of  $ V_1$ or  $ V_1$ is a homomorphic image of  $ V_2$, a contradiction to Claim 1.
\hfill $ \square_{\text{Claim 2}} $
\end{example}

\subsection{The branching spectrum of a local real closed ring of finite rank}\label{SUBSEC.branching}

Although individual branching ideals in an arbitrary local real closed ring $ A $ of finite rank are generally  not definable without parameters (Example \ref{ex.not_def}), the theory of $ A $ \enquote{knows} about the poset configuration of its branching ideals in a sense which is made precise in  Corollary \ref{corll}. First, some preliminaries are needed.

\begin{definition}\label{def.br}
Let $ A $ be a local real closed ring of finite rank. Define the \textit{branching spectrum of $ A $} to be 
\[
	\text{BrSpec}(A):= \text{Spec}^{\text{min}}(A) \ \cup \  \{ \mathfrak{m}_A \} \ \cup \   \{ \mathfrak{p} \in \text{Spec}(A) \mid \mathfrak{p} \text{ is a branching ideal} \} 
.\] 
\end{definition}

\begin{definition}\label{def_root}
	Let $ (P, \sqsubseteq) $ be a  \textit{root system}, i.e., $(P, \sqsubseteq)   $  is a poset such that the principal up-set $ p^{\uparrow}:= \{ q \in P \mid p \sqsubseteq q \}   $ is a chain for all $ p \in P $.

	\begin{enumerate}[\normalfont(I)]	
	\item  $ q\in P $ is a  \textit{branching point} if there exist $ p_1, p_2 \in P $ such that $ p_1, p_2 \sqsubsetneq q $ and $\{ q \} = (p_1 ^{\uparrow} \cap p_2^{\uparrow})^{\text{min}} $ \footnote{If $ S \subseteq (P, \sqsubseteq)$ is any subset, define $ S^{\text{min}}:= \{ s \in S \mid s  \text{ is minimal in } S \text{ with respect to }\sqsubseteq \} $.}.
	\item  $ P $ is a \textit{root} if there exists an element  $ \top \in P $ such that  $ p \sqsubseteq \top   $ for all $ p \in P $ (note that if such element exists, it must be unique). 
	\item Suppose that  $ P $ is a finite root (i.e., $ P $ is a finite root system with unique maximal element).
		\begin{enumerate}[\normalfont(i)]
			\item  The \textit{rank of $ P $} is $\text{rk}(P):= |P^{\text{min}}|  $.
			\item The  \textit{branching root of $ P $} is the subposet  \[
					\text{Br}(P):= P^{\text{min}} \ \cup \  \{ \top \} \ \cup \   \{ p \in P \mid p \text{ is a branching point} \} 
			.\] 
		\item $ P $ is \textit{reduced} if $ P = \text{Br}(P) $, i.e., $ P $ is reduced if every  $ p \in P \setminus (P^{\text{min}} \cup \{ \top \} ) $  is a branching point.
		\end{enumerate}
	\end{enumerate}
\end{definition}

\begin{remark}\label{R.TOW.roots}
	\begin{enumerate}[\normalfont(i), ref=\ref{R.TOW.roots} (\roman*)]
	\item\label{R.TOW.roots.i} 		Any finite root is a $ \vee $-semilattice with join operation given by 	$ p_1 \vee p_2:=(p_1 ^{\uparrow} \cap p_2^{\uparrow})^{\text{min}}  $.
	\item By Remark \ref{R.RCSVR.branch_2}, local real closed rings of finite rank have finitely many branching ideals; in particular, the poset $ (\text{BrSpec}(A), \subseteq )$ is a finite reduced root. 
	\end{enumerate}
\end{remark}

If  $ (P, \sqsubseteq) $ is a finite root system, then there exists a real closed ring $ A  $ such that  $ (\text{Spec}(A), \subseteq) \cong(P, \sqsubseteq) $. Indeed, by \cite{dickmann/gluschankof/the_order_str} there exists a ring $ B $ such that $ (\text{Sper}(B), \subseteq) \cong(P, \sqsubseteq) $, where $ \text{Sper}(B) $ is the real spectrum of $ B $  (\cite[Section 13]{dickmann/schwartz/tressl.specbook}); then $ (\text{Spec}(A), \subseteq) \cong (P, \sqsubseteq)$, where $ A:= \rho(B) $ is the real closure of $ B $, see \cite[Section 13.6.3]{dickmann/schwartz/tressl.specbook} and the references therein. The next lemma shows that one can in fact choose $ A $ to be a real closed SV-ring:

\begin{lemma} 
	Let $ (P, \sqsubseteq) $ be a finite root system. There exists a  real closed SV-ring $ A $  such that  $ (\emph{Spec}(A), \subseteq) \cong (P, \sqsubseteq) $.
\end{lemma}

\begin{proof}
	First note that it suffices to prove the statement for finite roots. Indeed, if  $ (P, \sqsubseteq) $ is any finite root system, then $ P= P_1 \ \dot{\cup} \ \dots \  \dot{\cup} \ P_m$, where each $ (P_i, \sqsubseteq) $ is a finite root; if $ A_i $ is a real closed SV-ring such that $ (\text{Spec}(A_i), \subseteq) \cong (P_i, \sqsubseteq) $ for all $ i \in [m] $, then $ A:= A_1 \times \dots \times A_m $ is a real closed SV-ring  (by Proposition \ref{P.SV.constr_SV-ring.i} and Theorem \ref{T.RCSVR.properties_rcr.I}) such that $ (\text{Spec}(A), \subseteq) \cong (P, \sqsubseteq)  $. 

	The proof is now by induction on the rank of finite roots $ (P, \sqsubseteq) $. If $ \text{rk}(P) = 1$, then  $ (P, \sqsubseteq) $   chain of $ n$ elements for some $ n \in \N $, therefore choosing $ A $ to be a real closed valuation ring of Krull dimension $ n-1 $ does the job. Let  $k \in \N  $ and assume  that the statement holds for all finite  roots of rank $ k $. Let $ (P, \sqsubseteq)  $ be a finite  root of rank $ k+1 $ with minimal elements $p_1, \dots, p_{k}, p_{k+1}  $. Choose $ i \in [k+1] $ such that  $ p_i^{\uparrow}  $ is a chain of maximal cardinality $ n \in \N $ in $( P, \sqsubseteq)  $, pick any $ j \in [k+1] \setminus \{ i \}  $, and define $ P':= \bigcup_{\ell \in [k+1]\setminus \{ j \} } p_{\ell}^{\uparrow} \subseteq P$; then $ (P', \sqsubseteq) $ is a finite root of rank $ k $, and thus by inductive hypothesis there exists a real closed SV-ring $ A' $ and a poset isomorphism $f: (P', \sqsubseteq)\lra (\text{Spec}(A'), \subseteq) $. Define  $ V:= A'/f(p_{i})  $ (noting that $ V $ is a real closed valuation ring of Krull dimension $ n-1 $) and  $ q:= p_{i} \vee p_{j} $ (Remark \ref{R.TOW.roots.i}); then $ |p_{j}^{\uparrow} |:= m\leq n $ by assumption on $ p_i^{\uparrow} $, and this yields two possible cases:
	\begin{enumerate}[-]
		\item $ m = n $.  In this case, $ (\text{Spec}(A), \subseteq) \cong (P, \sqsubseteq)$ for $ A:= A'\times_{B/f(q)} B $, see \cite[Section 12.5.7]{dickmann/schwartz/tressl.specbook}.
		\item $ m < n $. In this case, $ (\text{Spec}(A), \subseteq) \cong (P, \sqsubseteq)$ for $ A:= A'\times_{B/f(q)} B/f(r) $, where $ r \in p_i^{\uparrow} $  is such that $ |r^{\uparrow}| = m $.
\end{enumerate}
In each of the cases above, $ A $ is a real closed SV-ring by Proposition \ref{P.SV.constr_SV-ring.iv} and Theorem \ref{T.RCSVR.properties_rcr.I}; this concludes the inductive step and thus the proof.
\end{proof}

\begin{lemma}\label{L.br_spec}
	Let $ (P, \sqsubseteq)$ be a finite reduced root of rank at least $ 2 $. There exists an $ \mathscr{L} $-sentence $ \phi_{(P, \sqsubseteq)} $ such that  \[
		A \models \phi_{(P, \sqsubseteq)} \iff (\emph{BrSpec}(A), \subseteq) \cong (P, \sqsubseteq) 
	\]  for all local real closed rings $ A $ of finite rank.	
\end{lemma}

\begin{proof}
	This is clear from combining Lemma \ref{L.MOD_TH.rk_n_first-order}  together with the following facts about a  local real closed ring  $ A  $ of rank $ n\in \N^{\geq 2} $:

	\begin{enumerate}[\normalfont(i)]
		\item 	$ \text{Spec}^{\text{min}}(A)= \{ \text{Ann}(a_i) \mid i \in [n] \}  $ for all non-zero pairwise orthogonal  elements $ a_1, \dots, a_n \in A $ (Lemma \ref{L.SV.ann_intersection_of_min_prime_id.ii.b});
		\item   each branching ideal of  $ A $ is a sum of two distinct minimal prime ideals (Remark \ref{R.RCSVR.branch_1}); and
		\item the maximal ideal is a branching ideal if and only if every non-unit is a sum of two zero divisors (Proposition \ref{P.RCSVR.equiv_branching_max_id}).
	\end{enumerate}  
	More precisely, assume without loss of generality that $ (P, \sqsubseteq) $ is a finite reduced root of rank $ n \in \N^{\geq 2} $ such that $ \top $ is a branching point. Then $ \phi_{(P, \sqsubseteq)} $ can be taken to be the conjunction of: $ \phi_{\text{rk}=n} $ (Lemma \ref{L.MOD_TH.rk_n_first-order}), the $ \mathscr{L} $-sentence expressing \enquote{every non-unit is a sum of two zero divisors}, and the $ \mathscr{L} $-sentence expressing \enquote{there exist non-zero orthogonal elements $ a_1, \dots, a_n\in A $ such that $ (\{ \text{Ann}(a_i) \mid i \in [n] \}  \cup \{ \text{Ann}(a_i) + \text{Ann}(a_j)\mid i, j  \in [n] \}, \subseteq) $ is poset-isomorphic to $ (P, \sqsubseteq) $}. For instance, if $ (P, \sqsubseteq) $  is the finite reduced root 

\noindent\adjustbox{center}{
\begin{tikzpicture}
[
    level distance=10mm,
    sibling distance=13mm
]
\node (child4) {$ \top $}
		child {node {$ \bullet $}
			child {node {$ p_1$}}
			child {node {$ p_2 $}}}
		child {node {}edge from parent [draw=none]
			child {node {}edge from parent [draw=none]}
			child {node {$ p_3 $}edge from parent [draw=none]}}
;
\draw[-] (child4) -- (child4-2-2);
\end{tikzpicture}	
}
then the last $ \mathscr{L} $-sentence described above would be the one expressing \enquote{there exist non-zero pairwise orthogonal elements $ a_1, a_2, a_3 \in A $ such that  $ \text{Ann}(a_1)+ \text{Ann}(a_3) = \text{Ann}(a_2)+ \text{Ann}(a_3) $ and $ \text{Ann}(a_1)+ \text{Ann}(a_2)\subsetneq \text{Ann}(a_2)+ \text{Ann}(a_3)  $}.
\end{proof}

\begin{corollary}\label{corll}
Let $ A $ and $ B $ be local real closed rings of finite rank. If $ A \equiv B $, then $ (\emph{BrSpec}(A), \subseteq) \cong (\emph{BrSpec}(B), \subseteq) $.
\end{corollary}

\begin{proof}
	Immediate from Lemma \ref{L.br_spec}.
\end{proof}

Let $ A $ and $ B $ are local real closed SV-rings of finite rank $ n\in \N^{\geq 2} $ with one branching ideal and suppose that $ (\text{BrSpec}(A), \subseteq) \cong (\text{BrSpec}(B), \subseteq)  $; then $ A $ is of type $ (n,j ) $  ($ j \in [2] $) if and only if $ B $ is of type $ (n,j) $, therefore $ A \equiv B $ by Corollary \ref{C.MOD_TH.complete}. This observation gives rise to the following conjecture on an elementary classification of local real closed SV-rings of finite rank:

\begin{conjecture}\label{conj_elem}
Let $ A  $ and $ B $ be local real closed SV-rings of finite rank.  Then $ A \equiv B$ if and only if $  (\emph{BrSpec}(A), \subseteq) \cong (\emph{BrSpec}(B), \subseteq)  $.
\end{conjecture}

If $ A $ is a local real closed ring of finite rank, then  $ \text{BrSpec}(A) $ is a finite subset of the spectral space  $ \text{Spec}(A) $, and as such,  $ \text{BrSpec}(A) $ is proconstructible in  $ \text{Spec}(A) $ (i.e., it is a spectral subspace of $ \text{Spec}(A) $, see \cite{dickmann/schwartz/tressl.specbook}); in particular, $ A $ corresponds to the \textit{radical relation} $ \preceq_{\text{BrSpec}(A)} \subseteq A^2 $ on $ A $. Radical relations on rings are certain binary relations which were introduced in \cite{prestel/schmid.ex_cl_dom_with_rad_rel} and later used in \cite{prestel.schwartz/mod_th_rcr} for the model-theoretic analysis of real closed rings. It is shown in \cite{prestel.schwartz/mod_th_rcr} that if $ A $ is any real closed ring, then \[
	X \subseteq \text{Spec}(A) \ \longmapsto \ a \preceq_{X} b \overset{\text{def}}{\iff} \forall \mathfrak{p} \in X[b \in \mathfrak{p} \Ra a \in \mathfrak{p}]
\] is a bijection between proconstructible subsets $ X \subseteq \text{Spec}(A) $ and radical relations  $ \preceq $ on $ A $, therefore a real closed ring  with a radical relation $ (A, \preceq_X) $ \enquote{knows} about the spectral space $ X $ since the bounded and distributive lattice $ \overline{\mathcal{K}}(X) $ of closed constructible subsets of $ X $  is interpretable in  $ (A, \preceq_X) $; furthermore, the model theory of real closed valuation rings with radical relations is well-understood from the work carried in the last three sections of \cite{prestel.schwartz/mod_th_rcr}. In view of all of the above, a possible approach to a uniform model-theoretic analysis of all local real closed SV-rings of finite rank and to  answer  Conjecture \ref{conj_1} in the affirmative is to study such rings equipped with the radical relation corresponding to their branching spectrum.

%% file: appendix.tex
The aim of this section is proving Theorem \ref{T.APP.rcvf_hahn_embedd_II}. Familiarity with the basic notions and properties of valued fields is assumed throughout (see for example \cite[Chapter 2]{engler/prestel.valued_fields} or \cite[Chapter 3]{aschenbrenner/vdD/vdHoeven.asympt}), as well as familiarity with  ordered and real closed fields;  in what follows, fix the notation and conventions used for this appendix.

Every valuation on a field $ K $ is denoted by $ v $, with the exception of the canonical valuation on fields of Hahn series $ \emph{\textbf{k}}((\Gamma)) $, in which case the valuation is denoted by $\nu$, see Theorem \ref{T.RCSVR.Hahn}; in particular, if $ K \subseteq L $ is an extension of valued fields, then the valuation on  $ L $ is $ v $ and the valuation on $ K  $ is $ v_{\upharpoonright K} $. If $ K  $ is a valued field, then $ V_K $ is the corresponding valuation ring and $ \lambda_K : V \lonto V/\mathfrak{m}_V $ is the residue field map; write $ V:= V_K $ and $ \lambda:= \lambda_K $ if  $ K $ is clear from the context.

An \textit{ordered valued field}  is a totally ordered field $ K $ equipped with an \textit{order-compatible valuation}  (also called \textit{convex valuation}) $ v: K \lonto \Gamma $, i.e., for all $ a, b \in K  $, if  $ 0 < a < b $, then  $ v(b) \leq v(a) $. If $ K $ is a totally ordered field and $ v : K \lonto \Gamma $ is a valuation on  $ K $, then the following are equivalent:
		\begin{enumerate}[\normalfont(i)]
		\item 	$ v : K \lonto \Gamma $ is an order-compatible valuation.
		\item $ V$ is convex in $ K $.
		\item The composite map  $ K^{>0} \onto \Gamma \onto \Gamma^{\text{op}} $ given by $ a \mapsto -v(a) $ is a surjective morphism of totally ordered groups, where $ \Gamma^{\text{op}} $ is the totally ordered group obtained by reversing the order of $ \Gamma $; in particular, $ \text{ker}(v_{\upharpoonright K^{>0}}) $ is a convex subgroup of $ K^{>0} $.
		\end{enumerate}
		Since convex subrings of totally ordered fields are valuation rings (\cite[Proposition 2.2.4]{knebusch/scheiderer.real_algebra}), ordered valued fields can be equivalently defined as pairs $(K, V) $, where $ K $ is an ordered field and $ V \subseteq K$ is a convex subring; in particular, if $ K $ is an ordered valued field, then its residue field $ \textbf{\emph{k}}:= V/\mathfrak{m}_{V} $ is endowed with a canonical total order turning it into a totally ordered field in such a way that the residue field map $ \lambda: V \lonto \textbf{\emph{k}} $ is order-preserving. 

		A \textit{real closed valued field} is an ordered valued field  which is real closed as a field, i.e., it is a real closed field equipped with an order-compatible valuation; equivalently, it is a real closed field  with a distinguished convex subring. If $ K$ is an ordered valued field, then its real closure $ \rho(K) $ will be regarded as a real closed valued field with the valuation induced by $ K $, i.e.,  $ V_{\rho(K)}$ is defined as the convex hull of $ V_K $ in $ \rho(K) $; if the value group and the residue field of $ K $  are  $ \Gamma $ and  $ \emph{\textbf{k}}  $ (respectively), then the value group  and the residue field of $ \rho(K) $ are $ \Q \Gamma$ and $ \rho(\emph{\textbf{k}}) $ (respectively), and the field embedding $ \rho_K: K \linto \rho(K) $ is an embedding of valued fields, see \cite[Corollary 3.5.18]{aschenbrenner/vdD/vdHoeven.asympt}. Any isomorphism of ordered valued fields $ K \lra L $ extends uniquely to an isomorphism of valued fields $ \rho(K) \lra \rho(L) $, which is also  order-preserving  since $ \rho(K)  $ and $ \rho(L) $ are real closed; therefore, if $ R $ is a real closed valued field and $ \epsilon: K \linto R $ is an embedding of ordered valued fields, then $ \epsilon $ can be extended  uniquely to an embedding of valued fields $ \rho(K)\linto R $.

\begin{lemmaApp}\label{L.APP.mon_grp_rcvf}
	Let $ K$ be a real closed  valued field with value group $ \Gamma $ and $ G \subseteq K^{>0} $ be a subgroup. The following are equivalent:
	\begin{enumerate}[\normalfont(i)]
		\item  $ G $ is a \emph{monomial group of} $ K $, i.e., $ v_{\upharpoonright G}: G \lra \Gamma $ is a group isomorphism.
		\item $ G $ is a subgroup of $ K^{>0} $ maximal with $ G \cap \emph{ker}(v_{\upharpoonright K^{>0}}) = (1) $.
	\end{enumerate}
 In particular:
	\begin{enumerate}[\normalfont(a)]
	\item 	Every real closed valued field has a monomial group.
	\item If $ K \subseteq L $ is an extension of real closed valued fields and $ G $ is a monomial group of $ K $, then there exists a monomial group $ H $ of $ L $ containing $ G $.
	\end{enumerate}
\end{lemmaApp}

\begin{proof}
	\underline{(i) $ \Ra $  (ii).} Since $ v_{\upharpoonright G} $ is injective, $ G \cap \text{ker}(v_{\upharpoonright K^{>0}}) = (1) $. Assume for contradiction that there exists a subgroup $ G \subsetneq G' \subseteq K^{>0} $  with $ G' \cap \text{ker}(v_{\upharpoonright G'}) = (1) $ and pick $ g' \in G' \setminus G $; since $ v_{\upharpoonright G} $ is surjective, there exists $ g \in G $ with $ v(g) = v(g') $, hence $ g'g^{-1} \in G' \cap  \text{ker}(v_{\upharpoonright K^{>0}})  = (1) $, and thus  $g' = g    $, a contradiction to the choice of $ g'$.

	\underline{(ii) $ \Ra $ (i).} Since  $ v $ is an order-compatible valuation on $ K $, $ \text{ker}(v_{\upharpoonright K^{>0}}) $ is a convex subgroup of $ K^{>0} $; since $ K $ is real closed, $ K^{>0} $ is divisible, and thus $ \text{ker}(v_{\upharpoonright K^{>0}}) $ is a divisible subgroup of $ K^{>0} $. By \cite[Theorem 21.2]{fuchs.inf_ab_gps_vol_1} and by choice of $ G $, $ K^{>0}=  \text{ker}(v_{\upharpoonright K^{>0}})\cdot G $, i.e., $ v_{\upharpoonright G}: G \lra \Gamma $ is a group isomorphism, as required.

\noindent Items (a) and (b) follow from the implication   (ii) $ \Rightarrow $ (i)  and an application of Zorn's lemma.
\end{proof}

\begin{lemmaApp}\label{L.APP.coeff_field_rcvf}
	Let $ K $ be a real closed valued field with residue field $ \textbf{k} $ and $ \textbf{k}_0 \subseteq V $ be a subfield. The following are equivalent: 
	\begin{enumerate}[\normalfont(i)]
		\item  $ \textbf{k}_0 $ is a \emph{coefficient field of} $ K $, i.e., $ \lambda_{\upharpoonright \textbf{k}_0} : \textbf{k}_0 \lra \textbf{k}$  is a field isomorphism. 

		\item $ \textbf{k}_0 $ is a maximal subfield of $ V $.
	\end{enumerate}
 In particular:
	\begin{enumerate}[\normalfont(a)]
	\item 	Every real closed valued field has a coefficient field.
	\item If $ K \subseteq L $ is an extension of real closed valued fields and $ \textbf{k}_0 $ is a coefficient field of $ K $, then there exists a coefficient field $ \textbf{l}_0 $ of $ L $ containing $ \textbf{k}_0 $.
	\end{enumerate}
\end{lemmaApp}

\begin{proof}
	\underline{(i) $ \Ra $  (ii).} Assume for contradiction that there exists a subfield $ \emph{\textbf{k}}_0 \subsetneq \emph{\textbf{k}}' \subseteq V $ and pick $ a' \in \emph{\textbf{k}}' \setminus \emph{\textbf{k}}_0 $; since $ \lambda_{\upharpoonright \emph{\textbf{k}}_0} $ is surjective, there exists $ a \in \emph{\textbf{k}}_0 $ with $ \lambda(a) = \lambda(a') $, hence $ a-a' \in \emph{\textbf{k}}'  \cap \mathfrak{m}_V= (0)$, and thus $ a' = a  $, a contradiction to the choice of $ a' $.

	\underline{(ii) $ \Ra $ (i).} Folklore; see for instance \cite[Proposition 2.5.3]{knebusch/scheiderer.real_algebra} or \cite[Proposition 2.1]{schwartz.rcvr}.

\noindent Items (a) and (b) follow from the implication (ii) $ \Rightarrow $ (i) and an application of Zorn's lemma.
\end{proof}

\begin{lemmaApp}\label{lem.ord_pres_mon_group}
	Let $ K $ be an ordered valued field  with value group $ \Gamma $ and residue field  $ \textbf{k} $, and suppose that $ G \subseteq K^{>0} $ is a monomial group of $ K$. If  $ \epsilon : K  \linto \textbf{k}((\Gamma))$ is an embedding of valued fields such that $ \epsilon(g) = x^{v(g)} $ for all $  g \in G $, then $ \epsilon $ preserves the order.
\end{lemmaApp}

\begin{proof}
	Let $ r \in K^{>0} $, assume without loss of generality that  $ r \in V $ (otherwise replace $ r $ by $ r^{-1} $), and write \[
		\epsilon(r) := a_{\gamma_0}x^{\gamma_0} + \sum a_{\gamma}x^{\gamma}	,\]
		where $ \gamma_0 := \nu(\epsilon(r)) = v(r) \in \Gamma $; it must be shown that $ \epsilon(r) >0 $, i.e., that $ a_{\gamma_0} >0 $. Let $ g\in G  $ be such that $ v(g) =  \gamma_0 $; then  $   0 = v(rg^{-1}) = \nu(\epsilon(rg^{-1}))$, and \[
			\epsilon(rg^{-1}) =  \epsilon(r)\epsilon(g^{-1}) = a_{\gamma_0} + \sum a_{\gamma}x^{\gamma-\gamma_0} \in \emph{\textbf{k}}[[\Gamma]] 
		,\] therefore $ 0 \neq a_{\gamma_0} = \lambda_{\emph{\textbf{k}}((\Gamma))}(\epsilon(rg^{-1})) = \lambda_{K}(rg^{-1})    $, and since $g >0  $, $ r>0 $, and  $ \lambda_{K}:V \lonto \textbf{\emph{k}}$ is order-preserving, $a_{\gamma_0} = \lambda_{K}(rg^{-1})>0 $ follows, as required.
\end{proof}

\begin{theoremApp}\label{T.APP.rcvf_hahn_embedd}
	Let $ K $ be a real closed valued field with value group $ \Gamma $ and residue field $ \textbf{k} $. Suppose that $ G \subseteq K^{>0} $ is a monomial group of $ K $ and $ \textbf{k}_0 \subseteq V $ is a coefficient field of $ K $. There exists an embedding of valued fields $\epsilon:  K \linto \textbf{k}((\Gamma)) $ such that  $ \epsilon(g) = x^{v(g)}$ for all $ g \in G $ and  $  \epsilon(a) = \lambda(a)$ for all $ a \in \textbf{k}_0 $. 
\end{theoremApp}

\begin{proof}
	See \cite[Satz 21, p. 62]{priess-crampe.ord_str}.
\end{proof}

\begin{theoremApp}\label{T.APP.rcvf_hahn_embedd_II}
	Let $ K \subseteq L $ be an extension of real closed valued fields with value groups $ \Gamma $ and $ \Delta $, and residue fields $ \textbf{k} $ and $ \textbf{l} $, respectively.  There exist embeddings of valued fields $ \epsilon_{K} : K \linto \textbf{k}((\Gamma)) $ and $ \epsilon_{L} : L \linto \textbf{l}((\Delta)) $ such that $ \epsilon_{L \upharpoonright K} = \epsilon_K $.
\end{theoremApp}

\begin{proof}
	Let $ G \subseteq K^{>0} $ be a monomial group of $ K $, $ H\subseteq L^{>0} $ be a monomial group of $ L $ containing $ G $, $ \emph{\textbf{k}}_0\subseteq V_K $ be a coefficient field of  $ K$, and $ \emph{\textbf{l}}_0 \subseteq V_L $ be a coefficient field of $ L $ containing  $ \emph{\textbf{k}}_0 $; these exist by items (a) and (b) in Lemmas \ref{L.APP.mon_grp_rcvf} and \ref{L.APP.coeff_field_rcvf}.  

\noindent\textit{Claim.}  $ K^{>0} \cap H = G $ and $ K \cap \emph{\textbf{l}}_0 = \emph{\textbf{k}}_0 $.

\noindent \textit{Proof of Claim.} Clearly $ G \subseteq K^{>0} \cap H $	and $ \emph{\textbf{k}}_0 \subseteq K \cap \emph{\textbf{l}}_0 $. Since   $ K \subseteq L $ as valued fields,  $ \text{ker}(v_{\upharpoonright K^{>0}})  \subseteq \text{ker}(v_{\upharpoonright L^{>0}})   $, and since $ H $ is a monomial group of $ L $, $ H \cap \text{ker}(v_{\upharpoonright L^{>0}}) = (1) $; therefore $ H \cap \text{ker}(v_{\upharpoonright K^{>0}}) = (1) $, hence $ (K^{>0} \cap H) \cap \text{ker}(v_{\upharpoonright K^{>0}}) = (1)$, and thus $ K^{>0} \cap H = G $ by the implication (i) $ \Rightarrow  $  (ii) in Lemma \ref{L.APP.mon_grp_rcvf}.  Similarly, since $ K \subseteq L $ as valued fields, $ K \cap V_L  = V_K$, hence $ \emph{\textbf{k}}_0 \subseteq K \cap \emph{\textbf{l}}_0  \subseteq K \cap V_L = V_K$, and thus  $ \emph{\textbf{k}}_0= K \cap \emph{\textbf{l}}_0 $ by the implication (i) $ \Rightarrow $  (ii) in Lemma \ref{L.APP.coeff_field_rcvf}. \hfill $ \square_{\text{Claim}} $

	By Theorem \ref{T.APP.rcvf_hahn_embedd}, there exists an embedding of valued fields $ \epsilon_K: K \linto \emph{\textbf{k}}((\Gamma)) $ such that  $ \epsilon(g) = x^{v(g)}$ for all $ g \in G $ and  $  \epsilon(a) = \lambda(a)$ for all $ a \in \emph{\textbf{k}}_0 $; the goal is to extend $ \epsilon_K$ to an embedding of valued fields $  \epsilon_L: L \linto \emph{\textbf{l}}((\Delta))  $. If the extension $ K \subseteq L   $ is immediate, then $ G = H $,  $ \textbf{\emph{k}}_0 = \emph{\textbf{l}}_0 $, and  $ \emph{\textbf{k}}((\Gamma)) = \emph{\textbf{l}}((\Delta)) $, therefore the existence of an embedding of valued fields $ \epsilon_L: L \linto \emph{\textbf{l}}((\Delta))$  such that $ \epsilon_{L \upharpoonright K} = \epsilon_K$ follows by \cite[Theorem 5]{kaplansky}. Suppose now that the extension $ K \subseteq L $ is not immediate; by induction it suffices to consider the case  $ L := K \langle r \rangle  $, where $ r \in L $ is transcendental over  $ K $ and $ K \langle r \rangle := \rho(K(r))  $. By the Wilkie inequality (\cite[Corollary 5.6]{vdD.T-convex_II}), there are two cases to consider:

	\underline{Case 1.} $ \emph{\textbf{k}} = \emph{\textbf{l}} $ (hence $ \emph{\textbf{k}}_0 = \emph{\textbf{l}}_0 $) and there exists $ \delta \in \Delta \setminus \Gamma $ such that  $ \Delta = \Gamma \oplus \Q\delta $.  Let $ h \in H $ be such that  $ v(h) = \delta $,  so that  $ h \in H \setminus G $ and $ H = G \cdot h^{\Q} $.  Since $ K^{>0} \cap H = G $ and $ h \notin G $, it follows that  $ h \in L \setminus K $, and since   $ K $ is a real closed field, $ h  $ is transcendental over $ K $; similarly,  $ x^{\delta}\in \emph{\textbf{l}}((\Delta)) $ is transcendental over $ K':= \epsilon_K(K) \subseteq \emph{\textbf{k}}((\Gamma)) $, and thus there exists a unique field isomorphism $\widetilde{\epsilon}_K: K(h) \lra K'(x^{\delta}) $ extending $ \epsilon_K  $ and sending $ h $ to $ x^{\delta}$. Note that $  K'(x^{\delta}) \subseteq \emph{\textbf{k}}((\Gamma))(x^{\delta})  \subseteq \emph{\textbf{k}} ((\Gamma\oplus\Z \delta)) $, therefore $ \widetilde{\epsilon}_K $ is the unique field embedding $ K(h) \linto \emph{\textbf{k}}((\Gamma\oplus\Z\delta)) $ extending $ \epsilon_K $ with $ \widetilde{\epsilon}_K(h) = x^{\delta} $. 

	Since $ \Gamma $ is divisible and  $ \Delta $ is torsion-free,  $ n\delta \notin \Gamma $ for all $ 0 \neq n \in \Z $, and thus given $ a:=\sum_{i=0}^na_ih^i \in K[h] $ with $ a_i \neq 0$ for all $ i \in [n]  $, it follows that $ v(a_ih^i) \neq v(a_jh^j) $ for all $ i, j \in [n] $ with $ i \neq j  $, therefore
\[
v(a)= v\left( \sum_{i=0}^na_ih^i \right) = \underset{0 \leq i \leq n}{\text{min}}\{v(a_ih^i)\}  =  \underset{0 \leq i \leq n}{\text{min}}\{\nu(\epsilon(a_i))+ i\delta\} =\nu\left( \sum_{i=0}^n\epsilon(a_i)(x^{\delta})^i \right) = \nu(\widetilde{\epsilon}_K(a)),
\] 
and thus $ \widetilde{\epsilon}_K:  K(h) \linto \emph{\textbf{k}}((\Gamma \oplus \Z \delta))  $ is an embedding of valued fields; moreover, the value group of  $ K(h) $ is $ \Gamma \oplus \Z \delta $ and its residue field is $ \emph{\textbf{k}} $ (\cite[Corollary 2.2.3]{engler/prestel.valued_fields}), therefore $ G\cdot h^{\Z} \subseteq K(h)^{>0}$ is a monomial group of $ K(h) $ such that $ \widetilde{\epsilon}_K(h') = x^{v(h')} $ for all $ h' \in G \cdot h^{\Z}$. By Lemma \ref{lem.ord_pres_mon_group}, $\widetilde{\epsilon}_K  $ is an embedding of ordered valued fields, and since $ h \in L=K \langle r \rangle \setminus K  $, $ K \langle h \rangle = K \langle r \rangle = L $ by the exchange property (\cite[Theorem 4.1]{pillay/steinhorn}), and thus it follows that $ \widetilde{\epsilon}_K $ can be extended to an embedding of valued fields $ \epsilon_L: L \linto \emph{\textbf{l}}((\Delta)) $.

\underline{Case 2.} $ \Gamma = \Delta $ (hence $ G = H $) and  there exists $ t \in \emph{\textbf{l}} \setminus \emph{\textbf{k}} $ such that $ \emph{\textbf{l}} = \emph{\textbf{k}}\langle t \rangle  $.  Let $ b \in \emph{\textbf{l}}_0 $ be such that  $ \lambda(b) = t$, so that $ b \in \emph{\textbf{l}}_0 \setminus \emph{\textbf{k}}_0 $ and $ \emph{\textbf{l}}_0 = \emph{\textbf{k}}_0 \langle b \rangle  $. Since $   K \cap \emph{\textbf{l}}_0= \emph{\textbf{k}}_0  $ and $ b \notin \textbf{\emph{k}}_0 $, it follows that $ b \in L \setminus K $, and since both  $ K $ and $ \emph{\textbf{k}}_0 $ are real closed fields, $ b  $ is transcendental both over $ K $ and over $ \emph{\textbf{k}}_0 $; similarly,  $ t \in \textbf{\emph{l}} \subseteq \textbf{\emph{l}}((\Delta)) $ is transcendental both over $ K':= \epsilon_K(K) \subseteq \emph{\textbf{k}}((\Gamma))$ and over $ \emph{\textbf{k}} $, and thus there exists  a unique field isomorphism $\overline{\epsilon}_K: K(b) \lra K'(t) $ extending $ \epsilon_K  $ and sending $ b $ to $t $. Note that $  K'(t) \subseteq \emph{\textbf{k}}((\Gamma))(t)  \subseteq \emph{\textbf{k}}(t) ((\Gamma)) $, therefore $ \overline{\epsilon}_K $ is the unique  field embedding $ K(b) \linto \emph{\textbf{k}}(t) ((\Gamma)) $ extending $ \epsilon_K$ with $ \overline{\epsilon}_K(b) = t$.

Since  $ v(b) = \nu(\lambda(b)) = \nu(t)  = 0 $,  it follows that both $ K \subseteq K(b) $ and $ K' \subseteq K'(t) $ are Gauss extensions (\cite[Corollary 2.2.2]{engler/prestel.valued_fields}). In particular, given $ a:=\sum_{i=0}^na_ib^i \in K[b] $ with $ a_i \neq 0 $ for all  $  i \in [n] $, 
\begin{align*}
	v(a)= v\left( \sum_{i=0}^na_ib^i \right) &= \underset{0 \leq i \leq n}{\text{min}}\{v(a_i)\} =  \underset{0 \leq i \leq n}{\text{min}}\{\nu(\epsilon(a_i))\} =\nu\left( \sum_{i=0}^n\epsilon(a_i)t^i \right) = \nu(\overline{\epsilon}_K(a)),
\end{align*}
and thus $ \overline{\epsilon}_K:  K(b) \linto \emph{\textbf{k}}(t) ((\Gamma))  $  is an embedding of valued fields; moreover, since $ K \subseteq K(b) $ is the Gauss extension, $ K(b) $ has value group $ \Gamma $ and residue field $\emph{\textbf{k}}(t)$, therefore $ G \subseteq K^{>0}\subseteq K(b)^{>0} $ is a monomial group of $ K(b) $ such that $ \overline{\epsilon}_K(g) = x^{v(g)} $ for all $ g \in G $. By Lemma \ref{lem.ord_pres_mon_group}, $ \overline{\epsilon}_K $ is an embedding of ordered valued fields, and arguing as in Case 1 it follows that $ \overline{\epsilon}_K$ can be extended to an embedding of valued fields $ \epsilon_L: L \linto \emph{\textbf{l}}((\Delta)) $.
\end{proof}